\documentclass[a4paper, 12pt]{amsart}
\usepackage{amssymb, amsmath, amsthm, color, MnSymbol}
\usepackage[colorlinks=true, breaklinks=true, linkcolor=black, citecolor=blue, urlcolor=red]{hyperref}
\usepackage[francais,english]{babel}
\usepackage{graphicx}

\usepackage{tikz-cd}%For commutative diagrams
\usepackage{tikz}
\usetikzlibrary{arrows.meta, shapes, positioning}
\usetikzlibrary{decorations.markings}

\pgfdeclarelayer{background}
\pgfdeclarelayer{nodelayer}
\pgfdeclarelayer{edgelayer}
\pgfsetlayers{background,edgelayer,nodelayer,main}
\tikzset{
  none/.style={}
}
\tikzset{every loop/.style={min distance=12mm, looseness=4}} % Adjust distance and looseness globally

\newtheorem{theorem}{Theorem}[section]

\newtheorem{lemma}[theorem]{Lemma}
\newtheorem{corollary}[theorem]{Corollary}
\newtheorem{proposition}[theorem]{Proposition}

\newtheorem{definition}[theorem]{Definition}

\newtheorem{remark}[theorem]{Remark}
\newtheorem{example}[theorem]{Example}

\newtheorem{theoremIntro}{Theorem}

\newcommand{\act}{\curvearrowright}
\newcommand{\onto}{\twoheadrightarrow}

\newcommand{\cA}{\mathcal A}
\newcommand{\cAL}{\mathcal A_\Lambda}

\newcommand{\Ball}{B_{\mathbf C}}
\newcommand{\C}{\mathbf C}

\newcommand{\CKA}{Cuntz--Krieger algebra}
\DeclareMathOperator{\diag}{diag}

\DeclareMathOperator{\End}{End}

\DeclareMathOperator{\fd}{fd}

\DeclareMathOperator{\full}{full}

\DeclareMathOperator{\id}{id}
\DeclareMathOperator{\Irr}{Irr}
\newcommand{\fK}{\mathfrak K}

\DeclareMathOperator{\Mod}{Mod}
\DeclareMathOperator{\mr}{MR}
\newcommand{\N}{\mathbb N}

\newcommand{\cO}{\mathcal O}
\DeclareMathOperator{\ob}{ob}
\newcommand{\cOL}{\mathcal O_\Lambda}

\newcommand{\cP}{\mathcal P}
\DeclareMathOperator{\Pfd}{Pfd}

\newcommand{\cPL}{\mathcal P_\Lambda}
\DeclareMathOperator{\Pdim}{Pdim}
\DeclareMathOperator{\PU}{PU}

\DeclareMathOperator{\Red}{Red}
\DeclareMathOperator{\Rep}{Rep}
\DeclareMathOperator{\Hilb}{Hilb}

\DeclareMathOperator{\Spec}{Spec}

\newcommand{\ti}{\tilde}
\newcommand{\cT}{\mathcal T}
\newcommand{\T}{\mathbb T}

\newcommand{\varep}{\varepsilon}
\newcommand{\wh}{\widehat}
\newcommand{\Z}{\mathbb Z}

\title{Classification of representations of higher-rank graph C$^*$-algebras}

\author{Arnaud Brothier, Aidan Sims and Dilshan Wijesena}
	\address{Arnaud Brothier\\ University of Trieste, Via Weiss, 2 34128 Trieste, Italia and
	School of Mathematics and Statistics, University of New South Wales, Sydney NSW 2052, Australia}
	\email{arnaud.brothier@gmail.com\endgraf
		\url{https://sites.google.com/site/arnaudbrothier/}}
\address{Aidan Sims\\ School of Mathematics and Statistics, University of New South Wales, Sydney NSW 2052, Australia}
\email{aidan.sims@unsw.edu.au}
	\email{dilshan.wijesena@hotmail.com}

\begin{document}

\begin{abstract}
We develop new techniques for the construction and classification of representations of row-finite and locally convex higher-rank graph C$^*$-algebras $\mathcal O$.
This class includes Cuntz--Krieger algebras associated to row-finite directed graphs.
Our approach relies on the representation theory of a certain non-self-adjoint algebra and a lifting process of representations.
We introduce a novel dimension vector for representations of $\mathcal{O}$ yielding a countable partition of the spectrum.
Given a Cuntz--Krieger algebra and a finite dimension vector, we construct a smooth manifold parametrising the corresponding spectral component.
Our techniques are both explicit and functorial.
\end{abstract}

\maketitle

\section{Introduction}
Objects such as groups or algebras are often best understood through the study of their actions.
In the case of C$^*$-algebras, this viewpoint typically leads to the investigation of their representations on Hilbert spaces, which are in general difficult to construct explicitly.
Jones observed that, under suitable conditions, actions of relatively basic objects can be lifted to actions of more complex ones \cite{Jones17,Jones18}.
This idea was successfully applied to construct new unitary representations of the Richard Thompson groups $F$, $T$, and $V$, as well as related groups, and to establish analytic and geometric properties of these groups \cite{Brothier-Jones19a,Brothier23}.
Although these techniques were originally developed in a group-theoretic setting, they were later adapted to the construction and classification of explicit representations of C$^*$-algebras \cite{Brothier-Jones19b}.
The first C$^*$-algebra to benefit from this approach was the Cuntz(--Dixmier) algebra $\mathcal{O}_2$ \cite{Cuntz77,Dixmier64}.
The first and third authors subsequently developed powerful techniques which,
among other results, unravel geometric structures on a large piece of the
spectrum of $\cO_2$
\cite{Brothier-Wijesena25,Brothier-Wijesena26,Brothier-Wijesena24}. This was a
complete surprise---especially since $\cO_2$ is a non-type I simple
C$^*$-algebra and thus, by general theorems, has, a priori, a highly
pathological spectrum \cite{Glimm61}.

The Cuntz(--Dixmier) algebra falls into the very general class of graph
C$^*$-algebras (sometimes called Cuntz--Krieger algebras of directed graphs)
\cite{Cuntz-Krieger80,Kumjian-Pask-Raeburn-Renault97,Kumjian-Pask-Raeburn98},
see also \cite{Raeburn05}. Specifically, the Cuntz algebra $\cO_n$ is the
C$^*$-algebra of the directed graph with one vertex and $n$ loops. Graph
C$^*$-algebras form a rich class of C$^*$-algebras. They include, up to
stable isomorphism, all simple purely infinite nuclear C$^*$-algebras with
torsion-free $K_1$-groups, and all AF algebras (see for instance
\cite{Raeburn05}). A further generalisation, developed by Kumjian and Pask
based on earlier work of Robertson and Steger on groups acting on buildings \cite{RobSte},
replaces directed graphs with a larger class of combinatorial objects called
\emph{higher-rank graphs} \cite{Kumjian-Pask00} (see also \cite{Raeburn-Sims-Yeend03} for a further extension and also \cite{Lawson-Vdovina20, Lawson-Vdovina22}). 
Their aim was to realise larger classes of
C$^*$-algebras including purely infinite C$^*$-algebras with torsion in $K_1$,
and non-AF stably finite C$^*$-algebras. These $k$-graph C$^*$-algebras include interesting
examples beyond the reach of graph C$^*$-algebras, such as groupoid C$^*$-algebras constructed
from Brin's celebrated higher dimensional Thompson groups \cite{Brin04} (see Section \ref{sec:example}).
Graph C$^*$-algebras and their higher-rank generalisation have been intensively studied.
Their ideal structure is well-understood and thus there exist explicit formulae to
describe their primitive spectra \cite{Huef-Raeburn97, Hong04, Bates-Hong-Raeburn-Szymanski02, ChristensenNeshveyev24, BrixCarlsenSims25}.
However, these C$^*$-algebras are almost never of type~I, and thus their representation theory, and hence their spectra, remain deeply mysterious.

\textbf{Purpose of this paper.}
We establish powerful new tools for constructing and studying representations of (higher-rank) graph C$^*$-algebras, and apply them to key examples.
We also develop a deeper understanding of the techniques, which were first used on the Cuntz algebra in a rather indirect manner,
revealing a critical use of a non-self-adjoint algebra.

\textbf{Strategy and main results.}
A higher-rank graph $\Lambda$ or $k$-graph is roughly speaking the path space of a countable directed graph whose edges are coloured by $\{1,\cdots,k\}$ that we mod out by relations of the form $\mu\circ \nu = \nu'\circ \mu'$ for bi-coloured paths
$(\mu,\mu')$ and $(\nu,\nu')$ of length two with colours in opposite orders (see Section \ref{sec:higher-rank graphs} for details).
The $1$-graphs correspond to directed graphs (also called quivers).
A $k$-graph $\Lambda$ defines a universal C$^*$-algebra $\cOL$ generated by isometries $x_\lambda$ indexed by vertices and edges of $\Lambda$.
A representation $x_\lambda \mapsto X_\lambda$ of $\cOL$ on $L$ decomposes over the vertices $\oplus_v L_v$ and each edge $\lambda :v\to w$ defines a partial isometry $X_\lambda$ with domain $L_v$ and codomain contained in $L_w$.
The generators are subject to four axioms in general but if $k=1$ we may remove one of them and recover the Cuntz--Krieger algebras of directed graphs.

In this article we define a new universal C$^*$-algebra $\cPL$, termed \emph{Pythagorean}, that is obtained by removing one type of relations (the one asking that the $x_\lambda$ are partial isometries).
Note that for convenience, we denote by $a_\lambda$ the generators of $\cPL$  so that the quotient map $\cPL\onto\cOL$ is given by the formula $a_\lambda\mapsto s_\lambda^*$ (hence taking adjoints).
Each representation $H$ of $\cPL$ still decomposes into a direct sum of $H_v$ over the vertices $v$ of $\Lambda$.
For any $\Lambda$ the C$^*$-algebra $\cPL$ admits many finite dimensional representations while generically $\cOL$ has none.
By adapting a construction pioneered by Jones we define a directed system of Hilbert spaces constructed from $H$ which yields a representation $\Pi(H)$ of $\cOL$.
This process is functorial $\Pi:\Rep(\cPL)\to\Rep(\cOL)$ and is moreover essentially surjective: any representation $L$ of $\cOL$ is in fact of the form $\Pi(H)$, up to isomorphism.
The construction is explicit: we may interpret vectors of $\Pi(H)$ as classes of trees whose leaves are decorated by vectors of $H$. The generators of C$^*$-algebra $\cOL$ act by growing and reducing trees.

We now have a new and practical way to construct \emph{all} representations of $\cOL$. We want to go further and perform some classification.
We could naively take the dimension vector of $H$ and hope to obtain an invariant for $L\simeq\Pi(H)$.
This does not work because multiple $H$ can give the same representation $L$ of $\cOL$.
But if there exists at least one \emph{finite dimensional} $H$ satisfying $\Pi(H)\simeq L$, then there is a \emph{smallest} such $H$, which we denote by $\Sigma(L)$.
The dimension vector of $\Sigma(L)$ is then an invariant of $L$ (which depends on the choice of the $k$-graph $\Lambda$). Better still, $L\mapsto\Sigma(L)$ is functorial and so $\Sigma(L)$ itself is an invariant of $L$.
Finally, a copy of $\Sigma(L)$ can be found inside $L$: it is the smallest subspace of $L$ closed under the $x_\lambda^*$ that generates $L$ as a representation of $\cOL$.

Now, we could naively try to classify representations of $\cPL$ and hope to obtain a classification of representations of $\cOL$.
The situation is more delicate than this.
Indeed, there are intertwiners $\Pi(H)\to\Pi(K)$ of representations of $\cOL$ that do not come from intertwiners $H\to K$ of representations of $\cPL$:
an irreducible representation $H$ of $\cPL$ may yield a reducible representation $\Pi(H)$ of $\cOL$, see Example \ref{ex:reducible}.
Rather than considering the representation category $\Rep(\cPL)$ we consider the category of modules $\Mod(\cAL)$ of the non-self-adjoint subalgebra $\cAL$ of $\cPL$ generated by the $a_\lambda$ (we use the term \emph{modules} for continuous linear actions
of non-self-adjoint algebras, and the term \emph{representations} for C$^*$-algebras).
The categories $\Mod(\cAL)$ and $\Rep(\cPL)$ have same objects but the former has more arrows, because its arrows need not intertwine the adjoints of the generators.
Our functor $\Pi:\Rep(\cPL)\to \Rep(\cOL)$ extends to $\Mod(\cAL)$: this means that $\Pi$ transports arrows of $\Mod(\cAL)$ as well.

The category $\Mod(\cAL)$ behaves differently than the representation category of a C$^*$-algebra: the orthogonal complement of a submodule may not be a submodule.
However, our analysis demonstrates in Proposition~\ref{prop:decomposition} that if $H$ is a finite dimensional $\cAL$-module,
then it has a unique decomposition
\[
H = H_1\oplus\cdots\oplus H_n \oplus Z
\]
such that each $H_i$ is an irreducible submodule, and $Z$ is a Hilbert subspace (the largest containing no subrepresentation).
Then $\Pi(H)$ is isomorphic to the direct sum of the $\Pi(H_i)$ and each $\Pi(H_i)$ is irreducible (as a representation of $\cOL$).
Moreover, $\Pi(H_i)\simeq \Pi(H_j)$ as representations of $\cOL$ if and only if $H_i\simeq H_j$ as $\cAL$-modules.
Hence, when restricted to finite dimensional $\cAL$-modules $H$ for which $Z$ is trivial (we call such modules $H$ \emph{full} modules),
$\Pi$ becomes a bijection on hom-spaces between $\Mod(\cAL)$ and $\Rep(\cOL)$.
This allows us to efficiently classify representations of $\cOL$.

The upshot is that, in practice, $\cAL$-modules are significantly easier to construct and to test for irreducibility than representations of $\cOL$.
We derive a novel dimension vector and dimension, termed \emph{Pythagorean} dimension, for a representation $L$ of $\cOL$---namely
$$v\in \Lambda^0\mapsto\dim(a_v\cdot H) \text{ and } \dim(H)$$
under the assumption that $H$ is a full $\cAL$-module satisfying $\Pi(H)\simeq L$.
These are new invariants for classifying representations of $\cOL$ that are built using the choice of the $k$-graph $\Lambda$.
We have obtained, among others, the following result.

\begin{theoremIntro}(Corollary \ref{cor:categories})
Let $\Lambda$ be a countable row-finite locally convex $k$-graph with $k\geq 1$ with associated higher-rank graph C$^*$-algebra $\cOL$ and non-self-adjoint algebra $\cAL$.
Let $\Mod_{\full,\fd}(\cAL)$ be the category of finite dimensional and full $\cAL$-modules and let $\Rep_{\Pfd}(\cOL)$ be the category of representations of $\cOL$ with finite Pythagorean dimension.
Then our constructions $H\mapsto \Pi(H)$ and $L\mapsto \Sigma(L)$ define an equivalence of categories between $\Mod_{\full,\fd}(\cAL)$ and $\Rep_{\Pfd}(\cOL)$.
Additionally, they preserve irreducibility and isomorphism classes.
\end{theoremIntro}

Our results rely heavily on a technical lemma (Lemma~\ref{lem:key}) whose proof uses a compactness argument. This limits our results to representations $L$ of $\cOL$ with \emph{finite} Pythagorean dimension.

We now specialise to ordinary graphs $\Lambda$ and aim to provide \emph{moduli spaces} for representations of $\cOL$: geometrical objects indexing irreducible classes of representations of $\cOL$.
Given a finite dimension vector $d$ (i.e.~a finitely supported map from the vertex set to the natural numbers) consider the matrix $\cAL$-modules of dimension $d$.
An easy argument shows that they form a compact real smooth manifold $\mr(d)$. The irreducible ones $\Irr(d)$ form an open subset of $\mr(d)$ and thus $\Irr(d)$ is either empty or forms a submanifold of same dimension.
The space of unitary equivalence classes of elements of $\Irr(d)$ has a unique smooth-manifold structure.
Our functor $\Pi$ realises a bijection from the space $\Spec(d)$ of $d$-dimensional irreducible representations of $\cOL$ modulo unitary conjugacy to
the smooth manifold of unitary equivalence classes of elements of $\Irr(d)$; we call this the \emph{Pythagorean dimension $d$ spectrum} of $\cOL$.

\begin{theoremIntro}\label{theo:B}
Let $\cOL$ be the Cuntz--Krieger algebra of a countable row-finite directed graph $\Lambda$.
The spectrum $\Spec$ of $\cOL$ partitions as $\bigcup_{d \in \N[\Lambda^0] \sqcup \{\infty\}} \Spec(d)$.
For $d \in \N[\Lambda^0]$, the functor $\Pi$ defines a bijection between $\Spec(d)$ and the space $\Irr(d)/{\sim_u}$
of $d$-dimensional irreducible matrix modules for the non-self-adjoint algebra $\cAL$ modulo unitary conjugacy.
For $v \in \Lambda^0$, let $D(v):=\sum_{\lambda\in v\Lambda^1} d{s(\lambda)}$
(the sum over the entries of $d$ at the sources of edges $\lambda$ with range $v$).
If $\Irr(d)/{\sim_u} \not= \emptyset$, then it is a real, smooth manifold of dimension
$$1 + 2  \sum_{v\in \Lambda^0} dv(D(v) - dv).$$
\end{theoremIntro}

Recall again that for most graphs $\Lambda$ the Cuntz--Krieger algebra is non-type $I$ and thus has a large and pathological spectrum by Glimm \cite{Glimm61}.
So it is a pleasant surprise to find harmony in $\Spec(d)$ when $d$ is finite.
However, $\Spec(d)$ remains mysterious to us when $\sum_{v\in\Lambda^0} dv=\infty$.

We illustrate these results on various graph C$^*$-algebras in Section \ref{sec:example}.
In particular, we consider two graphs $\Lambda,\Lambda_2$ whose associated \CKA s are both isomorphic to the Cuntz algebra $\cO_2$ in two generators.
We obtain two families of moduli spaces (one family per graph) and derive a mapping between these two families.
Finally, we explain how to easily construct a large class of representations of Brin's higher dimensional Thompson groups using a $2$-graph.

\section{higher-rank graphs and higher-rank graph \texorpdfstring{C$^*$}{C*}-algebras}

This is a preliminary section in which we define higher-rank graphs and their C$^*$-algebras.
We recommend \cite{Kumjian-Pask00,Sims10} for additional details on $k$-graphs and to the book of Raeburn for the graph case \cite{Raeburn05}.

\subsection{Higher-rank graphs}\label{sec:higher-rank graphs}

%k-graph
Fix a natural number $k\geq 1$ and consider the free Abelian monoid in $k$-generators $\N^k$ (hence we take the convention that $\N$ contains $0$) with standard basis $(e_1,\cdots,e_k)$.
A $k$-graph is a pair $(\Lambda,d)$ where $\Lambda$ is a countable category, and $d:\Lambda\to\N^k$ is a functor called the \emph{degree map} satisfying the \emph{factorisation property}:
for all $\lambda\in \Lambda$ and $m,n\in\N^k$ if $d(\lambda)=m+n$, then there exists a unique pair $(\mu,\nu)$ in $\Lambda$ satisfying $\lambda=\mu\circ \nu$ and $d(\mu)=m,d(\nu)=n$.

The category $\Lambda$ with objects, arrows and composition of arrows $\circ$ will often be identified with, and treated as,
a set whose elements are the arrows, equipped with a partially defined binary operation (the composition $\circ$).
Moreover, we identify objects $v$ of $\Lambda$ with their identity arrow $\id_v$ forming the subset $\Lambda^0$.
Finally, we will write $r,s$ for the range and source maps; so for $\lambda \in \Lambda$, we have
$r(\lambda), s(\lambda) \in \Lambda^0$, and $r(\lambda)\lambda = \lambda = \lambda s(\lambda)$.

Elements of $\Lambda$ will be called paths and elements of $\Lambda^0$ vertices. The factorisation property
implies that $\Lambda$ is both left- and right-cancellative, and then also that $\Lambda^0 = d^{-1}(0)$. By
extension of this notation, we write $\Lambda^n := d^{-1}(n)$. For $1 \le i \le k$, we call elements of
$\Lambda^{e_i}$ \emph{edges of colour $i$.} We compose from right to left so that
$$r(\lambda\circ \mu)=r(\lambda),\ s(\lambda\circ\mu)=s(\mu)
\text{ and } r(\lambda)\circ \lambda\circ s(\lambda) =
\lambda.$$
We may suppress the symbols $\circ$ and parentheses. Given $\lambda \in \Lambda$ and $S\subset \Lambda$, we write
$\lambda S = \{\lambda \mu : \mu \in S\text{ and }r(\mu) = s(\lambda)\}$, $S\lambda$ is defined analogously. So for $v \in \Lambda^0$,
\[
    v\Lambda^m = \{\lambda \in \Lambda : d(\lambda) = m\text{ and }r(\lambda) = v\}.
\]
In particular, $v\Lambda^{e_i}$ is the set of all edges of colour $i$ with range $v$.
We write $\Lambda^1$ for the collection of all edges; that is, $\Lambda^1:=\cup_{1\leq i\leq k} \Lambda^{e_i}.$

We give $\Lambda$ the left algebraic order: $\lambda\leq \mu$ if $\mu \in \lambda\Lambda$. 

\begin{example}
If $k=1$, then a $k$-graph is the path-space of a countable directed graph. The degree map is then valued in $\N$ and is the usual length.
\end{example}

%Row-finite and locally convex
We will need the following two properties for our $k$-graphs.
\begin{definition}
A $k$-graph $(\Lambda,d)$ is
\begin{itemize}
\item \emph{row-finite} if for all $v\in\Lambda^0$ there are finitely many edges with range $v$; that is, $|v\Lambda^1|<\infty$;
\item \emph{locally convex} if, given $1 \le i < j \le k$, $v\in\Lambda^0$ and $\mu \in v\Lambda^{e_i}$ and $\nu \in v\Lambda^{e_j}$, both $s(\mu)\Lambda^{e_j}$ and $s(\nu)\Lambda^{e_i}$ are nonempty.
\end{itemize}
\end{definition}

\begin{center}\textbf{From now on $\Lambda$ denotes a row-finite and locally convex $k$-graph equipped with a degree map $d$.}
\end{center}

\begin{remark}
\begin{enumerate}
	\item \emph{Row-finiteness} ensures that the sums appearing in the presentations of the key C$^*$-algebras $\cPL$ and $\cOL$ of this paper (see Definitions \ref{def:P} and \ref{def:O}) are finite.
	\emph{Local convexity} ensures that the universal presentation of $\cOL$ is \emph{non-degenerate} (see \cite[Theorem 3.15]{Raeburn-Sims-Yeend03}).
	Local convexity also guarantees existence of \emph{common right-multiples} for trees associated to $\Lambda$ (see Lemma \ref{lem:complete-tree}).
	This property allows us to construct the direct limits of Hilbert spaces that we use to describe the representations of $\cOL$.
	\item By \cite[Theorem~2.1]{Farthing08}, given any row-finite locally convex $k$-graph $\Lambda$, there is a row-finite $k$-graph $\overline{\Lambda}$ with no sources such that $C^*(\Lambda)$ embeds as a full corner of $C^*(\overline{\Lambda})$. In particular,
	$C^*(\Lambda)$ and $C^*(\overline{\Lambda})$ have the same representation theory. So we could in principle restrict attention to $k$-graphs with no sources and still capture the representation theory of arbitrary row-finite locally convex
	$k$-graphs. In practice the passage from $\Lambda$ to $\overline{\Lambda}$ is complex, and it can be difficult to use it to translate results, so we have chosen to deal directly with $k$-graphs that are locally convex but admit sources.
\end{enumerate}
\end{remark}

\subsection{Higher-rank graph \texorpdfstring{C$^*$}{C*}-algebras and their representations}\label{sec:graph-algebras}

%C$^*$-algebra
\begin{definition}\label{def:O}
Let $\Lambda$ be a row-finite $k$-graph with no sources. The graph C$^*$-algebra $\cOL$ of $\Lambda$ is the universal C$^*$-algebra generated by $\{x_\mu: \mu\in\Lambda\}$ subject to the relations
\begin{enumerate}
\item $\{x_v:\ v\in\Lambda^0\}$ are mutually orthogonal projections;
\item $x^*_\lambda x_\lambda=x_{s(\lambda)}$ (partial isometry condition);
\item $x_v=\sum_{\lambda \in v\Lambda^{e_i}} x_\lambda x_\lambda^*$ for each $v\in \Lambda^0$ and $1\leq i\leq k$ with $v\Lambda^{e_i} \not= \emptyset$;
\item $x_\lambda x_\mu = x_{\lambda\circ\mu}$ when composable (we keep the same order);
\end{enumerate}
\end{definition}
When $k=1$, then we recover the row-finite graph C$^*$-algebras (and thus Cuntz--Krieger algebras) and when $k=1$ and there is only one vertex we recover the Cuntz algebras $\cO_n$.

%Reps
By ``universal C$^*$-algebra,'' we mean that given a Hilbert space $H$ and bounded linear operators $X_\lambda\in B(H)$ indexed by $\Lambda$ satisfying conditions (1)--(4), the map $x_\lambda\mapsto X_\lambda$ uniquely extends to a $^*$-homomorphism $\pi_X
: \cOL\to B(H)$. Conversely, given any representation $\pi$ of $\cOL$, defining $X_\lambda := \pi(x_\lambda)$ yields a family of operators $X_\lambda$ such that $\pi = \pi_X$. In view of this, we sometimes refer to a map $X : \lambda \mapsto X_\lambda$
from $\Lambda$ to $B(H)$ such that the $X_\lambda$ satisfy (1)--(4) as a \emph{representation of $\Lambda$}.

In the sequel we will only consider representations of $\Lambda$ that are nondegenerate in the following sense: the net $\big(\sum_{v \in F} X_v\big)_{F \Subset \Lambda^0}$ (where $\Subset$ denotes a finite subset) converges in the strong operator topology
to $1_H$.
Such representations form a category $\Rep(\cOL)$: an arrow from $(\pi^H,H)$ to $(\pi^K,K)$ is a bounded linear map $\phi:H\to K$ intertwining the actions as follows
$$\phi\circ \pi^H(a)=\pi^K(a)\circ \phi \text{ for all } a\in\cOL.$$

Unlike, for example, the theory of groups, not every system of generators and relations defines a universal C$^*$-algebra: one must check that for any finite formal linear combination of words in the generators and their adjoints, the set of norms of its
images in representations of the given relations is bounded. For the above relations, (1) ensures that each $\|x_v\| \in \{0,1\}$ and then~(2) and the $C^*$-identity ensure that each $\|x_\lambda\| \in \{0,1\}$, so submultiplicativity and subadditivity of
the norm on $B(H)$ guarantees the desired uniform bound (see, for example \cite[Section II.8.3]{Blackadar06}). To see that all the generators of the universal C$^*$-algebra are nonzero, it suffices to construct a representation in which they are nonzero;
such a representation is constructed in the proof of \cite[Theorem~3.15]{Raeburn-Sims-Yeend03}.

Relation (4) of Definition \ref{def:O} implies that $\cOL$ is generated by the $x_\lambda$ with $\lambda$ running over vertices and edges (rather than all paths).
We deduce the following presentation.

\begin{proposition}\label{prp:edge presentation}
The graph C$^*$-algebra $\cOL$ is universal for generators $\{x_v:\ v\in \Lambda^0\}\cup\{x_\lambda:\ \lambda\in\Lambda^1\}$ subject to the relations
\begin{enumerate}
\item $\{x_v:\ v\in\Lambda^0\}$ are mutually orthogonal projections ;
\item $x^*_\lambda x_\lambda=x_{s(\lambda)}$ for all $\lambda\in\Lambda^1$;
\item $x_v=\sum_{\lambda \in v\Lambda^{e_i}} x_\lambda x_\lambda^*$ for each $v\in \Lambda^0$ and $1\leq i\leq k$;
\item $x_\lambda x_\mu = x_{\mu'} x_{\lambda'}$ for all $\lambda, \lambda', \mu, \mu' \in \Lambda^1$ satisfying $\lambda\mu = \mu'\lambda'$.
\end{enumerate}
\end{proposition}

This results is folklore. We will provide a proof, but defer it until after Proposition~\ref{prp:Pythagorean edge presentation}.

Observe that if $k=1$, then relation~(4) Proposition~\ref{prp:edge presentation} is trivial, so $\cOL$ is universal for generators $\{x_v, x_\lambda\}$ satisfying (1)--(3).

\begin{remark}A representation of $\cOL$ consists of a Hilbert space $H$ with a decomposition $\oplus_{v\in\Lambda^0}H_v$ over the vertices, and operators $X_\lambda$ with $\lambda\in\Lambda^1$ that are partial isometries such that each $x_\lambda$ has
domain $H_{s(\lambda)}$, codomain contained in $H_{r(\lambda)}$, and such that the $x_\lambda$ collectively satisfy axioms (3,4).
Hence, we may think of the $X_\lambda$ as operators $H_{s(\lambda)}\to H_{r(\lambda)}$ so that representations of $\cOL$ corresponds to a certain class of functors from $\Lambda$ to the category of Hilbert spaces. We will often identify these two
descriptions.
\end{remark}

\section{New algebras associated to (higher-rank) graphs}

We define the Pythagorean C$^*$-algebra $\cPL$ associated to a higher-rank graph $\Lambda$ generalising the construction appearing in \cite{Brothier-Jones19b}.
Then we define the non-self-adjoint algebra $\cAL$ that is generated by the standard set of generators of $\cPL$.

\subsection{Pythagorean type \texorpdfstring{C$^*$}{C*}-algebras and their representations}
We define a new algebra $\cPL$, universal for generators indexed by $\Lambda$ satisfying relations (1),(3)~and~(4) of Proposition~\ref{prp:edge presentation}, but not necessarily the partial-isometry relation~(2).
Under our conventions, the generators $a_\lambda$ of $\cPL$ correspond to the \emph{adjoints} of the generators $x_\lambda$ of $\cOL$.

\begin{definition}\label{def:P}
Define $\cPL$ to be the universal C$^*$-algebra with generator set $\{a_\lambda:\ \lambda\in\Lambda\}$ satisfying the relations
\begin{enumerate}
\item $\{a_v:\ v\in \Lambda^0\}$ are mutually orthogonal projection;
\item $\sum_{\lambda\in v\Lambda^{\le m}}a_\lambda^*a_\lambda=a_v$ for all $v\in \Lambda^0$ and $m \in \N^k$;
\item $a_\mu a_\nu=a_{\nu\circ\mu}$ for all composable $\lambda,\mu\in\Lambda$ (order reversed).
\end{enumerate}
\end{definition}

Again, relation~(1) implies that each $\|a_v\| \in \{0,1\}$, so~(2) implies that $0 \le \|a_\lambda\|^2 \le 1$ for $\lambda \in \Lambda^1$,
and \cite[Section II.8.3]{Blackadar06} shows that there is a universal C$^*$-algebra $\cPL$ as claimed. The universal property ensures that
there is a homomorphism $\cPL \to \cOL$ such that $a_\lambda \mapsto x_\lambda^*$ for all $\lambda$, and since the $x_\lambda$ are all nonzero, we
deduce that the $a_\lambda$ are nonzero as well.

The universal C$^*$-algebra $\cP$ with two generators $a,b$ and one relation
\begin{equation}\label{eq:pythagorean}a^*a+b^*b=1\end{equation}
was introduced in \cite{Brothier-Jones19b}. This is precisely $\cPL$ when $\Lambda$ is the $1$-graph with one vertex and two loops.
Jones named this the \emph{Pythagorean algebra} since~\eqref{eq:pythagorean} is a noncommutative analogue
of the Pythagorean equality. This C$^*$-algebra was designed primarily for the purpose of
constructing unitary representations of Richard Thompson's group $F$. Following Jones we call $\cPL$ the \emph{Pythagorean} algebra of $\Lambda$.

A representation on $H$ of $\cPL$ consists of a map $A : \Lambda \to B(H)$, $\lambda \mapsto A_\lambda$ satisfying the axioms (1)--(3) of Definition~\ref{def:P}.
We write $H_v$ for the range of the projection $A_v$ when $v$ is a vertex of $\Lambda$. As we did for $\cOL$, we will restrict attention to representations of $\cPL$
that are nondegenerate in the sense that $\big(\sum_{v \in F} A_v\big)_{F \Subset \Lambda^0}$ strong-operator converges to $1_H$. If $v\neq w$, then $H_v \perp H_w$, so then $H = \bigoplus_{v\in \Lambda^0} H_v$.
For $\lambda \in \Lambda$, $A_\lambda(H_{r(\lambda)})\subset H_{s(\lambda)}$ (we reverse the direction of the arrow because of relation~(2)). For $\xi \in H_{r(\lambda)}^\perp$,
\[
\langle A_\lambda A^*_\lambda \xi, \xi \rangle \le \langle A_{s(\lambda)} \xi, \xi \rangle = 0,
\]
so $H_{r(\lambda)}^\perp \subset \ker(A_{\lambda})$. Hence, we can interpret $A_\lambda$ as map from $H_{r(\lambda)}$ to $H_{s(\lambda)}$. That is, we can regard a representation
$A$ of $\cPL$ as a contravariant functor from $\Lambda$ to the category $\Hilb$ of Hilbert spaces with bounded linear maps as arrows.

As before, the nondegenerate representations of $\cPL$ are the objects of a category $\Rep(\cPL)$ whose arrows are the bounded linear maps intertwining the actions.

\begin{remark} The canonical quotient map
$$q:\cPL\onto\cOL, \ a_\lambda\mapsto x_\lambda^*$$
determines a forgetful functor
$$U:\Rep(\cOL)\to\Rep(\cPL),\ \pi\mapsto \pi\circ q$$
between representation categories.
\end{remark}

There is a more condensed presentation of $\cPL$ where the generator set consists of vertices and edges rather than all paths of $\Lambda$.

\begin{proposition}\label{prp:Pythagorean edge presentation}
The C$^*$-algebra $\cPL$ admits the presentation with generators the $a_\lambda$ with $\lambda\in\Lambda^0\cup\Lambda^1$ and the relations:
\begin{enumerate}
\item $\{a_v:\ v\in \Lambda^0\}$ are mutually orthogonal projection;
\item $\sum_{\lambda\in v\Lambda^{e_j}}a_\lambda^*a_\lambda=a_v$ for all $v\in \Lambda^0, 1\leq j\leq k$ such that $v\Lambda^{e_j} \not= \emptyset$.
\item $a_\lambda a_\mu = a_{\mu'} a_{\lambda'}$ for all $\lambda, \lambda', \mu, \mu' \in \Lambda^1$ satisfying $\mu\lambda = \lambda'\mu'$.
\end{enumerate}
\end{proposition}
\begin{proof}
Given a family satisfying (1)--(3) of Definition~\ref{def:P}, the subset $\{a_v:\ v \in  \Lambda^0\} \cup \{a_\lambda:\ \lambda \in \Lambda^1\}$
clearly satisfies (1)--(3) of this proposition, and generates $C^*(\{a_\lambda:\ \lambda \in \Lambda\})$ by Definition~\ref{def:P}(3) and the factorisation property.

Conversely, given a family satisfying (1)--(3) of this proposition, for $\lambda \in \Lambda$ we use the factorisation property repeatedly to write $\lambda = \lambda_1\circ\cdots \circ \lambda_{\|d(\lambda)\|_1}$ with each $\lambda_i \in \Lambda^1$, and
define $a_\lambda := a_{\lambda_{\|d(\lambda)\|_1}}\cdots a_{\lambda_1}$. Condition~(3) ensures that this is independent of the choice of factorisation $\lambda = \lambda_1\circ\cdots \circ \lambda_{\|d(\lambda)\|_1}$. This family $a_\lambda$ clearly
satisfies Definition~\ref{def:P}(1). Given composable $\mu,\nu \in \Lambda$ with $\|d(\mu)\|_1 = m$ and $\|d(\nu)\|_1 = n$, fix factorisations $\mu = \mu_1 \circ \cdots \circ \mu_m$ and $\nu = \nu_1 \circ \cdots \circ \nu_n$ with each $\mu_i,\ni_j \in
\Lambda^1$. Then $\mu\nu = \mu_1 \circ \cdots \circ \mu_m \circ \nu_1 \circ \cdots \circ \nu_n$ is also such a factorisation, so $a_{\mu\nu} = a_{\nu_n} \cdots \circ a_{\nu_1} a_{\mu_m} \cdots a_{\mu_1} = a_\nu a_\mu$, which is Definition~\ref{def:P}(3).
That the $a_\mu$ satisfy Definition~\ref{def:P}(2) follows from the argument of \cite[Proposition~3.11]{Raeburn-Sims-Yeend03} (which does not make use of the Cuntz--Krieger relation~(4) of \cite[Definition~3.3]{Raeburn-Sims-Yeend03}).
\end{proof}

We can now pay the debt incurred at Proposition~\ref{prp:edge presentation}.

\begin{proof}[Proof of Proposition~\ref{prp:edge presentation}]
By definition, $\cOL$ is the quotient of $\cPL$ by the ideal $I$ generated by $\{a_{s(\lambda)} - a_\lambda a^*_\lambda : \lambda \in \Lambda\}$.
By Proposition~\ref{prp:Pythagorean edge presentation}, the C$^*$-algebra of Proposition~\ref{prp:edge presentation} is the quotient of $\cPL$ by
the $I_0$ generated by $\{a_{s(\lambda)} - a_\lambda a^*_\lambda : \lambda \in \Lambda^1\}$. Clearly $I_0 \subset I$. Let $q_0 : \cPL
\to \cPL/I_0$ be the quotient map. Since the $q_0(a_\lambda)$ satisfy Proposition~\ref{prp:edge presentation}(3),
an induction on $\|d(\lambda)\|_1$ shows that they satisfy Definition~\ref{def:O}(3). Hence $a_{s(\lambda)} - a_\lambda a^*_\lambda \in I_0$
for all $\lambda \in \Lambda$. Thus $I \subset I_0$, and the two are equal.
\end{proof}

\subsection{A non-self-adjoint operator algebra}

\begin{definition}\label{def:A}
We define
\[
\cAL := \overline{\operatorname{span}}\{a_\lambda:\ \lambda \in \Lambda\} \subset \cPL.
\]
A $\cAL$-\emph{module} (or just a \emph{module}) is a Hilbert space $H$ together with a continuous homomorphism $\pi:\cAL\to B(H)$ such that for each $v \in \Lambda^0$, the operator $\pi(a_v)$ is a projection, and
\[
    \sum_{\lambda\in v\Lambda^{e_j}} \pi(a_\lambda)^*\pi(a_\lambda)=\pi(a_v)
\]
for all $1\leq j\leq k$ such that $v\Lambda^{e_j} \not= \emptyset$.
\end{definition}

Note that we use the term ``representations'' for continuous linear actions of the C$^*$-algebras $\cOL$ and $\cPL$, while we use the term ``modules'' for continuous linear actions of the non-self-adjoint algebras $\cAL$
by operators satisfying the conditions of Definition~\ref{def:A}. As before, we will restrict attention exclusively to modules $(\pi,H)$ of $\cAL$ that are nondegenerate in the sense that $\big(\sum_{v \in F} \pi(a_v)\big)_{F \Subset \Lambda^0}$ weak-operator converges to $1_H$.
We will often write $A_\lambda$ for $\pi(a_\lambda)$.

\begin{remark} The algebra $\cAL$ is by definition not closed under involution in $\cPL$. Indeed, for $\lambda \in \Lambda$, we have $a_\lambda^* \in \cAL$ if and only
if $\lambda \in \Lambda^0$. To see this, observe that the ``if'' implication is trivial. For the ``only if'' implication, let $\gamma$ be the gauge
action of $\mathbb{T}^k$ on $\cOL$ from \cite[Section~4.1]{Raeburn-Sims-Yeend03}. Then the spectral subspaces
$(\cOL)_n := \{b \in \cOL:\ \gamma_z(b) = z^{-n} b\}$ satisfy $q(\cAL) \subset \bigcup_{n \in \N^k} (\cOL)_{-n}$ where $q:\cPL\onto\cOL$ is the quotient map.
Whereas each $q(a_\lambda^*) \in (\cOL)_{d(\lambda)}$. Since $(\cOL)_m \cap (\cOL)_n = \{0\}$ for $m \neq n$, it follows that $a_\lambda^* \not \in \cAL$
when $d(\lambda) \not= 0$.
\end{remark}

Note that a $\cAL$-module is equivalent to a collection of Hilbert spaces $(H_v)_{v\in\Lambda^0}$ and a family of bounded linear operators $(A_\lambda)_{\lambda\in\Lambda^1}$ such that $A_\lambda : H_{r(\lambda)}\to H_{s(\lambda)}$, $A_{\lambda}A_{\mu} =
A_{\mu'}A_{\lambda'}$ for all $\lambda, \lambda', \mu, \mu' \in \Lambda^1$ satisfying $\mu\lambda = \lambda'\mu'$, and for each vertex $v$ and colour $i$ with $v\Lambda^{e_i}$ non-empty the map
$$H_v\to \oplus_{\lambda\in v\Lambda^{e_i}} H_{s(\lambda)}, \ \xi\mapsto (A_\lambda\xi)_{\lambda\in v\Lambda^{e_i}}$$
is an isometry.

The collection of $\cAL$-modules constitutes the object space of a category $\Mod(\cAL)$ whose arrows $\phi : (\pi, H) \to (\rho, K)$ are bounded linear maps intertwining $\pi$ and $\rho$: $\phi \circ \pi(a_\lambda) = \rho(a_\lambda) \circ \phi$.
If $(\pi, H)$ is a representation of $\cPL$, then restricting the action of $\cPL$ to an action of $\cAL$ makes $H$ into a $\cAL$-module. This provides a forgetful functor
$$\Rep(\cPL)\to\Mod(\cAL).$$
Conversely, any module $(\pi,H)$ extends to a representation of $\cPL$ by the universal property of the latter.

\begin{remark}\label{rmk:wide subcategory}
The category $\Mod(\cAL)$ has the same objects as $\cPL$. But it has more arrows: an arrow $\phi : (\pi, H) \to (\rho, K)$ is required only to be a Banach-space morphism that intertwines the images of generators $a_\lambda$ under $\pi$ with those under $\rho$.
Since $\phi$ is not required to intertwine the adjoints of the $a_\lambda$ it may not define an arrow in $\Rep(\cPL)$.
We will often implicitly consider $\Rep(\cPL)$ as a wide subcategory of $\Mod(\cAL)$ (same objects but fewer arrows) and thus identify $\cAL$-modules with representations of $\cPL$.
\end{remark}

\section{Lifting representations}

This section explains how to dilate an $\cAL$-module (equivalently a representation of $\cPL$) to a representation of $\cOL$.
This is the most important construction of this paper.
We start by defining a family of trees constructed from $\Lambda$.
Then we explicitly construct a Hilbert space $L$ from an $\cAL$-module $H$. For clarity, we provide two equivalent constructions of $L$: one using classes of path decorated by vectors of $H$ and one using classes of trees with leaves decorated by vectors of $H$.
We then define an action of $\cOL$ on $L$, explain how this construction is functorial, and collect some categorical results.

\subsection{Trees associated to a \texorpdfstring{$k$}{k}-graph}

We will formally define the trees of $\Lambda$ as subsets of its path space obtained by an inductive process; then we will explain how we view these subsets as trees.

\begin{definition}
Fix $v \in \Lambda^0$. The \emph{trivial caret} $\vee_v$ at $v$ is the singleton $\vee_v = \{v\}$. The \emph{basic carets} (or simply \emph{carets}) at $v$ are the sets
\[
\vee_{v, i} = v \Lambda^{\le e_i} = \begin{cases} v\Lambda^{e_i} &\text{ if $v\Lambda^{e_i} \not= \emptyset$}\\ \{v\} &\text{ otherwise}\end{cases}
\]
indexed by $1 \le i \le k$. We inductively define the collection of trees at $v$ as follows: the trivial caret $\vee_v$ is a tree at $v$, and given a tree $t$ at $v$, an element $\mu \in t$, and an index $i \in \{1, \dots, k\}$, the set
\[
t *_\mu {\vee_{s(\mu), e_i}} := (t \setminus \{\mu\}) \cup \mu {\vee_{s(\mu),i}}
\]
is a tree at $v$ where recall $\mu\vee_{s(\mu), i} := \{\mu\lambda: \lambda \in \vee_{s(\mu), i}\}$.

We call $v$ the \emph{root} of $t$, and denote it by $r(t)$, and we call the elements of $t$ the \emph{leaves}. The lone vertex $v$ of $\vee_v$ is simultaneously the root and the unique leaf.
\end{definition}

The basic carets are trees since each $\vee_{v, i} =  {\vee_v} *_v {\vee_{v, i}}$.

We think of the trees of $\Lambda$ pictorially as trees of height~1: that is, given a tree $t$, we represent it as a planar diagram by placing a root labelled $v$ at the bottom, and leaves labelled by the elements of $t$ at the top ordered from left to
right in an arbitrary manner. For instance if $v\Lambda^{e_1}=\{a,b,c\}$, then the $(v, 1)$-caret $\vee_{v, 1}$ is represented by:
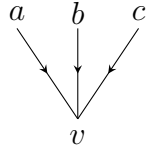
\begin{center}
	\pgfdeclarelayer{nodelayer}
	\pgfdeclarelayer{edgelayer}
	\pgfsetlayers{nodelayer,edgelayer,main}
	\begin{tikzpicture}[scale = 0.8, decoration={markings, mark=at position 0.5 with {\arrow{stealth}}}]
		\begin{pgfonlayer}{nodelayer}
			\node [] (0) at (0, 0) {};
			\node [] (1) at (-1, 1.5) {};
			\node [] (2) at (0, 1.5) {};
			\node [] (3) at (1, 1.5) {};
			\node [] (4) at (0, -0.3) {$v$};
			\node [] (5) at (1, 1.75) {$c$};
			\node [] (6) at (0, 1.75) {$b$};
			\node [] (7) at (-1, 1.75) {$a$};
		\end{pgfonlayer}
		\begin{pgfonlayer}{edgelayer}
			\draw[postaction={decorate}] (1.center) to (0.center);
			\draw[postaction={decorate}] (2.center) to (0.center);
			\draw[postaction={decorate}] (3.center) to (0.center);
		\end{pgfonlayer}
	\end{tikzpicture}
\end{center}

Given a tree $t$ at $v$, an element $\mu \in t$ and an index $i \in \{1,\dots, k\}$ such that $s(\mu)\Lambda^{e_i} \not= \emptyset$, the diagram for $t'$ is obtained from the diagram for $t$ as follows: first we glue a copy of the diagram for
$\vee_{s(\mu), i}$ to the leaf of the diagram for $t$ labelled $\mu$; and then we delete the internal vertex labelled $\mu$ and relabel each leaf of $\vee_{s(\mu), i}$ as the composition of its existing label with $\mu$. For example, suppose that $t =
\{\lambda, \mu, \nu\}$ and $s(\mu)\Lambda^{e_i} = \{a, b\}$. Let $t' := t *_\mu \vee_{s(\mu), i}$. Gluing $\vee_{s(\mu), i}$ to the leaf $\mu$ of $t$ yields the graph $t^0$ on the left below; deleting the internal vertex $s(\mu)$ and relabelling the leaves
of $\vee_{s(\mu), i}$ then yields the diagram for $t'$ on the right below.
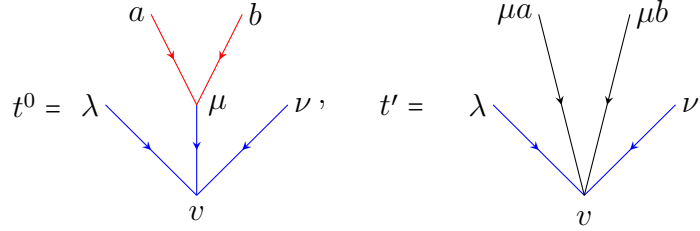
\begin{center}
	\pgfdeclarelayer{nodelayer}
	\pgfdeclarelayer{edgelayer}
	\pgfsetlayers{nodelayer,edgelayer,main}
	\begin{tikzpicture}[scale = 0.6, decoration={markings, mark=at position 0.5 with {\arrow{stealth}}}]
		\begin{pgfonlayer}{nodelayer}
			\node [] (0) at (0, 0) {};
			\node [] (1) at (-2, 2) {};
			\node [] (2) at (0, 2) {};
			\node [] (3) at (2, 2) {};
			\node [anchor=north, inner sep=0pt] (7) at (0.south) {$v$};
			\node [anchor=west, inner sep=0pt] (8) at (2.east) {$\mu$};
			\node [anchor=east, inner sep=-2pt] at (1.west) {$\lambda$};
			\node [anchor=west, inner sep=-2pt] at (3.east) {$\nu$};
			\node [] (9) at (-3.5, 2) {$t^0 = $};
			\node [] (10) at (8.5, 0) {};
			\node [] (11) at (6.5, 2) {};
			\node [] (13) at (10.5, 2) {};
			\node [anchor=east, inner sep=-2pt] at (11.west) {$\lambda$};
			\node [anchor=west, inner sep=-2pt] at (13.east) {$\nu$};
			\node [] (17) at (8.5, -0.5) {$v$};
			\node [] (19) at (4.5, 2) {$t' = $};
			\node [] (20) at (-1, 4) {};
			\node [] (21) at (1, 4) {};
			\node [anchor=east, inner sep=-2pt] at (20.west) {$a$};
			\node [anchor=west, inner sep=-2pt] at (21.east) {$b$};
			\node [] (24) at (7.5, 4) {};
			\node [] (25) at (9.5, 4) {};
			\node [] (28) at (2.75, 2) {$,$};
			\node [anchor=east, inner sep=-2pt] at (24.west) {$\mu a$};
			\node [anchor=west, inner sep=-2pt] at (25.east) {$\mu b$};
		\end{pgfonlayer}
		\begin{pgfonlayer}{edgelayer}
			\draw [postaction={decorate}, blue] (1.center) to (0.center);
			\draw [postaction={decorate}, blue] (2.center) to (0.center);
			\draw [postaction={decorate}, blue] (3.center) to (0.center);
			\draw [postaction={decorate}, blue] (11.center) to (10.center);
			\draw [postaction={decorate}, blue] (13.center) to (10.center);
			\draw [postaction={decorate}, red] (20.center) to (2.center);
			\draw [postaction={decorate}, red] (21.center) to (2.center);
			\draw [postaction={decorate}] (25.center) to (10.center);
			\draw [postaction={decorate}] (24.center) to (10.center);
		\end{pgfonlayer}
	\end{tikzpicture}
\end{center}

We define a partial order on trees: $t \le t'$ if $t' \subset t\Lambda$. Equivalently, $\le$ is the partial order generated by $t \le t *_\mu \vee_{s(\mu), i}$ whenever $\mu \in t$ and $s(\mu)\Lambda^{e_i} \not= \emptyset$.

\begin{example} \label{ex:tree}
A key artifact of our definition of a tree is that the same tree can be constructed by gluing carets of different colours in different orders. This is a consequence of removing interior nodes and labelling edges and vertices by paths in $\Lambda$, which
brings the factorisation rules of $\Lambda$ into play.

For instance, consider the $2$-graph $\Lambda$ where $\Lambda^0 = \{v\}$ and $\Lambda^{e_1} = \{\nu\},\ \Lambda^{e_2} = \{\mu_1, \mu_2\}$ (i.e. the one-skeleton is the bouquet with three loops), with factorisation property $\nu\mu_1 = \mu_2e$, $\nu\mu_2 =
\mu_1\nu$. Let $t_1 = {\vee_{v, 1}} *_{\nu} {\vee_{v, 2}}$ and $t_2 = ({\vee_{v, 2}} *_{\mu_1} {\vee_{v,1}}) *_{\mu_2} {\vee_{v, 1}}$. So $t_1$ is the tree formed by starting with the caret $\vee_{v, 1}$ and gluing $\vee_{v, 2}$ at the unique leaf; and
$t_2$ is the tree formed by starting with $\vee_{v,2}$ and gluing $\vee_{v_1}$ at each leaf (in either order). The factorisation rules in $\Lambda$ ensure that $t_1$ and $t_2$ are equal. This is diagrammatically shown below, where for the two diagrams on
the sides we include the interior nodes to demonstrate the carets used to construct them (hence they are not \emph{trees} in the sense of this article) while the middle diagram is the common tree $t_1 = t_2$.
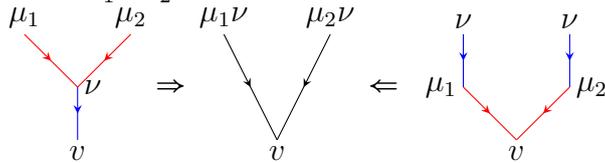
\begin{center}
	\pgfdeclarelayer{nodelayer}
	\pgfdeclarelayer{edgelayer}
	\pgfsetlayers{nodelayer,edgelayer,main}
	\begin{tikzpicture}[scale=0.7, decoration={markings, mark=at position 0.5 with {\arrow{stealth}}}]
		\begin{pgfonlayer}{nodelayer}
			\node [] (0) at (-0.5, 0) {};
			\node [anchor=north, inner sep=-2pt] at (0.south) {$v$};
			\node [] (1) at (-0.5, 1) {};
			\node [anchor=west, inner sep=-2pt] at (1.east) {$\nu$};
			\node [] (2) at (-1.5, 2) {};
			\node [] (3) at (0.5, 2) {};
			\node [anchor=south, inner sep=-2pt] at (2.north) {$\mu_1$};
			\node [anchor=south, inner sep=-2pt] at (3.north) {$\mu_2$};
			\node [] (4) at (1.25, 1) {$\Rightarrow$};
			\node [] (5) at (7.75, 0) {};
			\node [anchor=north, inner sep=-2pt] at (5.south) {$v$};
			\node [] (6) at (6.75, 1) {};
			\node [] (7) at (8.75, 1) {};
			\node [anchor=east, inner sep=-2pt] at (6.west) {$\mu_1$};
			\node [anchor=west, inner sep=-2pt] at (7.east) {$\mu_2$};
			\node [] (8) at (8.75, 2) {};
			\node [] (9) at (6.75, 2) {};
			\node [anchor=south, inner sep=-2pt] at (8.north) {$\nu$};
			\node [anchor=south, inner sep=-2pt] at (9.north) {$\nu$};
			\node [] (10) at (5.25, 1) {$\Leftarrow$};
			\node [] (11) at (3.25, 0) {};
			\node [anchor=north, inner sep=-2pt] at (11.south) {$v$};
			\node [] (12) at (2.25, 2) {};
			\node [] (13) at (4.25, 2) {};
			\node [anchor=south, inner sep=-2pt] at (12.north) {$\mu_1\nu$};
			\node [anchor=south, inner sep=-2pt] at (13.north) {$\mu_2\nu$};
		\end{pgfonlayer}
		\begin{pgfonlayer}{edgelayer}
			\draw [postaction={decorate}, blue] (1.center) to (0.center);
			\draw [postaction={decorate}, blue] (9.center) to (6.center);
			\draw [postaction={decorate}, blue] (8.center) to (7.center);
			\draw [postaction={decorate}, red] (2.center) to (1.center);
			\draw [postaction={decorate}, red] (3.center) to (1.center);
			\draw [postaction={decorate}, red] (6.center) to (5.center);
			\draw [postaction={decorate}, red] (7.center) to (5.center);
			\draw [postaction={decorate}] (12.center) to (11.center);
			\draw [postaction={decorate}]  (13.center) to (11.center);
		\end{pgfonlayer}
	\end{tikzpicture}
\end{center}
\end{example}

If $m\in\N^k$, then we write $\Lambda^{\leq m}$ for the set of path $\mu$ of degree $m$ or of maximal subdegree (i.e.~$d(\mu)=m$ or $d(\mu)<m$ and if $\mu\circ \nu$ is a strictly larger path, then $d(\mu\circ\nu)$ is not smaller or equal to $m$).
Local convexity implies that
\begin{equation}\label{eq:LambdaLE factorisation}
\Lambda^{\leq m}\cdot \Lambda^{\leq n} = \Lambda^{\leq (m+n)} \text{ for all } m,n\in\N^k
\end{equation}
where $\Lambda^{\leq m}\cdot \Lambda^{\leq n}$ means the set of all products $\mu\circ\nu$ with $\mu\in\Lambda^{\leq m},\nu\in\Lambda^{\leq n}$, see \cite[Lemma 2.2]{Sims10} for a proof.
This implies that the set $v\cT$ admits a cofinal sequence and is thus directed.

\begin{lemma} \label{lem:complete-tree}
For each vertex $v \in \Lambda^0$ and $m \in \N^k$, the set $t_v^m := v\Lambda^{\le m}$ is a tree at $v$.
The collection of all $t_v^m$ with $m\in\N^k$ is a \emph{cofinal} net in the poset $(v\cT,\leq)$: if $m\leq n$ in $\N^k$, then $t_v^m\leq t_v^n$ and for every tree $s\in v\cT$ there exists $m\in \N^k$ such that $t\leq t^m_v$.
In particular, the poset $(v\cT,\leq)$ is \emph{directed}: for all trees $s,s'\in v\cT$ there exists $z\in v\cT$ so that $s\leq z$ and $s'\leq z$.
\end{lemma}
\begin{proof}
We proceed by induction on the length of $m$. The base case when the length is $0$ is trivial. Now suppose the above statement is true for all degrees with length $\ell$ or less. Let $m$ be a degree with length $\ell+1$. There exists $i$ such that the
$i$th component of $m$ is nonzero. Let $n = m - e_i \in \N^k$. By local convexity we have $v\Lambda^{\leq m} = v\Lambda^{\leq n} \cdot \Lambda^{\leq e_i}$. Let $p = |\{\mu \in v\Lambda^{\leq n} : s(\mu)\Lambda^{e_i} \neq \emptyset\}|$, and fix an
enumeration $\{\mu \in v\Lambda^{\leq n} : s(\mu)\Lambda^{e_i} \neq \emptyset\} = \{\mu_1, \mu_2, \dots, \mu_p\}$. Then
\[
v \Lambda^{\leq n} \cdot \Lambda^{\leq e_i}
    = ((v \Lambda^{\leq n} *_{\mu_1} \Lambda^{e_i}) *_{\mu_2} \Lambda^{e_i}) \cdots *_{\mu_p} \Lambda^{e_i},
\]
so the induction assumption proves the first statement. The rest of the lemma follows easily.	
\end{proof}

\subsection{Construction of a Hilbert space}\label{sec:lifting-representation}
Fix an $\cAL$-module $H$ with associated operators $A_\lambda,\lambda\in\Lambda$.
We decompose $H$ as the Hilbert direct sum of the $H_v,v\in\Lambda^0$, and denote the inner-product on $H$ by $\langle \cdot, \cdot\rangle_H$.

\subsubsection{The first construction of \texorpdfstring{$L$}{L}.} \label{subsec:first-constr-Hilb}

\begin{lemma}\label{lem:isometries}
Let $v \in \Lambda^0$ and $m \in \N^k$. The operator $$A_{v\Lambda^{\le m}} := (A_\lambda)_{\lambda \in v\Lambda^{\le m}} : H_v \to \bigoplus_{\lambda \in \Lambda^{\le m}} H_{s(\lambda)}$$  is an isometry.
\end{lemma}
\begin{proof}
We induct on $\|m\|_1$. If $m = 0$ the results is trivial since $A_v$ is the identity on $H_v$. If $\|m\|_1 = 1$ then $m = e_i$ for some $i$,
and either $v\Lambda^{e_i}$ is empty in which case $A_{v\Lambda^{\le m}} = A_v$ is an isometry as already discussed, or
$A_{v\Lambda^{\le m}} = (A_\lambda)_{\lambda \in v\Lambda^{e_i}}$ is an isometry by definition of a $\cAL$-module.
Now suppose that $A_{v\Lambda^{\le m}}$ is an isometry
whenever $\|m\|_1 \le N$ and fix $m$ with $\|m\|_1 = N+1$. Then there is an index $i$ such that $m_i > 0$; let $n = m - e_i \in \N^k$. By \cite[Lemma~3.12]{Raeburn-Sims-Yeend03},
we have $\Lambda^{\le m} = \Lambda^{\le n} \Lambda^{\le e_i}$. So for $h \in H_v$,
\[
\|A_{v\Lambda^{\le m}} h \|^2
    = \sum_{\lambda \in \Lambda^{\le m}} \|A_\lambda h \|^2
    = \sum_{\mu \in \Lambda^{\le n}} \Big(\sum_{\alpha \in \Lambda^{\le e_i}} \|A_\alpha (A_\mu h)\|^2\Big).
    = \sum_{\mu \in \Lambda^{\le n}} \big\|A_{s(\mu) \Lambda^{\le e_i}} (A_\mu h)\big\|^2.
\]
The base case applied to each $A_{s(\mu) \Lambda^{\le e_i}}$ shows that this is equal to $\|A_{v\Lambda^{\le n}} h\|^2$,
and the inductive hypothesis then gives the result.
\end{proof}

\subsubsection{First construction of the Hilbert space.}\label{sec:first construction}
Let $L^{00}$ be the formal linear span of pairs $(\lambda,\xi)$ with $\lambda\in \Lambda$ and $\xi\in H_{s(\lambda)}$ (recall that $s(\lambda)$ is the source of the path $\lambda$).
Let
\begin{align*}
N = \operatorname{span}\Big(\{c&(\lambda,\xi) - (\lambda, c\xi):\ c \in \C, \lambda \in \Lambda, \xi \in H_{s(\lambda)}\}\\
    &\cup \{(\lambda,\xi) + (\lambda, \xi') - (\lambda, \xi + \xi'):\ \lambda \in \Lambda, \xi,\xi' \in H_{s(\lambda)}\}\\
    &\cup \Big\{(\lambda,\xi) - \sum_{\mu\in s(\lambda)\Lambda^{e_i}} (\lambda\mu , A_\mu\xi):\ \lambda \in \Lambda, \xi \in H_{s(\lambda)}, 1 \le i \le k, s(\lambda)\Lambda^{e_i} \not= \emptyset\Big\}\Big),
\end{align*}
and let $L^0 := L^{00}/N$; we write $[\lambda,\xi] := (\lambda,\xi) + N \in L^0$ for each $\lambda \in \Lambda$ and $\xi \in H_{s(\lambda)}$. By
definition of $N$, the spanning elements of $L^0$ satisfy the \emph{caret} relations
\begin{align}\tag{C} \label{eqn:caret}
	[\lambda,\xi] = \sum_{\mu\in s(\lambda)\Lambda^{e_i}} [\lambda\mu , A_\mu\xi]	
\end{align}
whenever $s(\lambda)\Lambda^{e_i}$ is nonempty.
\makeatletter
\def\@currentlabel{C}
\makeatother
We interpret this equality as gluing the $(v,i)$-\emph{caret} at the source of the path $\lambda$ and lifting the vector $\xi$ from the source of $\lambda$ to the source of each edge $\mu$ of the caret.
An induction on $\|n\|_1$ using~\eqref{eq:LambdaLE factorisation} and the caret relations shows that
\begin{equation} \label{eqn:general caret}
	[\lambda,\xi] = \sum_{\mu\in s(\lambda)\Lambda^{\le n}} [\lambda\mu , A_\mu\xi]	\quad\text{ for all $n \in \N^k$.}
\end{equation}

We equip $L^0$ with an inner product.

\begin{proposition}\label{prp:L inner prod}
	There is a unique sesquilinear form $\langle \cdot, \cdot \rangle$ on $L^0$ such that for $\lambda, \mu \in \Lambda$, and $\xi \in H_{s(\lambda)},\ \eta \in H_{s(\mu)}$,
	\begin{enumerate}
		\item if $r(\lambda) \neq r(\mu)$ then
		\[\langle [\lambda, \xi], [\mu, \eta] \rangle = 0;\]
		\item if $d(\lambda) = d(\mu)$ then
			\[\langle [\lambda, \xi], [\mu, \eta] \rangle =
			\begin{cases}
				\langle \xi, \eta \rangle_{H}, &\quad \lambda = \mu \\
				0, &\quad \lambda \neq \mu
			\end{cases};\]		
		\item if $i \le k$ satisfies $d(\lambda)_i < d(\mu)_i$ and $s(\lambda) \Lambda^{e_i} = \emptyset$, then
        \[\langle [\lambda, \xi], [\mu, \eta] \rangle = 0.\]
	\end{enumerate}
    This form is an inner product, and for $\lambda, \mu \in \Lambda$, and $\xi \in H_{s(\lambda)}$ and $\eta \in H_{s(\mu)}$, we have
    \begin{equation}\label{eq:ip formula}
    \langle [\lambda, \xi], [\mu, \eta]\rangle = \sum_{\lambda\nu = \mu\gamma \in \Lambda^{\leq d(\lambda) + d(\mu)}} \langle A_{\nu} \xi, A_{\gamma} \eta\rangle_H.
    \end{equation}
\end{proposition}

\begin{proof}
Since the spanning elements $(\lambda, \xi)$ of $L^{00}$ are by definition linearly independent, there is a unique sesquilinear form $[\cdot \mid \cdot] : L^{00} \to \C$
such that
\begin{equation}\label{eq:formdef}
 \big[ (\lambda, \xi) \mathbin{\big|} (\mu, \eta)\big] = \sum_{\lambda\nu = \mu\gamma \in \Lambda^{\leq d(\lambda) + d(\mu)}} \langle A_{\nu} \xi, A_{\gamma} \eta\rangle_H
\end{equation}
for all $\lambda,\mu \in \Lambda$ and all $\xi \in H_{s(\lambda)}$ and $\eta \in H_{s(\mu)}$. Since $\langle \cdot, \cdot \rangle_H$ is an inner product,
we have
\[
\big[c(\lambda,\xi) - (\lambda, c\xi) \mathbin{\big|} v\big] = 0 = \big[(\lambda,\xi) + (\lambda, \xi') - (\lambda, \xi - \xi') \mathbin{\big|} v\big]
\]
for all $c \in \C$, $\lambda \in \Lambda$, $\xi, \xi' \in H_{s(\lambda)}$ and $v \in L^{00}$. Moreover, given
$\lambda \in \Lambda$, $\xi \in H_{s(\lambda)}$, and $1 \le i \le k$ such that $s(\lambda)\Lambda^{e_i} \not= \emptyset$,
we have
\begin{align*}
\Big[(\lambda, \xi) \mathbin{\Big|} &\sum_{\alpha \in s(\lambda)\Lambda^{e_i}} (\lambda\alpha, A_\alpha \xi)\Big]
    = \sum_{\alpha \in s(\lambda)\Lambda^{e_i}} \Big[(\lambda, \xi) \mathbin{\Big|} (\lambda\alpha, A_\alpha \xi)\Big]\\
    &= \sum_{\alpha \in s(\lambda)\Lambda^{e_i}} \sum_{\lambda\tau = \lambda\alpha\rho \in \Lambda^{\le 2d(\lambda) + e_i}} \langle A_\tau \xi A_\rho A_\alpha\xi\rangle_H
    = \sum_{\tau \in s(\lambda)\Lambda^{\le 2d(\lambda) + e_i}} \langle A_\tau \xi, A_\tau \xi\rangle_H,
\end{align*}
and this is equal to $\big[(\lambda, \xi) \mathbin{\big|} (\lambda,\xi)\big]$ by Lemma~\ref{lem:isometries}. Using this multiple times, we see that
\begin{align*}
\Big[(\lambda, \xi) - {}&\sum_{\alpha \in s(\lambda)\Lambda^{e_i}} (\lambda\alpha, A_\alpha \xi) \mathbin{\Big|}
        (\lambda, \xi)- \sum_{\alpha \in s(\lambda)\Lambda^{e_i}} (\lambda\alpha, A_\alpha \xi)\Big]\\
    &= \big[(\lambda, \xi) \mathbin{\big|} (\lambda, \xi)\big]
        - \Big[(\lambda, \xi) \mathbin{\Big|} \sum_{\alpha \in s(\lambda)\Lambda^{e_i}} (\lambda\alpha, A_\alpha \xi)\Big]\\
        &\qquad{}- \Big[\sum_{\alpha \in s(\lambda)\Lambda^{e_i}} (\lambda\alpha, A_\alpha \xi) \mathbin{\Big|} (\lambda, \xi)\Big]
        + \Big[\sum_{\alpha \in s(\lambda)\Lambda^{e_i}} (\lambda\alpha, A_\alpha \xi) \mathbin{\Big|} \sum_{\alpha \in s(\lambda)\Lambda^{e_i}} (\lambda\alpha, A_\alpha \xi)\Big]\\
    &= \big[(\lambda, \xi) \mathbin{\big|} (\lambda, \xi)\big] - \big[(\lambda, \xi) \mathbin{\big|} (\lambda, \xi)\big]  - \big[(\lambda, \xi) \mathbin{\big|} (\lambda, \xi)\big]  + \big[(\lambda, \xi) \mathbin{\big|} (\lambda, \xi)\big]\\
    &= 0.
\end{align*}
So $N$ is contained in the kernel $\{\xi : [\xi \mid \xi] = 0\}$ of $[\cdot \mid \cdot]$ and hence $[\cdot \mid \cdot]$ descends to a sesquilinear form $\langle\cdot, \cdot\rangle$
on $L^0$.

This form satisfies~(1) because if $r(\lambda) \neq r(\mu)$ then the sum~\eqref{eq:formdef} is empty. It satisfies~(2) because
if $d(\lambda) = d(\mu)$, then the sum is empty if $\lambda \neq \mu$, and if $\lambda = \mu$, then
\[
\sum_{\lambda\nu = \mu\gamma \in \Lambda^{\leq d(\lambda) + d(\mu)}} \langle A_\tau \xi, A_\tau \xi\rangle_H
    = \sum_{\nu \in \Lambda^{\leq d(\lambda)}} \langle A_\nu \xi, A_\nu \eta\rangle_H
    = \langle \xi, \eta\rangle_H
\]
by Lemma~\ref{lem:isometries}. It satisfies~(3) because if $d(\lambda)_i < d(\mu)_i$ and $s(\lambda) \Lambda^{e_i} = \emptyset$,
then in particular there are no solutions to $\lambda \nu = \mu\gamma$, so the sum~\eqref{eq:formdef} is again empty.

For uniqueness, note that if $\langle \cdot, \cdot \rangle$ is a sesquilinear form on $L^0$ satisfying (1)--(3), then for $\lambda,\mu \in \Lambda$ and $\xi \in H_{s(\lambda)}$ and $\eta \in H_{s(\mu)}$,
use~\eqref{eqn:general caret} and then properties (2)~and~(3) to see that
\[
    \langle [\lambda, \xi], [\mu, \eta]\rangle
        = \sum_{\substack{\nu \in s(\lambda)\Lambda^{\le d(\mu)}\\ \gamma \in s(\mu)\Lambda^{\le d(\lambda)}}} \langle [\lambda\nu, A_\nu \xi], [\mu\gamma, A_\gamma \eta]\rangle
        = \sum_{\lambda\nu = \mu\gamma \in \Lambda^{\le d(\lambda) + d(\mu)}} \langle A_\nu \xi, A_\gamma \eta]\rangle_H;
\]
that is, any sesquilinear form satisfying (1)--(3) satisfies~\eqref{eq:ip formula}.

It remains to show that $\langle\cdot, \cdot\rangle$ is an inner product. For this, fix a finite sum $v = \sum_{\lambda \in F} c_\lambda(\lambda, \xi_\lambda) \in L^{00} \setminus N$. We must show
that $\langle v  + N, v + N\rangle \neq 0$. Let $n = \bigvee_{\lambda \in F} d(\lambda)$. Then each
\[
c_\lambda(\lambda, \xi_\lambda) - \sum_{\mu \in s(\lambda)^{\lambda^{\le n - d(\lambda)}}} \big(\lambda\mu, c_\lambda A_\mu \xi_\lambda\big)
\]
belongs to $N$, and so
\[
v + N = \sum_{\lambda \in F} \sum_{\mu \in s(\lambda)^{\lambda^{\le n - d(\lambda)}}} \big(\lambda\mu, c_\lambda A_\mu \xi_\lambda\big).
\]
So we can assume that $F \subseteq \Lambda^{\le n}$. In this case, for distinct $\lambda, \lambda' \in F$ we have
$\lambda\Lambda \cap \lambda'\Lambda = \emptyset$. Hence
\[
\langle v + N, v + N\rangle
    = \sum_{\lambda \in F} \langle [\lambda, c_\lambda\xi_\lambda], [\lambda, c_\lambda\xi_\lambda]\rangle
    = \sum_{\lambda \in F} \langle c_\lambda\xi_\lambda, c_\lambda\xi_\lambda\rangle_H
    = \sum_{\lambda \in F} \|c_\lambda \xi_\lambda\|^2,
\]
which is nonzero because $\sum_\lambda c_\lambda(\lambda,\xi_\lambda) \not\in N$, and so at least one $c_\lambda \xi_\lambda$ is nonzero.
\end{proof}

\begin{remark}\label{rmk:tree identity}
Let $t$ be a tree of $\Lambda$ at $v$, suppose that $\mu \in t$ and suppose that $s(\mu)\Lambda^{e_i} \not= \emptyset$. Fix $\xi \in H_v$.
By the caret relation~\eqref{eqn:caret}, we have $[\mu, A_\mu \xi] = \sum_{\alpha \in s(\mu)\Lambda^{e_i}} [\mu, A_\alpha A_\mu \xi] = \sum_{\lambda \in \mu\Lambda^{e_i}} [\lambda, A_\lambda\xi]$.
Consequently,
\[
\sum_{\nu \in t} [\nu, A_\nu \xi]
    = \sum_{\nu \in t \setminus \{\mu\}} [\nu, A_\nu \xi] + \sum_{\lambda \in \mu\Lambda^{e_i}} [\lambda, A_\lambda \xi]
    = \sum_{\nu \in t *_\mu {\vee_{s(\mu), i}}} [\nu, A_\nu \xi].
\]
So by definition of the set $\cT$ of trees of $\Lambda$ and an induction, we obtain the tree identity
\[
[v,\xi] = \sum_{\mu \in t} [\mu, A_\mu \xi]\qquad\text{ for any tree $t$ of $\Lambda$ at $v$.}
\]
\end{remark}

\begin{definition}\label{dfn:Lspace}
We define $L$ to be the Hilbert space completion of $L^0$ with respect to the inner product of Proposition~\ref{prp:L inner prod}.
\end{definition}

Note that $L$ is equal to the closed linear span of equivalence classes $[\lambda,\xi]$ where $\lambda \in \Lambda$ and $\xi\in H_{s(\lambda)}$.

\subsubsection{A second construction of \texorpdfstring{$L$}{L}.} \label{subsec:scond-constr-Hilb}
We now give a second description of $L$ using a collection of Hilbert spaces indexed by \emph{trees}.
Fix a vertex $v\in\Lambda^0$ and consider $v\cT$ the set of trees of $\Lambda$ at $v$.

For each $t \in v\cT$, let $H_t$ denote the set $\{t\} \times \big(\bigoplus_{\mu \in t} H_\mu\big)$,
regarded as a copy of the Hilbert space $\bigoplus_{\mu \in t} H_{s(\mu)}$.
So
\[
H_t = \{(t,\xi) \mid \xi : t \to H \text{ satisfies } \xi_\mu \in H_{s(\mu)}\text{ for all $\mu \in t$}\}
\]
and
\[\langle (t, \xi), (t, \eta)\rangle = \sum_{\mu \in t} \langle \xi_\mu, \eta_\mu\rangle_H.\]

Suppose that $\mu \in t$ and $s(\mu) \Lambda^{e_i}$ is nonempty, and let $t' := t *_\mu {\vee_{s(\mu), i}}$, the tree obtained
by attaching the caret $s(\mu)\Lambda^{e_i}$ to $t$ at $s(\mu)$. Define
\[
\tilde{A}_{\mu, i} : \bigoplus_{\mu \in t} H_{s(\mu)} \to \bigoplus_{\nu \in t'} H_{s(\nu)}\quad\text{ by }\quad
    (\tilde{A}_{\mu, i} \xi)_\lambda
        = \begin{cases}
            \xi_\lambda &\text{ if $\lambda \in t \setminus \{\mu\}$}\\
            A_\alpha \xi_\mu &\text{ if $\lambda = \mu\alpha \in t' \setminus t$,}
        \end{cases}
\]
and let $\psi_{t, t'} \colon H_t \to H_{t'}$ be the map defined by applying $\tilde{A}_{\mu,i}$ in the second coordinate.
Lemma~\ref{lem:isometries} implies that $\psi_{t, t'}$ is an isometry.

An induction shows that given a sequence of trees
\[
t_0, t_1, t_2, \dots, t_n,
\]
and sequence of paths $\mu_j \in t_j, j < n$ and of indices $i_0, \dots, i_{n-1} \in \{1, \dots, k\}$ such
that $t_{j+1} = t_j *_{\mu_j} {\vee_{s(\mu_j), i_j}}$ for each $j < n$, the map
\[
\psi_{t_{n-1}, t_n} \circ \psi_{t_{n-2}, t_{n-1}} \circ \cdots \circ \psi_{t_0, t_1}
\]
is given by $(t_0, \xi) \mapsto (t_n, (A_\tau \xi_\mu)_{\mu \in t, \mu\tau \in t'})$. Thus,
we obtain a family of isometries $\psi_{t, t'}$, $t \le t' \in v\cT$ satisfying
$\psi_{t', t''} \circ \psi_{t, t'} = \psi_{t, t''}$. Hence
\[
(H_t, \psi_{t, t'})_{t \in v\cT}
\]
is a directed system of Hilbert spaces with connecting isometries, and we can form the
direct-limit Hilbert space
\[
\wh H_v = \varinjlim_{t \in v\cT} (H_t, \psi_{t, t'}).
\]
We write
\[
\wh H_v^0 := \bigcup_{t \in v\cT} \psi_t(H_t)
\]
for the canonical dense subspace.

Since the $\psi_{t,t'}$ are isometric, the canonical inclusions $\psi_t \colon H_t \to \wh H_v$
are also isometric, and the increasing union $\bigcup_{t \in v\cT} \psi_t(H_t)$ is dense in $\wh H_v$.
We have $\psi_t(t, \xi) = \psi_{t'}(t', \eta)$ if $\psi_{t, t''}(t, \xi) = \psi_{t',t''}(t', \eta)$ for some
(equivalently any) $t''$ satisfying $t, t' \le t''$. Hence the equivalence relation $\sim_\psi$ on
$\bigsqcup_{t \in v\cT} H_t$ given by $(t, \xi) \sim (t', \eta)$ if and only if
$\psi_t(\xi) = \psi_{t'}(\eta)$ is the equivalence relation generated by the equivalences
$(t, \xi) \sim (t *_\mu {\vee_{s(\mu), i}}, \tilde{A}_{\mu, i}\xi)$. So we can identify
$\wh H_v$ with $\big(\bigsqcup_{t \in v\cT} H_t\big)/{\sim_\psi}$, and we write
\[
[t, \xi] := \psi_t(t,\xi) \in \wh H_v.
\]

We define
\[
\wh H := \bigoplus_{v \in \Lambda^0} \wh H_v,\quad\text{ and }\quad
\wh H^0 := \bigoplus_{v \in \Lambda^0} \wh H_v^0.
\]

For $t \in v\cT$, and $\xi,\eta \in H_t$, we have
\[
\langle [t, \xi], [t, \eta]\rangle_{\wh H_v}
    = \langle (t, \xi), (t, \eta)\rangle
    = \sum_{\mu \in t} \langle \xi, \eta\rangle_H
    = \Big\langle \sum_{\mu \in t} [\mu, \xi_\mu], \sum_{\mu \in t} [\mu, \eta_\mu]\Big\rangle_L,
\]
so there is an isometric linear map $\iota_t : H_t \to L$ given by
$\iota_t([t, \xi]) = \sum_{\mu \in t} [\mu, \xi_\mu]$.  Remark~\ref{rmk:tree identity} implies
that $\iota_{t'} \circ \psi_{t, t'} = \iota_t$ for all $t$, so the $\iota_t$ assemble into
an isometry $\iota \colon \wh H \to L$. For $\mu \in \Lambda$ and $\xi \in H_{s(\mu)}$, let
$\xi 1_\mu \in \bigoplus_{\mu \in t} H_{s(\mu)}$ be the vector such that
$(\xi 1_\mu)_\nu = \delta_{\mu,\nu} \xi$. Then
\[
    [\mu, \xi] = \iota([t_{r(\mu)}^{d(\mu)}, \xi 1_\mu]).
\]
So $\iota$ has dense range, and is therefore a unitary isomorphism $\iota : \wh H \to L$.

\begin{remark}
We visualise a vector $[t, \xi] \in \wh H^0_v$ as the tree $t$ drawn as in the preceding section, with each leaf $\mu$ decorated by the vector $\xi_\mu$.
The caret relation is then represented visually by asserting that the tree $t$ with leaves decorated by $\xi$, and the
tree $t *_\mu {\vee_{s(\mu), i}}$ with leaved decorated by $\tilde{A}_{\mu, i} \xi$ represent the same vector. For example, if $t = \{a, b, c\}$ is a tree at $v$
and $\vee_{s(c), i} = \{d, e\}$ is a basic caret at $s(c)$, and if $t' = t *_e {\vee_{s(c), i}}$, then in the diagram below, the decorated tree on the left represents $[t,\xi]$ and
the one on the right represents $[\psi_{t, t'}(t, \xi)]$; the diagram in the middle indicates the transition from one to the other.
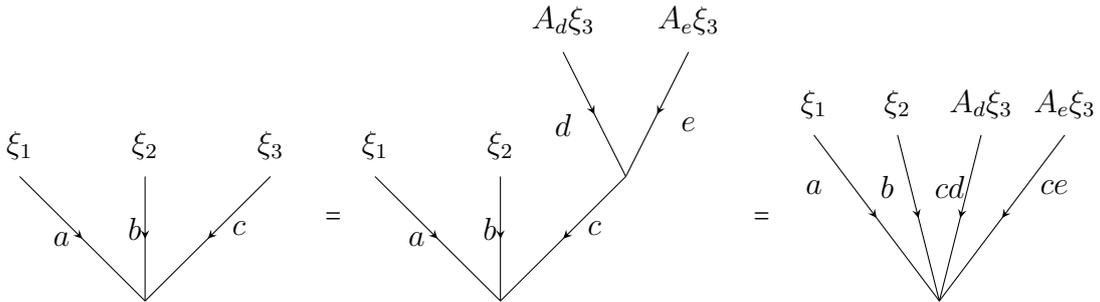
\[\begin{tikzpicture}[scale=0.55, decoration={markings, mark=at position 0.5 with {\arrow{stealth}}}]
	\begin{pgfonlayer}{nodelayer}
		\node [style=none] (0) at (0, 0) {};
		\node [style=none] (1) at (-3, 3) {};
		\node [style=none] (2) at (0, 3) {};
		\node [style=none] (3) at (3, 3) {};
		\node [style=none] (4) at (-2, 1.5) {};
		\node [style=none] (5) at (-2, 1.5) {$a$};
		\node [style=none] (6) at (-0.25, 1.75) {$b$};
		\node [style=none] (7) at (2.25, 1.75) {};
		\node [style=none] (8) at (2.25, 1.75) {$c$};
		\node [style=none] (9) at (-3, 3.75) {$\xi_1$};
		\node [style=none] (10) at (0, 3.75) {$\xi_2$};
		\node [style=none] (11) at (3, 3.75) {$\xi_3$};
%		\node [style=none] (12) at (-4.5, 1) {$(t,\xi)=$};
		\node [style=none] (14) at (4.5, 2) {$=$};
		\node [style=none] (15) at (4.5, 1) {};
		\node [style=none] (16) at (4.5, 1) {};
		\node [style=none] (17) at (4.5, 1) {};
		\node [style=none] (18) at (4.5, 1) {};
		\node [style=none] (19) at (8.5, 0) {};
		\node [style=none] (20) at (5.5, 3) {};
		\node [style=none] (21) at (8.5, 3) {};
		\node [style=none] (22) at (11.5, 3) {};
		\node [style=none] (23) at (6.5, 1.5) {};
		\node [style=none] (24) at (6.5, 1.5) {$a$};
		\node [style=none] (25) at (8.25, 1.75) {$b$};
		\node [style=none] (26) at (10.75, 1.75) {};
		\node [style=none] (27) at (10.75, 1.75) {$c$};
		\node [style=none] (28) at (5.5, 3.75) {$\xi_1$};
		\node [style=none] (29) at (8.5, 3.75) {$\xi_2$};
		\node [style=none] (30) at (10, 6) {};
		\node [style=none] (31) at (13, 6) {};
		\node [style=none] (33) at (10, 4.25) {$d$};
		\node [style=none] (34) at (13, 4.25) {$e$};
		\node [style=none] (35) at (10, 6.75) {$A_d\xi_3$};
		\node [style=none] (36) at (13, 6.75) {$A_e\xi_3$};
		\node [style=none] (37) at (14.75, 2) {$=$};
		\node [style=none] (38) at (14.75, 1) {};
		\node [style=none] (39) at (14.75, 1) {};
		\node [style=none] (40) at (14.75, 1) {};
		\node [style=none] (41) at (14.75, 1) {};
		\node [style=none] (42) at (19, 0) {};
		\node [style=none] (43) at (16, 4) {};
		\node [style=none] (44) at (18, 4) {};
		\node [style=none] (47) at (16, 2.75) {$a$};
		\node [style=none] (48) at (17.75, 2.75) {$b$};
		\node [style=none] (51) at (16, 4.75) {$\xi_1$};
		\node [style=none] (52) at (18, 4.75) {$\xi_2$};
		\node [style=none] (53) at (20, 4) {};
		\node [style=none] (54) at (22, 4) {};
		\node [style=none] (55) at (21.75, 2.75) {$ce$};
		\node [style=none] (56) at (19.25, 2.75) {$cd$};
		\node [style=none] (57) at (19.25, 2.75) {};
		\node [style=none] (58) at (20, 4.75) {$A_d\xi_3$};
		\node [style=none] (59) at (22, 4.75) {$A_e\xi_3$};
		\node [style=none] (60) at (6.75, 7.5) {};
	\end{pgfonlayer}
	\begin{pgfonlayer}{edgelayer}
		\draw[postaction={decorate}] (1.center) to (0.center);
		\draw[postaction={decorate}] (2.center) to (0.center);
		\draw[postaction={decorate}] (3.center) to (0.center);
		\draw[postaction={decorate}] (20.center) to (19.center);
		\draw[postaction={decorate}] (21.center) to (19.center);
		\draw[postaction={decorate}] (22.center) to (19.center);
		\draw[postaction={decorate}] (30.center) to (22.center);
		\draw[postaction={decorate}] (31.center) to (22.center);
		\draw[postaction={decorate}] (43.center) to (42.center);
		\draw[postaction={decorate}] (44.center) to (42.center);
		\draw[postaction={decorate}] (53.center) to (42.center);
		\draw[postaction={decorate}] (54.center) to (42.center);
	\end{pgfonlayer}
\end{tikzpicture}\]
\end{remark}

\subsubsection{Decompositions of \texorpdfstring{$L$}{L}}
Recall that $\le$ is defined on $\Lambda$ by $\lambda \le \mu$ if and only if $\mu \in \lambda\Lambda$. For $\lambda \in \Lambda$, we define
\[L_\lambda := \overline{\operatorname{span}}\{[\mu,\xi] : \lambda \le \mu, \xi \in H_{s(\mu)}\} \le L.\]
By definition of the inner-product on $L$, we have $L_\lambda \perp L_\mu$ whenever $\lambda \Lambda \cap \mu\Lambda = \emptyset$. In particular,
\[
L = \bigoplus_{\lambda \in \Lambda^{\le m}} L_\lambda\quad\text{ for any $m \in \N^k$.}
\]
Taking $m = 0$ we obtain the decomposition $L = \bigoplus_{v \in \Lambda^0} L_v$; observe that each $L_v = \iota(\wh H_v)$ in the notation of the preceding section.

For any $v \in \Lambda^0$ and $i \le k$ such that $v\Lambda^{e_i} \not= \emptyset$, we have
\[
L_v = \bigoplus_{e \in v\Lambda^{e_i}} L_e.
\]
More generally, for each $\lambda \in \Lambda$ and each $m \in \N^k$, we have
\[
L_\lambda = \bigoplus_{\mu \in s(\lambda)\Lambda^{\le m}} L_{\lambda\mu}.
\]

In particular, if $\vee_{v,i}$ and $\vee_{v,j}$ are nontrivial basic carets at $v$, then
\begin{equation}\label{eq:Lv decomp}
    L_v = \bigoplus_{e \in v\Lambda^{e_i}} L_e
        = \bigoplus_{f \in v\Lambda^{e_j}} L_f
        = \bigoplus_{\lambda \in v\Lambda^{\le e_i + e_j}} L_\lambda,
\end{equation}
three (in general) different decompositions of $L_v$, the third of which is a refinement of each of the other two.

\subsection{Action of \texorpdfstring{$\cOL$}{O(Lambda)} on \texorpdfstring{$L$}{L}}\label{subsec:action of OL on L}
We now define the action $\cOL$ on $L$ denoted $x_\lambda\mapsto X_\lambda$. Loosely speaking, $X_\lambda$ acts like a sort of forward shift by $\lambda$: it shifts the copy $[\mu, \xi]$ of a vector $\xi \in H_{s(\mu)}$ indexed by $\mu$
to the copy $[\lambda\mu, \xi]$ indexed by $\lambda\mu$ whenever this is allowed. Consequently, $X_\lambda^*$ acts by removing $\lambda$.

\begin{lemma}\label{lem:Xs on L}
For each $\mu \in \Lambda$, there is a partial isometry $X_\mu \in B(L)$ such that
\[
X_\mu[\lambda,\xi] =
\begin{cases}
	[\mu\lambda,\xi], &\text{ if } r(\lambda)=s(\mu) \\
	0, &\text{ otherwise.}
\end{cases}
\qquad\text{ and }\qquad
X_\mu^*[\lambda, \xi] =
\begin{cases}
    [\lambda', \xi] &\text{ if $\lambda = \mu\lambda'$} \\
    0 &\text{ otherwise.}
\end{cases}
\]
We have $X_\mu^* X_\mu L = L_{s(\mu)}$ and $X_\mu X_\mu^* L = L_\mu$.
\end{lemma}
\begin{proof}
Let $[\cdot \mid \cdot]$ be the sesquilinear form~\eqref{eq:formdef} on $L^{00}$. Since the $(\lambda,\xi)$ are linearly independent in $L^{00}$ there is a linear map $\widetilde{X}_\mu$ on $L^{00}$ such that $\widetilde{X_\mu}(\lambda,\xi) =
\delta_{r(\lambda),s(\mu)} (\mu\lambda, \xi)$. Fix $(\lambda, \xi), (\nu, \eta) \in L^{00}$ with $r(\lambda) = r(\nu) = s(\mu)$. Let $m = (d(\lambda) \vee d(\nu)) - d(\lambda)$ and $n = (d(\lambda) \vee d(\nu)) - d(\nu)$.
By the tree identity of Remark~\ref{rmk:tree identity}, we have
\begin{align*}
\big[\widetilde{X}_\mu(\lambda, \xi) \mathbin{\big|} \widetilde{X}_\mu(\nu, \eta)\big]
    &= \sum_{\substack{\tau \in s(\lambda)\Lambda^{\le m}\\ \rho\in s(\nu)\Lambda^{\le n}}} \big[(\mu\lambda\tau, A_\tau \xi) \mathbin{\big|} (\mu\nu\rho, A_\rho \xi)\big]\\
    &= \sum_{\substack{\tau \in s(\lambda)\Lambda^{\le m}\\ \rho\in s(\nu)\Lambda^{\le n}}} \big[(\lambda\tau, A_\tau \xi) \mathbin{\big|} (\nu\rho, A_\rho \xi)\big]
    = \big[(\lambda, \xi) \mathbin{\big|} (\nu, \eta)\big].
\end{align*}
Hence $\widetilde{X}_\mu$ annihilates the kernel $N$ of $[\cdot \mid \cdot]$, and descends to a partial isometry $X_\mu$, and standard calculations with basis vectors show that its adjoint is as described. We then have
$X_\mu^* X_\mu [\lambda, \xi] = \delta_{r(\lambda), s(\mu)} X_\mu^* [\mu\lambda, \xi] = \delta_{r(\lambda), s(\mu)} [\lambda, \xi]$, so $X_\mu^* X_\mu$ is the projection onto $L_{s(\mu)}$. Clearly the range of $X_\mu$
is contained in $L_\mu$, and since each spanning element $[\mu\lambda, \xi]$ of $L_\mu$ can be written as $[\mu\lambda, \xi] = X_\mu [\lambda, \xi]$ we obtain the reverse containment, so $X_\mu X_\mu^*$ is the projection
onto $L_\mu$ as claimed.
\end{proof}

\begin{proposition}\label{prop:OL acting on L}
The map $\lambda \mapsto X_\lambda$ satisfies the relations of Definition~\ref{def:O}. There is an action of $\cOL$ on $L$ such that $x_\mu \cdot [\lambda, \xi] = \delta_{s(\mu), r(\lambda)} [\mu\lambda, \xi]$ for all
$\mu,\lambda \in \Lambda$ and $\xi \in H_{s(\lambda)}$.
\end{proposition}
\begin{proof}
The second statement will follow from the universal property of $\cOL$ once we establish the first statement. We have already seen that each $X_v$ is the orthogonal projection onto $L_v$, and since the $L_v$ are mutually orthogonal
this gives~(1). That $X_\lambda^* X_\lambda = X_{s(\lambda)}$ is the first part of the final statement of Lemma~\ref{lem:Xs on L}, giving~(2). Fix $i \le k$ such that $v\Lambda^{e_i} \not= \emptyset$. By the final statement of
Lemma~\ref{lem:Xs on L} and~\eqref{eq:Lv decomp},
\[
\sum_{\lambda \in v\Lambda^{e_i}} X_\lambda X^*_\lambda
    = \sum_{\lambda \in v\Lambda^{e_i}} \operatorname{proj}_{L_\lambda}
    = \operatorname{proj}_{\bigoplus_{\lambda \in v\Lambda^{e_i}}} L_\lambda
    = \operatorname{proj}_{L_v} = X_v.
\]
This proves~(3). Finally, if $s(\lambda) = r(\mu)$, then
\[
    X_\lambda X_\mu [\nu, \xi] = \delta_{s(\mu), r(\nu)} X_\lambda [\mu\nu, \xi] = \delta_{s(\lambda\mu), d(\nu)} [\lambda\mu\nu, \xi] = X_{\lambda\circ\mu}[\nu,\xi],
\]
which proves~(4) and completes the proof.
\end{proof}

\begin{example}
	Let $\Lambda$ be as described in Example \ref{ex:tree} and consider the above representation.	Below demonstrates the action of $\mu_1$ and $\mu_1^*$.
	\begin{center}
		\pgfdeclarelayer{nodelayer}
		\pgfdeclarelayer{edgelayer}
		\pgfsetlayers{nodelayer,edgelayer,main}
		\begin{tikzpicture}[scale=0.5]
			\begin{pgfonlayer}{nodelayer}
				\node [] (0) at (0, 0) {};
				\node [] (1) at (0, 1) {};
				\node [] (2) at (-1, 0.5) {$\mu_1^*~ \cdot$};
				\node [] (3) at (1, 0.5) {};
				\node [] (4) at (1, 0.5) {$=$};
				\node [] (5) at (2.25, 0.5) {$\mu_1^*~ \cdot$};
				\node [] (6) at (4, 0) {};
				\node [] (7) at (4, 1) {};
				\node [] (9) at (6, 0.5) {$=$};
				\node [] (10) at (3, 2) {};
				\node [] (11) at (5, 2) {};
				\node [] (12) at (0, 1.5) {};
				\node [] (13) at (0, 1.5) {$\xi$};
				\node [] (14) at (3, 2.5) {$A_{\mu_1}\xi$};
				\node [] (15) at (5, 2.5) {$A_{\mu_2}\xi$};
				\node [] (16) at (9.25, 0) {};
				\node [] (17) at (8.25, 1) {};
				\node [] (18) at (10.25, 1) {};
				\node [] (19) at (8.25, 2) {};
				\node [] (20) at (10.25, 2) {};
				\node [] (21) at (8.25, 2.5) {$A_{\mu_2}\xi$};
				\node [] (22) at (10.25, 2.5) {$A_{\mu_1}\xi$};
				\node [] (23) at (7.25, 0.5) {$\mu_1^*~ \cdot$};
				\node [] (24) at (11.25, 0.5) {$=$};
				\node [] (26) at (12.5, 1) {};
				\node [] (27) at (12.5, 1.5) {$A_{\mu_2}\xi$};
				\node [] (28) at (12.5, 0) {};
				\node [] (29) at (-6, 0) {};
				\node [] (30) at (-6, 1) {};
				\node [] (31) at (-7, 0.5) {$\mu_1~ \cdot$};
				\node [] (32) at (-6, 1.5) {};
				\node [] (33) at (-6, 1.5) {$\xi$};
				\node [] (34) at (-5, 0.5) {$=$};
				\node [] (35) at (-3, 0) {};
				\node [] (36) at (-4, 1) {};
				\node [] (37) at (-4, 2) {};
				\node [] (38) at (-4, 2.5) {$\xi$};
				\node [] (39) at (-2.5, 0) {,};
			\end{pgfonlayer}
			\begin{pgfonlayer}{edgelayer}
				\draw [red] (0.center) to (1.center);
				\draw [red] (6.center) to (7.center);
				\draw [blue] (7.center) to (11.center);
				\draw [blue] (10.center) to (7.center);
				\draw [blue] (16.center) to (17.center);
				\draw [blue] (16.center) to (18.center);
				\draw [red] (19.center) to (17.center);
				\draw [red] (20.center) to (18.center);
				\draw [red] (26.center) to (28.center);
				\draw [red] (29.center) to (30.center);
				\draw [blue] (35.center) to (36.center);
				\draw [red] (36.center) to (37.center);
			\end{pgfonlayer}
		\end{tikzpicture}
	\end{center}	
\end{example}

\subsection{Functoriality}

We collect here some facts concerning our construction $H\mapsto L$ that we express using the language of categories. In particular, the construction is functorial, which is surprising since $\Mod(\cAL)$ has large hom-spaces.
Here $H=\oplus_{v\in\Lambda^0}H_v$ denotes an $\cAL$-module and $L$ denotes the representation of $\cOL$ of above whose carrier Hilbert space is the closed linear span of classes $[\lambda,\eta]$ with $\lambda\in\Lambda$ and $\eta\in H_{s(\lambda)}$.

\begin{proposition}\label{prop:functor}
The following assertions are true.
\begin{enumerate}
\item The construction $(H, (A_\lambda)_{\lambda \in \Lambda})\mapsto (L, (X_\lambda)_{\lambda \in \Lambda})$ of Sections \ref{subsec:first-constr-Hilb}~and~\ref{subsec:action of OL on L} defines a functor $\Pi:\Mod(\cAL)\to\Rep(\cOL)$
from the category of $\cAL$-modules to the representation category of $\cOL$.
    In details: there exists a unique covariant functor
    $$\Pi:\Mod(\cAL)\to \Rep(\cOL)$$
    sending $H$ to the space $\Pi(H) := L$ as in Section~\ref{subsec:first-constr-Hilb} and sending an arrow $\theta:H\to H'$ of $\Mod(\cAL)$ to the unique arrow $\Pi(\theta):\Pi(H)\to\Pi(H')$ such that
    $$\Pi(\theta)([\lambda,\xi]) = [\lambda,\theta(\xi)]\quad\text{ for all $\lambda \in \Lambda$ and $\xi \in H_{s(\lambda)}$.}$$
    This functor $\Pi$ restricts to a functor, also denoted $\Pi$, from $\Rep(\cPL)$ to $\Rep(\cOL)$. We call $\Pi$ the \emph{Pythagorean functor} (or \emph{P-functor}) associated to the $k$-graph $\Lambda$.
\item The formula $a_\lambda\mapsto s_\lambda^*$ for $\lambda$ a vertex or an edge extends uniquely to a surjective morphism of C$^*$-algebras $q:\cPL\to \cOL$. This defines a forgetful
    functor $U:\Rep(\cOL)\to\Rep(\cPL),\ \pi\mapsto \pi\circ q.$ Composing $U$ with the inclusion $\Rep(\cPL)\to\Mod(\cAL)$ yields a functor $\Rep(\cOL) \to \Mod(\cAL)$, also denoted $U$.
\item Given $(H, A)$ in $\Mod(\cAL)$, let $(L, X) = \Pi(H, A)$. There is a unique linear isometry $J_H : H \to L$ such that for each $v \in \Lambda^0$ we have $J_H(\xi) = [v, \xi]$. We have
    $$X_\lambda^* \cdot J_H(\xi) = J_H\circ A_\lambda(\xi).$$
    This construction is natural in $H$: the collection of $J_H$ indexed by $H\in\ob \Mod(\cAL)$ defines a natural transformation from the identity functor of $\Mod(\cAL)$ to the functor $U\circ \Pi$,
    and restricts to a natural transformation from the identity functor of $\Rep(\cPL)$ to $U \circ \Pi$.
\item Consider a representation $(L, X)$ of $\cOL$ and let $U(L,X)$ be the $\cAL$-module obtained from~(2). Let $(\wh L, \wh X) = \Pi(U(L, X))$ be the corresponding representation of $\cOL$
    from~(1). Let $\wh J_L = J_{\wh L}$ be the isometry implementing the natural transformation $J$ of~(3). Then $\wh J_L$ is a unitary isomorphism $L \to \wh L$, and intertwines not only the $\cAL$-module structure on $\wh L$, but also the actions of $\cOL$ on $\wh L$.
    This construction is natural in $L$: the collection of $\wh J_L$ indexed by $L\in\ob\Rep(\cOL)$ defines a natural \emph{isomorphism} from the identity functor of $\Rep(\cOL)$ to $\Pi\circ U$.
\item The functor $\Pi:\Mod(\cAL)\to\Rep(\cOL)$ is essentially surjective: for each representation $L$ of $\cOL$ there exists an $\cAL$-module $H$ satisfying $L\simeq \Pi(H)$.
\end{enumerate}
\end{proposition}
\begin{proof}
Proof of (1).
Consider $\cAL$-modules $(H, A)$ and $(H', A')$. Let $(L, X)$ and $(L', X')$ be the representations of $\cOL$ of Proposition~\ref{prop:OL acting on L}.
Fix an intertwiner $\theta:H\to H'$, so that $A_\lambda'\circ \theta=\theta\circ A_\lambda$ for all $\lambda\in\Lambda$.
Take $\lambda\in \Lambda$ and $\xi\in H_{s(\lambda)}$. For any tree of $\Lambda$ ar $r(\lambda)$, the tree identity of Remark~\ref{rmk:tree identity} gives
\[
    [\lambda,\xi] = \sum_{\mu \in t} [\lambda\mu, A_\mu \xi].
\]
The formula $[\lambda,\xi] \mapsto [\lambda,\theta(\xi)]$ is well-defined because $\theta$ is an intertwiner, and
defines a linear map $\Pi(\theta)$ from $L^0 \subset L$ to $L'$.
To show that $\Pi(\theta)$ is bounded (with in fact same norm as $\theta$) we use that $\Pi(\theta)$ maps
$L_\lambda$ to $L'_\lambda$ for each $\lambda$.
Fix $\xi$ in $L^0$, say $\xi = \sum_{\mu \in F} [\mu, \xi_\mu]$ for some
finite $F \subset \Lambda$. Put $m = \bigvee_{\mu \in F} d(\mu)$. Then the tree identity gives
\[
\xi = \sum_{\mu \in F} \sum_{\alpha \in s(\mu)\Lambda^{\le m - d(\mu)}} [\mu\alpha, A_\alpha \xi_\mu],
\]
so we can assume that $F \subset \Lambda^{\le m}$. Now the spaces $L_\mu$ indexed by $\mu \in F$ are mutually orthogonal.
We have
\begin{align*}
\|\Pi(\theta)(\xi)\|^2& = \| \sum_{\lambda \in F} [\lambda, \theta(\xi_\lambda)]\|^2
 = \sum_{\lambda \in F} \| \theta(\xi_\lambda)\|^2 \leq \|\theta\|^2 \sum_{\lambda\in F} \|\xi_\lambda\|^2
 = \|\theta\|^2 \cdot \|\xi\|^2.
\end{align*}
Hence, $\Pi(\theta)$ is bounded and extends uniquely into a bounded linear map from $L$ to $L'$.
The rest of the proposition follows easily (namely $\Pi(\theta)$ intertwines the $\cOL$-representations, $\Pi(\theta)\circ\Pi(\alpha)=\Pi(\theta\circ\alpha)$ and $\Pi(\id)$ is the identity).

Proof of (2).
Consider $\{a_\lambda : \lambda \in \Lambda^0 \cup \Lambda^1\} \subset \cPL$. We have seen that these elements generate $\cPL$.
Their adjoints $x_\lambda := a_\lambda^*$ satisfy all the axioms of Definition~\ref{def:O} except the quasi-isometry relation.
The universal property of $\cPL$ gives a unique continuous *-algebra morphism $q:\cPL\to\cOL$ sending $a_\lambda$ to $x_\lambda^*$.
Since $\{x_\lambda^* : \lambda \in \Lambda^0 \cup \Lambda^1\}$ generates $\cOL$, this $q$ is surjective.
Fix a representation $(L, X)$ of $\cOL$. For $a_\lambda\in \cPL$, define
$$a_\lambda\cdot \xi:= X_\lambda^*(\xi).$$
This defines an action of $\cPL$ on $L$.
This is the representation of $\cPL$ obtained by precomposing by $q$.
Functoriality follows tautologically.

Proof of (3).
Consider $H\in \ob \Mod(\cAL)$ with associated $\Pi(H)\in \ob \Rep(\cOL)$.
Write $U\Pi(H)$ for the associated $\cAL$-module. This means that $a_\lambda\cdot [\lambda,\xi]:= X_\lambda^*[\lambda,\xi]$ where $x_\lambda\mapsto X_\lambda$ is the representation $\Pi(H)$ of $\cOL$.
Define
$$J_H:H\to U\Pi(H)\quad\text{ by }\quad H_v\ni \xi\mapsto [v,\xi]\in U\Pi(H).$$
This $J_H$ is clearly isometric on each $H_v$. Moreover, $J_H(H_v)\subset U\Pi(H)_v$, so the $J_H(H_v)$ are pairwise orthogonal. Hence, $J_H$ is an isometry.
To see that it intertwines the $\cAL$-module structures, fix $v\in\Lambda^0,\xi\in H_v$ and $\lambda\in\Lambda$. Then for any tree $t$ at $v$,
$$X_\lambda^*J_H(\xi)=X_\lambda^*[v,\xi]=X_\lambda^* \sum_{\mu\in t}[\mu,A_\mu\xi].$$
If $r(\lambda) \neq v$, then necessarily $X_\lambda^*J_H(\xi)= 0 = J_H(A_\lambda\xi)$.
So suppose that $r(\lambda) = v$. Let $t = t^{d(\lambda)}_v = v\Lambda^{\le d(\lambda)}$. Then $\lambda \in t$. We have
$X_\lambda^*[\lambda, A_{\lambda}\xi] = [v, A_{\lambda}\xi]$. Since $X_\lambda^*$ is, by definition, zero on $U\Pi(H)_\lambda^\perp$,
we see that
$$X_\lambda^*[v,\xi]=[v,A_\lambda\xi]$$
and so $J_H$ intertwines the $\cAL$-module structures.
Naturality in $H$ follows easily.
Indeed, take an intertwiner $\theta:H\to H'$ and observe that if $\xi\in H_v$, then
$$U\Pi(\theta)\circ J_H(\xi)=U\Pi(\theta)[v,\xi]=[v,\theta(\xi)]=J_{H'}\circ \theta(\xi).$$
This proves the second statement: the collection of $J_H$ indexed by the object of $\Mod(\cAL)$ realises a natural transformation from the identity functor of $\Mod(\cAL)$ to the endofunctor $U\Pi$ of $\Mod(\cAL)$.

Proof of (4). Consider a representation $(L, X)$ of $\cOL$ and let $U(L, X) = (H, A)$ be the associated $\cAL$-module.
Then $H = L$ and for $e \in \Lambda^1$ we have $A_e = X^*_e$.
Applying the P-functor $\Pi$ to this $\cAL$-module, we obtain a representation $\Pi U(L) = (L', X')$ of $\cOL$.
Consider the linear map
$$\widehat J_L : L\to L', \quad L_v\ni \xi\mapsto [v,\xi]\in L'.$$
Recall that $L_v$ is simply the range of the projection $X_v$ or the domain of $X_\lambda$ for any path $\lambda$ with source $v$.
The map $\widehat J_L$ is exactly the function defined in the proof of~(3).
If we forget the $\cOL$-representation structure and remember only the $\cAL$-module structure, then $\widehat J_L$ is an arrow in $\Mod(\cAL)$ (from $UL$ to $U\Pi U(L)$).
We write $\widehat J_L$ rather than $J_{UL}$ to emphasise that $\widehat J_L$ is, at this point, only considered as a linear map between Hilbert spaces
rather than the arrow $J_{UL}$ previously defined in the category $\Mod(\cAL)$.
Intertwining the underlying $\cAL$-module structures means that $\widehat J_L\circ X_\lambda^* = (X'_\lambda)^*\circ\widehat J_L$ for all $\lambda\in \Lambda$.
We show that $\widehat J_L$ is a unitary transformation (hence it is surjective) and intertwines each $X_\lambda$ with $X'_\lambda$.
Since $L$ is a representation of $\cOL$, for each $i \le k$ and $v \in \Lambda^0$, caret map is either $1_{L_v}$ (if $v\Lambda^{e_i} = \emptyset)$, or the map
$$L_v\to \oplus_{\lambda\in v\Lambda^{e_i}} L_{s(\lambda)}, \ \xi\mapsto \sum_\lambda X_\lambda^*\xi.$$
Both are unitary transformations: the whole point of the process $H\mapsto \Pi(H)$ is that it dilates the isometric caret maps into unitary transformations.
Since every tree is generated by basic carets, we deduce that for every tree at $v$, we have
$$B_t:=\oplus_{\mu \in t} L_\mu \to L_v,\quad (\eta_\mu)_{\mu\in t}\mapsto \sum_{\mu\in t} X_\mu \eta_\mu$$
is a unitary transformation.
Fix $v \in \Lambda^0$, $\lambda \in v\Lambda$ and $\xi \in L_v$. As above, let $t = t^{d(\lambda)}_v = v\Lambda^{\le d(\lambda)}$,
so that $\lambda \in t$. We have
$$[\lambda,\xi] = [v, X_{\lambda} \xi] = \widehat J_L(X_\lambda\xi).$$
So $[\lambda,\xi] \in \widehat J_L L$, and thus $\wh J_L$ has dense image. Since it is isometric, it is therefore surjective,
and consequently a unitary transformation.
We have already seen that $\widehat J_L\circ X_\lambda^* = (X'_\lambda)^*\circ \widehat J_L$ for all $\lambda\in\Lambda$,
and taking adjoints and then multiplying on both sides by $\widehat J_L$ shows that
$\widehat J_L\circ X_\lambda = X'_\lambda \circ \widehat J_L$ for all $\lambda\in\Lambda$ as well.
In particular, $\widehat J_L$ is an isomorphism in the category $\Rep(\cOL)$.
It is immediate that $\widehat J_L$ is natural in $L$.

Proof of (5).
This is an immediate consequence of (4).
Indeed, take a representation $L$ of $\cOL$ and then take $H:=UL$. We have $\Pi(H)\simeq L$.
\end{proof}

\begin{remark}\label{rmk:cAL helps}
Proposition~\ref{prop:functor}(3) demonstrates the importance of introducing the non-self-adjoint algebra $\cAL$. The point is that since $\cOL$ is typically purely infinite, all of its representations are necessarily on infinite-dimensional spaces, and so there are no finite-dimensional reducing subspaces of $L$, or indeed of any representation of $\cOL$. However, the subalgebra $\cAL$ does not contain the adjoints of the generators of $\cOL$ and so it is possible---in fact, quite common---for an irreducible representation of $\cOL$ to contain many proper $\cAL$-submodules. In particular, the representation $L$ of $\cOL$ contains the $\cAL$-submodule $\overline{\operatorname{span}}\{[v,\xi] : v \in H, \xi \in H_v\}$, which is equivalent to the initial $\cAL$-module $H$. This means that whenever a representation $K$ of $\cOL$ has finite P-dimension (see Definition \ref{rem:finite-p-dim}), it contains a finite-dimensional $\cAL$-module. We exploit this in the proof of Proposition~\ref{prop:decomposition}. Moreover, $\Pi$ is functorial not just from $\Rep(\cPL)$ to $\Rep(\cOL)$ but from all of $\Mod(\cAL)$ to $\Rep(\cOL)$, so to obtain maximal information about arrows in $\Rep(\cOL)$, we must take into account all the intertwiners of $\cAL$-modules, and not just intertwiners of $\cPL$-modules (see Remark~\ref{rmk:wide subcategory}).
\end{remark}

\section{Representation theory}

Recall that $\cAL$ is the non-self-adjoint subalgebra of $\cPL$ generated by the generators $a_\lambda$ but not their adjoints.
In this section we introduce the concept of a complete module and its residual subspace.
We use these notions to overcome some technical problems inherent to the non-semi-simplicity of $\Mod(\cAL)$.
We then prove that the functor $\Pi:\Mod(\cAL)\to\Rep(\cOL)$ preserves irreducibility and is fully faithful when suitably restricted.

\subsection{Complete submodules and residual subspaces}

In what follows we will often say module rather than $\cAL$-module.
Consider a module $H$ and decompose the carrier Hilbert space as $H=H'\oplus Z$ where $H'\subset H$ is a submodule and where $Z\subset H$ is its orthogonal complement.

\begin{enumerate}
\item We say that $Z$ is a \emph{residual} subspace of $H$ if it does not contain any nonzero module. In that case we say that $H'$ is a \emph{weakly complete} submodule of $H$;
\item We say that $H'$ is a \emph{complete} submodule of $H$ if: for all $\varep>0,v\in\Lambda^0,\xi\in H_v$ there exists a tree $t$ at $v$ and a family $(\eta_\mu)_{\mu \in t}$ of vectors in $H'$ satisfying
$$\sum_{\mu\in t} \| A_\mu\xi -\eta_\mu\|^2 < \varep^2.$$
In other words, the $v\cT$-orbit of any $\xi \in H_v\setminus\{0\}$ comes arbitrarily close to $H'$.
\end{enumerate}
The module $H$ is called \emph{reducible} (resp.~\emph{decomposable}) if there exists a decomposition $H = H' \oplus Z$ as the above such that $Z\neq \{0\}$ (resp.~$Z$ contains a nonzero submodule).

\begin{proposition}\label{prop:complete}
Consider a module $H$, and let $\Pi(H)\in\ob\Rep(\cOL)$ be the associated representation.
Let $H'\subset H$ be a submodule.
The submodule $H'$ is complete if it satisfies one of the four equivalent properties (1)--(4):
\begin{enumerate}
\item The closed linear span $\Pi(H') := \overline{\operatorname{span}}\{[\lambda,\xi]:\ \lambda\in \Lambda\text{ and }\xi\in H'\} = \Pi(H)$;
\item The smallest subrepresentation of $\Pi(H)$ (as a representation of $\cOL$) that contains the space $J_H(H')$ is $\Pi(H)$;
\item For any $\xi\in H$ and $\varep>0$ there exists $n\in \N^k$ such that the distance between the tuple $(A_\mu\xi)_\mu$ indexed by the $\mu$ in $M:=\Lambda^{\leq n}$ and $(H')^M$ is smaller than $\varep$. In other words, there exists a tuple $(\eta_\mu)_\mu\in (H')^{M}$ so that $\sum_{\mu\in M} \| A_\mu\xi - \eta_\mu\|^2 <\varep^2.$
\item For any $v\in\Lambda^0$ and $\xi\in H_v$ and $\varep>0$, there exist a tree $t\in v\cT$ and vectors $(\eta_\mu)_{\mu \in t}$ in $H'$ such that $\sum_{\mu\in t}\| A_\mu\xi - \eta_\mu\|^2<\varep^2$.
\end{enumerate}
Moreover, completeness implies weak completeness:
\begin{enumerate}\setcounter{enumi}{4}
\item If $H'\subset H$ satisfies one of the equivalent properties of above, then the orthogonal complement $H\ominus H'$ is a residual subspace.
\end{enumerate}
\end{proposition}

We will prove in Proposition \ref{prop:weak-complete} that the notions of completeness and weak completeness coincide when $H$ is finite dimensional.

\begin{proof}[Proof of Proposition~\ref{prop:complete}]
\mbox{(1)\;$\iff$\;(2)}.
Identify $J_H(H)$ with $H$, and fix a subrepresentation $K$ of $H$ that contains $H'$.
If $\xi\in H'_v$ and $\lambda \in \Lambda v$, then $X_\lambda \xi =[\lambda,\xi]$.
Hence, $\Pi(H') = \overline{\operatorname{span}}\{[\lambda,\xi] : \lambda\in\Lambda\text{ and }\xi\in H'\}$.
Shows that \mbox{(1)\;$\iff$\;(2)}.

\mbox{(3)\;$\iff$\;(4)}.
Decompose $H = \bigoplus_{v \in \Lambda^0} H_v$ and use that $t^n_v = v\Lambda^{\leq n}$ is a tree by Lemma~\ref{lem:complete-tree}.

\mbox{(2)\;$\implies$\;(4)}.
Let $v\in\Lambda^0,\xi\in H,\varep>0$.
First we approximate $\xi$ inside $\Pi(H)$ by a linear combination of elements of $[\lambda,\eta_\lambda]$ with $\eta_\lambda\in H'$.
By density there exist a finite set $F$ of $\lambda\in v\Lambda$ and vectors $\eta_\lambda\in H'_v$ so that $[v,\xi]\in J_H(H)$ and $\sum_{\lambda\in F} [\lambda,\eta_\lambda]$ are $\varep$-close.
Second, we approximate $\xi$ using an element of the form $[t,\zeta]$ for a tree $t$ and $\zeta=(\zeta_\mu)_{\mu\in t}$ a family of vectors of $H'$ indexed by the leaves of $t$.
Let $n = \bigvee_{\mu \in F} d(\mu)$, so that $v\Lambda^{\leq n} \subset F\Lambda$.
By Lemma~\ref{lem:complete-tree}, for each $\lambda \in F$ the set $t_\lambda = t^{n - d(\lambda}_{s(\lambda)} = s(\lambda)\Lambda^{\leq n - d(\lambda)}$ is a tree.

For $\mu \in v\Lambda^{\le n}$ let $\zeta_\mu := \sum_{\lambda \in F, \mu = \lambda \alpha} A_\alpha \eta_\lambda$.
The tree identity of Remark~\ref{rmk:tree identity} gives
\[
\sum_{\lambda\in F} [\lambda,\eta_\lambda]
    = \sum_{\lambda \in F} \sum_{\alpha \in t_\lambda} [\lambda\alpha, A_\alpha \eta_\lambda]
    = \sum_{\mu \in t^n_v} [\nu, \zeta_\nu].
\]
So $[v,\xi]$ is $\varep$-close to $[t,\zeta]$. Since $H'$ is a module and each $\eta_\lambda\in H'$, each $\zeta_\ell \in H'$ as well.

Given a tree $t$ we write
$$A_t:H\to \oplus_{\mu \in t }H_{s(\mu)}$$
for the mapping
$$\xi\mapsto (A_\mu \xi)_{\mu \in t}.$$
So $A_t$ is an isometry from $H_{r(t)}$ to $\oplus_{\mu\in t} H_{s(\mu)}$.
The caret relation \eqref{eqn:caret} implies that
$$[v,\xi]=[t,A_t\xi] \text{ for every tree } t\in v\cT \text{ and } \xi\in H_v.$$
Hence $[t, A_t\xi] = [v, \xi]$ is $\varep$-close to $[t, \zeta]$. So
$$\sum_{\mu\in t} \| A_\mu\xi - \zeta_\mu\|^2 <\varep^2.$$

\mbox{(4)\;$\implies$\;(2)}.
Denote by $L'$ the subrepresentation of $\Pi(H)$ generated by $H'$.
Let us show that $L'=\Pi(H)$.
It is sufficient to show that each vector of the form $[\lambda,\xi]$ is at distance zero from $L'$ since the $[\lambda,\xi]$ form a total family of vectors of $\Pi(H)$ and since $L'$ is norm-closed.
By assumption, for all $\varep>0$ there exist $t\in v\cT$ and $\eta \in (H')^t$ such that
\[
    \Big\|\sum_{\mu\in t} (A_\mu\xi - \eta_\mu)\Big\|^2<\varep^2.
\]
Let $s = t^{d(\lambda)}_{r(\lambda)} = r(\lambda)\Lambda^{\le d(\lambda)}$, a tree containing $\lambda$.
Let $z = s *_\lambda t = (s \setminus \{\lambda\}) \cup \lambda t$.
Define $(\zeta_\nu)_{\nu \in z}$ by $\zeta_{\lambda\mu} = \eta_{\mu}$ for $\mu \in t$, and $\zeta_{\nu} = 0$ for $\nu \in s \setminus \{\lambda\}$.
Since each $\zeta_\nu \in H'$, we have $\sum_{\nu \in z} [\nu,\zeta_\nu] \in L'$. The tree identity of Remark~\ref{rmk:tree identity} shows that
$[\lambda, \xi] = \sum_{\mu \in t} [\lambda\mu, A_\mu \xi]$, so
\[
\Big\|[\lambda, \xi] - \sum_{\nu \in z} [\nu,\zeta_\nu]\Big\|
    =  \Big\|\sum_{\mu \in t} [\lambda\mu, A_\mu \xi] - [\lambda\mu, \eta_\mu]\Big\|
    = \Big\|\sum_{\mu\in t}(A_\mu\xi - \eta_\mu)\Big\| < \varep.
\]
Since $\varep$ was arbitrary, we deduce that $[\lambda,\xi]$ is at distance zero from $L'$.
So $\Pi(H)=L'$.

(1) implies (5). We prove the contrapositive.
Assume that $Z:=H\ominus H'$ is not residual.
This means that $Z$ contains a nontrivial submodule $M$.
Hence $J_H(M)$ is nonzero and orthogonal to the closed linear span of the $[\lambda,\xi]$ with $\xi\in H'$ inside $\Pi(H)$.
\end{proof}

\begin{example}
Here is an example showing that (5) does not imply (1) in general in the infinite dimensional case.
We take for $k$-graph the 1-graph with one vertex and two edges writing $A,B$ the two operators (hence we are in the classical case of the Pythagorean algebra and the Cuntz algebra).
Take separable Hilbert spaces $H_0,H_1$ with orthonormal basis $(\delta_n)_{n\geq 0}$ and $(\gamma_n)_{n\geq 0}$, respectively.
Fix a sequence $(c_n)_{n\geq 0}$ of real numbers in $(0,1)$ such that $\prod_{k=0}^N c_k$ converges to a strictly positive real number $c_\infty$.
Let $H:=H_0\oplus H_1$ and let $A, B \in B(H)$ be the operators such that
\begin{gather*}
(A\oplus B)(\delta_{2n}) = (\delta_{2n+1},0), \quad (A \oplus B)(\delta_{2n+1}) = (\delta_{2n},\delta_{2n+1})/\sqrt 2,\text{ and}\\
    \quad (A \oplus B)(\gamma_n) = (c_n\gamma_{n+1}, \sqrt{1-c_n^2} \delta_{2n}).
\end{gather*}
Note that $A\oplus B$ is an isometry and thus $(A,B,H)$ defines an $\cAL$-module.
The space $H_0$ is a submodule of $H$ with $H_0^\perp = H_1$. Since $B(H_1)\subset H_0$,
the module $H_1$ is a residual subspace.

We now show that $H_0\subset H$ is not complete.
Let $\xi:=\gamma_0\in H_1$. We show that the vectors $([t,A_t\xi])_{t \in \cT}$ are bounded away from $H_0$ (again writing $A_t\xi = (A_\mu \xi)_{\mu \in t}$).
An induction using the definition of a tree shows that for any nontrivial tree $t$ there is a strictly positive $n(t) \in \N$ such that $0^{n(t)} \in t$ and
$0^m \not\in t$ for all $m \not= n(t)$. Since the range of $B$ is contained in $H_0$, which is invariant for $A \oplus B$, we see that
$\operatorname{proj}_{H_1}(A_t \xi) = \operatorname{proj}_{H_1}(A^{n(t)}(\gamma_0)) = \big(\prod^{n(t)}_{j = 0} c_j\big) \gamma_{n(t)+1}$. Consequently,
$d(A_t\xi, H_0^{t}) = \prod^{n(t)}_{j = 0} c_j > c_\infty > 0$.
\end{example}

We now prove a key lemma that unlocks the classification of representations of $\cOL$.

\begin{lemma}\label{lem:key}
Suppose that $H$ is an $\cAL$-module containing two nonzero submodules $H_0$ and $H_1$.
Assume that $H_1$ is complete and at least one of $H_0$ and $H_1$ is finite dimensional.
Then $H_0\cap H_1\neq \{0\}$. If, in addition, $H_i$ is irreducible (for either of $i = 0, 1$), then $H_i \subset H_{1-i}$.
\end{lemma}
\begin{proof}
Fix $\xi\in H_0\setminus\{0\}$. There exists a vertex $v\in \Lambda^0$ such that $A_v\xi\neq 0$. Since $H_0$ is a module, $A_v \xi \in H_0$.
So by replacing $\xi$ with $A_v\xi/\|A_v\xi\|$, we may assume that $\xi\in A_vH_0$ and has norm $1$.
Fix $0<\varep<1$. By completeness of $H_1\subset H$ there exists a tree $t\in v\cT$ and vectors $(\eta_\mu)_{\mu \in t}$ in $H_1$ such that $\sum_{\mu\in t} \| A_\mu \xi - \eta_\mu\|^2<\varep^2.$

Claim: There exists $\mu \in t$ such that $A_\mu\xi \neq 0 \neq \eta_\mu$ and such that
\[
\Big\|\frac{A_\mu\xi}{\|A_\mu\xi\|} - \frac{\eta_\mu}{\|\eta_\mu\|}\Big\| < 2\varep.
\]

To see this, let $L = \{\mu \in t :\ A_\mu\xi\neq 0\}$. Suppose for contradiction that
\[
\min_{\mu \in L} \frac{\|A_\mu\xi-\eta_\mu\|}{\|A_\mu \xi\|}\geq \varep.
\]
Then, using at the last line that $A_t : A_v \to H_{t} = \bigoplus_{\mu \in t} H_{s(\mu)}$ is an isometry, we calculate
\begin{align*}
\sum_{\mu\in t} \|A_\mu\xi - \eta_\mu\|^2 & \geq \sum_{\mu\in L}  \|A_\mu\xi - \eta_\mu\|^2
 = \sum_{\mu\in L} \|A_\mu\xi\|^2\cdot \| \frac{A_\mu\xi-\eta_\mu}{\|A_\mu \xi\|}\| ^2\\
&\geq \sum_{\mu\in L} \|A_\mu\xi\|^2\cdot \varep^2
 = \varep^2
\end{align*}
This is impossible, so there exists $\mu \in t$ such that $A_\mu\xi\neq 0$ and
$$\frac{\|A_\mu\xi-\eta_\mu\|}{\|A_\mu \xi\|}<\varep.$$
Now
\begin{align*}
\| \frac{A_\mu\xi}{\|A_\mu \xi\|} - \frac{\eta_\mu}{\|\eta_\mu\|}\| & \leq \| \frac{A_\mu\xi}{\|A_\mu \xi\|} - \frac{\eta_\mu}{\|A_\mu\xi\|}\| + \| \frac{\eta_\mu}{\|A_\mu \xi\|} - \frac{\eta_\mu}{\|\eta_\mu\|}\|  \\
& = \frac{\|A_\mu\xi-\eta_\mu\|}{\|A_\mu \xi\|} +| \frac{\|\eta_\mu\|}{\|A_\mu\xi\|} - 1|\\
& = \frac{\|A_\mu\xi-\eta_\mu\|}{\|A_\mu \xi\|} + | \frac{\|\eta_\mu\| - \|A_\mu\xi\|}{ \|A_\mu \xi\|} |\\
& \leq 2\frac{\|A_\mu\xi-\eta_\mu\|}{\|A_\mu \xi\|}
 < 2\varep.
\end{align*}
Note that since $\xi\in H_0$ and $H_0$ is a submodule, then $A_\mu\xi\in H_0$ for all path $\mu$.
Hence, we have shown that for all $\varep>0$ there exists a pair of \emph{unit} vectors $\eta\in H_0$ and $\zeta\in H_1$ so that $\|\eta-\zeta\|<\varep.$
Thus, the unit spheres of $H_0$ and $H_1$ are at distance zero.
Since $H_0$ or $H_1$ is finite dimensional one of these unit spheres is compact.
Therefore, there is a vector $\xi$ in one of the unit sphere at distance zero from the other unit sphere.
This means that $\xi$ is a unit vector in $H_0\cap H_1$ and thus $H_0\cap H_1$ is a nonzero submodule.

Now assume additionally that $H_i$ is irreducible.
By the first statement $H_0\cap H_1$ is a nontrivial submodule of $H_i$.
By irreducibility of $H_i$ we deduce that $H_0\cap H_1=H_i$, so~$H_i\subset H_{1-i}$.
\end{proof}

We may now prove that for finite-dimensional modules, completeness and weak completeness coincide.

\begin{proposition} \label{prop:weak-complete}
Assume that $H$ is a finite dimensional $\cAL$-module.
Then a submodule $H'\subset H$ is weakly complete if and only if it is complete.
\end{proposition}
\begin{proof}
We have already proven that weakly complete modules are complete. For the converse we prove the contrapositive.
Assume that $H$ is a finite dimensional $\cAL$-module and that $H = H'\oplus Z$ where $Z$ is a vector subspace
and where $H'\subset H$ is a submodule that is not complete. That is, the closed linear span
$\Pi(H') = \overline{\operatorname{span}}\{[\lambda,\xi]:\ \lambda\in\Lambda\text{ and }\xi \in H'\} \lneq \Pi(H)$.
Then $\Pi(H) = \Pi(H')\oplus L$ with $\Pi(H'),L$ nonzero representations of $\cOL$.
In particular, $L$ is a nonzero submodule of $\Pi(H)$ (for the usual action $a_\lambda\cdot \zeta:=X_\lambda^*\zeta$)
and $H\subset \Pi(H)$ is a finite dimensional complete submodule
(after identifying $H$ with $J_H(H)$ and restricting the $\cOL$-action into an $\cAL$-action).
Hence, Lemma~\ref{lem:key} implies that $L\cap H$ is a nonzero submodule of $H$.
Since $L$ is orthogonal to $\Pi(H')$ it is in particular orthogonal to $H'$.
Therefore, $L\cap H$ is a nonzero submodule of $H$ contained in $Z$.
In particular, $Z$ is not residual and thus $H'$ is not weakly complete.
\end{proof}

\subsection{Decompositions of \texorpdfstring{$\cAL$}{AΛ}-modules and representations of \texorpdfstring{$\cOL$}{OΛ}}
The aim of this section is to describe how to decompose $\cAL$-modules and then to classify representations of $\cOL$.
Our techniques work under a finite dimensionality assumption defined below. We will later introduce in
Section~\ref{sec:dimension} a new dimension vector that is coherent with this definition.

\begin{definition}\label{def:finite-P-dimension}
An $\cAL$-module $H$ (resp.~a representation $L$ of $\cOL$) has \emph{finite Pythagorean dimension}
(in short finite P-dimension) if there is a finite dimensional $\cAL$-module $K$ satisfying $\Pi(H)\simeq \Pi(K)$
(resp.~$L\simeq \Pi(K)$).
\end{definition}

\begin{remark} \label{rem:finite-p-dim}
The notion of finite P-dimension of a representation of $\cOL$ depends on the choice of the graph $\Lambda$, but we do not make this explicit in our notation or terminology unless there is a possibility of ambiguity.
\end{remark}

The next proposition underpins the most important structural results of this paper.

\begin{proposition}\label{prop:decomposition}
Let $H$ be a nonzero $\cAL$-module with finite Pythagorean dimension.
\begin{enumerate}
\item
There exist $n\geq 1$, irreducible finite dimensional submodules $H_1,\cdots,H_n$ of $H$, and a residual subspace $Z\subset H$ such that
$$H=H_1\oplus\cdots\oplus H_n \oplus Z.$$
The module $H_s:=\bigoplus_{i=1}^n H_i \le H$ is a complete submodule.
Moreover, this decomposition is unique in the following sense: if $H \cong K_1\oplus\cdots\oplus K_m\oplus Y$
with each $K_j$ irreducible and $Y$ residual, then $n=m$ and there exists a bijection $\varphi$ of $\{1,\cdots,n\}$
such that $H_i\simeq K_{\varphi(i)}$ as module for each $1\leq i\leq n$.
\item Write $H_s$ for the submodule $H_1\oplus\cdots\oplus H_n \le H$ from~(1).
This is the \emph{smallest complete submodule} of $H$: it is complete and all complete submodules of $H$ contain $H_s$.
\item If $K$ is an $\cAL$-module with finite P-dimension and $K_s$ is its smallest complete submodule as in~(2), and if $\theta:H\to K$
    is a morphism of modules, then $\theta(H_s)\le K_s$.
\item The space $H_s$ is the smallest complete submodule of $U\Pi(H)$.
\item Let $H, K$ be $\cAL$-modules with finite P-dimension, and let $H_s$ and $K_s$ be their smallest complete submodules as in~(2).
    Then $\Pi(H)\simeq \Pi(K)$ as representations of $\cOL$ if and only if $H_s\simeq K_s$ as $\cAL$-modules.
\item The representation $\Pi(H)\in \ob \Rep(\cOL)$ is irreducible if and only if $H=H_1\oplus Z$ with $H_1$ an irreducible submodule and $Z$ a residual subspace, i.e.~$H$ is an indecomposable module.
\item If $H=H_1\oplus\cdots\oplus H_n \oplus Z$ as in (1), then $\Pi(H)\simeq \Pi(H_1)\oplus\cdots\oplus\Pi(H_n)$.
Moreover, each $\Pi(H_i)$ is irreducible and $\Pi(H_i)\simeq \Pi(H_j)$ as representations of $\cOL$ if and only if $H_i\simeq H_j$ as $\cAL$-modules.
\end{enumerate}
\end{proposition}

\begin{proof}
Proof of (1).
We first show by induction that every finite-dimensional complete submodule $K$ of $H$ contains an irreducible submodule. The base case is $\dim(K) = 1$, in which case $K$ itself is (trivially) an irreducible submodule of $K$. So suppose inductively that every complete submodule $K \subset H$ with $\dim(K) < N$ contains an irreducible submodule, and suppose that $K$ is a complete submodule with $\dim(K) = N$. If $K$ is irreducible the, again, $K$ is an irreducible submodule of itself, so suppose that $K$ is reducible. Then it decomposes as $K = K' \oplus Z$ with $K'$ weakly complete and $Z$ residual, and both nonzero. In particular $\dim(K') = \dim(K) - \dim(Z) < \dim(K)$. Proposition~\ref{prop:weak-complete} shows that $K'$ is complete. So the inductive hypothesis applies, and yields an irreducible submodule $K''$ of $K'$; this $K''$ is then the desired irreducible submodule of $K$.

We now show that if $K$ is a finite-dimensional complete submodule of $H$, then it admits a decomposition $K = K_1 \oplus K_2 \oplus \cdots K_n \oplus Z$ as above. Again, we proceed by induction on $\dim(K)$. If $\dim(K) = 1$, then $n = 1$, $K_1 = K$ and $Z = \{0\}$ is the desired decomposition. So suppose inductively that every complete submodule $K' \subset H$ with $\dim(K') < N$ has a decomposition of the desired form, and suppose that $\dim(K) = N$. By the preceding paragraph, $K$ has a nontrivial irreducible submodule $K_1$. Let $K' = K \ominus K_1$, the orthogonal complement of $K_1$ in $K$. Again, Proposition~\ref{prop:weak-complete} implies that $K'$ is complete, so the inductive hypothesis gives a decomposition $K' = K_2 \oplus \cdot \oplus K_n \oplus Z$ of the desired form, and then $K = K_1 \oplus K' = K_1 \oplus K_2 \oplus \cdot \oplus K_n \oplus Z$ as required.
%%%%%%%%%%%%%%%%%%%%%%

By assumption $H$ has finite P-dimension implying that there exists a finite dimensional $\cAL$-module $K'$ and an isomorphism $\theta:\Pi(K')\to \Pi(H)$.
Since $K'$ is finite dimensional it admits a decomposition as above $K'=K_1'\oplus\cdots\oplus K_n'\oplus Z'$ with $K_i'$ irreducible and $Z'$ residual.
Consider now $\Pi(H)$ as an $\cAL$-module rather than a representation of $\cOL$ and identify $H$ with its image $J_H(H)$ in
$\Pi(H)$ under the isometry $J_H$ of Proposition~\ref{prop:functor}(3) (see also Remark~\ref{rmk:cAL helps}).
By definition of $\Pi(H)$ the submodule $H\subset \Pi(H)$ is complete.
Hence, we may apply Lemma \ref{lem:key} to the modules $\theta(K_i')$ and $H$.
We deduce that $\theta(K_i')\subset H$ and thus $\theta(K_s')\subset H$ where $K_s':=\oplus_{i=1}^n K_i'$.
Note that $K_s'\subset K$ is complete and thus is complete inside $\Pi(H)$.
We deduce that $H$ contains a finite dimensional complete submodule, namely $\theta(K_s')$.
We obtain the existence of a decomposition as in the satetement (1) where $H_i:=\theta(K_i')$ and where $Z:= H\ominus \theta(K_s')$.

To prove uniqueness of the decomposition we use Lemma~\ref{lem:key} and reproduce a version of Schur's lemma.
In details, assume that $\theta$ is an isomorphism of modules from $H_s\oplus Z$ to $K_s\oplus Y$ where $H_s:=H_1^{c_1}\oplus \dots \oplus H_n^{c_n}$ and $K_s=K_1^{d_1}\oplus\cdots\oplus K_m^{d_m}$ for some $n,m,c_i,d_i\geq 1$, with the
$H_i$ and $K_i$ irreducible, and $Z,Y$ residual, and with $H_i\not\simeq H_j$ and $K_i\not\simeq K_j$ if $i\neq j$.
Note that $\theta$ is a continuous bijective linear map that intertwines the $A_\lambda^H$'s with the $A_\lambda^K$'s.
Write $H_{i,j}$ and $K_{i,j}$ for the $j$th summands of $H_i^{c_i}$ and $K_i^{d_i}$ respectively.
Fix $i,j$ and consider $\theta(H_{i,j})$.
Since $K_s$ is complete and $\theta(H_{i,j})$ is irreducible and both modules are finite dimensional we deduce that $\theta(H_{i,j})\subset K_s.$
Hence $\theta(H_s)\subset K_s$.
The same argument applied to $\theta^{-1}$ shows that $\theta(H_s)=K_s$.

Now, fix $\xi\in K_s$ and decompose it as $\xi = \sum_{i,j} \xi_{i,j}$ with $\xi_{i,j}\in K_{i,j}$.
Let $p_{k,l} : K_s \to K_{k,l}$ be the orthogonal projection.
We show that $p_{k,l}$ commutes with the $A^K_\lambda$'s (this requires an argument because $\cAL$ is non-self-adjoint).
Indeed, $A^K_\lambda \circ p_{k,l} \xi = A^K_\lambda \xi_{k,l}$. Since the $K_{i,j}$ are submodules we have
$p_{k,l} A^K_\lambda \xi = p_{k,l}\sum_{i,j} A^K_\lambda\xi_{i,j}=A^K_\lambda\xi_{k,l}$.
Thus $p_{k,l}\circ \theta(H_{i,j})$ is a submodule of $K_s$.
Since $\sum_{k,l} p_{k,l}$ acts as the identity on $K_s$ there exist $\varphi(i)$ and $l$ such that $p_{\varphi(i),l}\theta(H_{i,j})$ is nontrivial.
By irreducibility of both $H_{i,j}$ and $K_{\varphi(i),l}$ we deduce that $H_{i,j}$ is isomorphic to $K_{\varphi(i),l}$ via
$p_{\varphi(i),l}\circ \theta$\footnote{This does not imply that $\theta$ maps $H_{i,j}$ inside $K_{\varphi(i),l}$: it may, for instance, implement a diagonal embedding.}.
Thus $H_i\simeq K_{\varphi(i)}$ as modules. Assume now that $(k,r)$ satisfies $p_{k,r}\circ \theta(H_{i,j}) \not=  \{0\}$.
Then $K_r \simeq K_{\varphi(i)}$ and thus $k=\varphi(i)$. This implies that $\theta(H_{i,j})\subset K_{\varphi(i)}^{d_{\varphi(i)}}.$
Applying the same argument to $H_{i,k}$ for each $k$, we deduce that $\theta(H_i^{c_i})\subset K_{\varphi(i)}^{d_{\varphi(i)}}.$
The map $i\mapsto \varphi(i)$ is injective because $H_i \not\simeq H_j$ for $i\neq j$. Applying a similar argument to $\theta^{-1}$
we deduce that $\theta(H_i^{c_i})=K_{\varphi(i)}^{d_{\varphi(i)}}$ and that $i\mapsto \varphi(i)$ is bijective.

It remains to prove that the multiplicities are equal, namely $c_i=d_{\varphi(i)}$ for all $i$.
To see this, note that if $L$ is an irreducible module and $L^k\simeq L^j$, then $\End(L^k)\simeq \End(L^j)$. Since $\End(L^d)$ is a matrix algebra of size $d \times d$
for each of $d = k,l$, the isomorphism $\End(L^k) \simeq \End(L^j)$ forces $j = k$. We have now proved that $n=m$, that $\varphi$ is a bijection of $\{1,\cdots,n\}$,
that $c_i=d_{\varphi(i)}$ for all $i$, and that $H_i\simeq K_{\varphi(i)}$ for all $i$.

Proof of (2).
We have already proven that $H_s$ is complete in $H$.
Take $H'\subset H$ complete.
By Lemma~\ref{lem:key} we have $H_i\subset H$ for all $1\leq i\leq n$.
Hence, $\oplus_i H_i\subset H'$ meaning $H_s\subset H'$.

Proof of (3).
Decompose $H_s$ as $\bigoplus_{i=1}^n H_i$ with each $H_i$ irreducible as a module.
Note that $\theta(H_i)$ is either irreducible or is trivial. If it is trivial it is of course inside $K_s$.
Assume now that $\theta(H_i)$ is nontrivial.
Lemma \ref{lem:key} implies that $\theta(H_i)\subset K_s$ by irreducibility of $H_i$ and completeness of $K_s$.
Hence, for any $1\leq i\leq n$ we have $\theta(H_i)\subset K_s$ and thus $\theta(H_s)\subset K_s$.

Proof of (4).
Consider the representation $\Pi(H)$ of $\cOL$ and let $U\Pi(H)$ be the same Hilbert space considered as an $\cAL$-module.
Let $L\subset \Pi(H)$ be a complete submodule.
By Lemma \ref{lem:key}, $H_i\subset L$ for all $i$ and thus $H_s\subset L$.
This proves that $H_s$ is the smallest complete submodule of $U\Pi(H)$.

Proof of (5).
Consider an isomorphism $\theta:\Pi(H)\to \Pi(K)$ of representations of $\cOL$.
By (3) we have $\theta(H_s)\subset K_s$.
Reapplying this argument to $\theta^{-1}$ yields $\theta(H_s)\supset K_s$ and that $\theta$ defines an isomorphism of modules between $H_s$ and $K_s$.
The converse is given by functoriality of $\Pi$.

Proof of (6).
Decompose $H$ as $H_1\oplus \cdots \oplus H_n\oplus Z$ as in~(1).
If $n\geq 2$, then $\Pi(H_1)\subset \Pi(H)$ is a proper nonzero subrepresentation and thus $\Pi(H)$ is reducible.
Assume that $n=1$ and assume that $L\subset \Pi(H)$ is a nonzero subrepresentation.
Since $H_1=H_s$, it is complete. So Lemma \ref{lem:key} implies that $L\cap H_1$ is a nonzero module.
By irreducibility $H_1\subset L$ and thus $L$ is complete.
Since $L$ is complete and is closed under the $X_\lambda$ we see that $L=\Pi(H)$.
Hence, $\Pi(H)$ is an irreducible representation of $\cOL$.

Proof of (7).
Since $\bigoplus_i H_i$ is complete, we have $\Pi(H) \simeq \Pi\big(\bigoplus_i H_i\big)$. Functoriality of $\Pi$ then
gives $\Pi(H) \simeq \bigoplus_i \Pi(H_i)$. The $\Pi(H_i)$ are irreducible by~(6) and the final statement
follows from Proposition~\ref{prop:functor}.
\end{proof}

\begin{example}\label{ex:reducible}
We present an example of decomposable $\cAL$-module $H=H_1\oplus H_2\oplus Z$ that is irreducible as a representation of $\cPL$.
Let $\Lambda$ be the graph with one vertex and two loops and write $A_1,A_2$ for the generators of $\cPL$.
Take $H=\C^3$ with basis $e_0,e_1,e_2$ and
$$A_1\oplus A_2:H\to H\oplus H, \ e_0\mapsto (e_1,e_2)/\sqrt 2, \ e_1\mapsto (0,e_1), \ e_2\mapsto (e_2,0).$$
Since $A_1\oplus A_2$ is isometric, the pair $(A_1,A_2,H)$ defines an $\cAL$-module.
Let $H_i = \C e_i$ for $i = 1,2$, and $Z = \C e_0$. Then $H = H_1\oplus H_2\oplus Z$, each $H_i$ is a submodule of $H$,
and $Z$ is residual. The $H_i$ are irreducible as they are 1-dimensional.
Hence, $\Pi(H)=\Pi(H_1)\oplus \Pi(H_2)$ is a direct sum of two irreducibles representations.

We show that $H$ is irreducible as a representation of $\cPL$.
For this, we show that every nonzero vector in $H$ is cyclic for $\cPL$.
We first show that each of $e_0,e_1$, and $e_2$ is cyclic.
We have $A_1 e_0 = e_1/\sqrt 2$ and $A_2e_0=e_2/\sqrt 2$, and so $e_0$ is cyclic. Since
$A_1^* e_1 = e_0/\sqrt 2$ and $A_2^*e_2 = e_0/\sqrt 2$, it follows that $e_1$ and $e_2$ are also cyclic.

Now to see that every nonzero vector is cyclic, fix $x=\sum_i x_i e_i\in H \setminus \{0\}$. First suppose that $x_0 \neq 0$. Then
$$A_2A_1 x = A_2(x_0/\sqrt 2 e_1 + x_2 e_2)= x_0/\sqrt 2 e_1$$
is a nonzero scalar multiple of the cyclic vector $e_1$, so $x$ is cyclic. Now suppose that $x_0 = 0$. Then
\[
A_1x = A_1(x_1 e_1 + x_2e_2) = x_2 e_2,\quad\text{ and }\quad A_2 x = A_2(x_1 e_1 + x_2e_2) = x_1e_1.
\]
Since one of $x_1, x_2$ is nonzero, one of these is a nonzero multiple of a cyclic vector, so we see that $x$ is cyclic.
\end{example}

We reinterpret some of the results of Proposition \ref{prop:decomposition} in a categorical manner.

\begin{corollary}\label{cor:categories}
A module is called full if it does not contain a nonzero residual subspace.
Let $\Mod_{\full,\fd}(\cAL)$ be the category whose objects are finite dimensional, full modules and whose arrows are intertwiners.
Let $\Rep_{\Pfd}(\cOL)$ be the full subcategory of $\Rep(\cOL)$ whose objects are the representations of $\cOL$ with finite Pythagorean dimension.
The  functor $\Pi$ restricts into an essentially surjective and fully faithful functor
$$\Pi_{\full,\fd}:\Mod_{\full,\fd}(\cAL)\to \Rep_{\Pfd}(\cOL);$$
that is: (a) for all $L\in \ob \Rep_{\fd}(\cOL)$ there exists a finite dimensional, full module $H$ such that $\Pi(H)\simeq L$; and (b) $\Pi_{\full,\fd}$ determines bijection on hom-spaces.
Conversely, the operation $L\mapsto (UL)_s$ that take a representation $L$ of $\cOL$, regards it as an $\cAL$-module, and then passes to its smallest complete submodule, defines a functor
$$\Sigma:\Rep_{\Pfd}(\cOL)\to \Mod_{\full,\fd}(\cAL).$$
This functor is essentially surjective and fully faithful. The pair of functors  $(\Pi_{\full,\fd},\Sigma)$ constitutes an equivalence of categories.
\end{corollary}

\begin{remark}
Our proof of Corollary~\ref{cor:categories} is not maximally efficient: we can just prove directly that $(\Pi_{\full,\fd},\Sigma)$ is an equivalence of categories. We have chosen to present
a proof that details why $\Pi_{\full,\fd}$ is well-defined, essentially surjective and fully faithful because we believe that this will help to elucidate the details of our construction,
and helps to clarify some important properties of $\Pi$.
\end{remark}

\begin{proof}[Proof of Corollary~\ref{cor:categories}]
We have already seen that $\Pi:\Mod(\cAL)\to\Rep(\cOL)$ defined on arrows $\theta$ of $\Mod(\cAL)$ by
$$\Pi(\theta)[\lambda,\xi]:=[\lambda,\theta(\xi)]$$
is a functor. By definition, if $H$ is a finite dimensional module, then $\Pi(H)$ has finite P-dimension.
Hence, $\Pi_{\full,\fd}:\Mod_{\full,\fd}(\cAL)\to \Rep_{\Pfd}(\cOL)$ is well-defined.

Fix a full, nonzero, finite dimensional module $H$.
Proposition~\ref{prop:decomposition}(1) yields irreducible submodules $H_i$ and a residual subspace $Z$ such that $H = H_1\oplus \cdots\oplus H_n\oplus Z$.
Since $H$ is full, $Z=\{0\}$. By definition of $\Pi$, the representation $\Pi(H)$ of $\cOL$ decomposes as $\pi(H) = \Pi(H_1)\oplus\cdots\oplus\Pi(H_n)$.

By Proposition~\ref{prop:decomposition}(6) each $\Pi(H_i)$ is an irreducible representation of $\cOL$. Hence Schur's Lemma implies that $\Pi_{\full,\fd}$ induces bijections on hom-spaces.

Fix $L\in \ob \Rep(\cOL)$ of finite P-dimension.
Since $\Pi:\Mod(\cAL)\to\Rep(\cOL)$ is essentially surjective and $L$ has finite P-dimension, there exists a finite-dimensional $\cAL$-module $H$ satisfying $\Pi(H)\simeq L$.
By~(1) there are a full submodule $H'$ of $H$ and a residual space $Z$ such that $H = H'\oplus Z$. Hence $\Pi(H')\simeq L$, so $\Pi_{\full,\fd}$ is essentially surjective. This proves the first part of the corollary.

Let $\Sigma$ be as in the statement. Then $\Sigma$ is the composition of the forgetful functor $U:\Rep(\cOL)\to\Mod(\cAL)$ (restricted to finite Pythagorean dimensional representations of $\cOL$)
followed by passing to the smallest complete submodule. The latter is well-defined on $\cAL$-modules of finite P-dimension by Statements (1)~and~(2) of Proposition \ref{prop:decomposition}.
The resulting module is finite dimensional, and is full by minimality. Hence, $L\mapsto \Sigma(L)$ is well-defined from representations of $\cOL$ of finite P-dimension to full, finite-dimensional modules.
Any intertwiner $\theta:L\to L'$ of representations of $\cOL$ maps $\Sigma(L)$ to $\Sigma(L')$ by Proposition~\ref{prop:decomposition}(3).
Since $\theta$ intertwines the $\cOL$-actions we deduce that the restriction of $\Sigma(L)\to \Sigma(L')$ intertwines the module structures.
Hence $\Sigma$ is a well-defined functor from $\Rep_{\Pfd}(\cOL)$ to $\Mod_{\full,\fd}(\cAL)$.

To see that $\Sigma$ is essentially surjective and fully faithful, we employ a similar argument to the one above (that is, using semi-simplicity and Schur's lemma).
To see that $\Pi_{\full,\fd}$ and $\Sigma$ constitute an equivalence of categories, just calculate directly to see that the natural transformations given by the
maps $J_H:H\to \Pi(H)$ of Proposition~\ref{prop:functor} satisfy the requisite intertwining relations.
\end{proof}

\section{Pythagorean dimension vector} \label{sec:dimension}

Recall from Definition~\ref{def:finite-P-dimension} that $L \in \ob \Rep(\cOL)$ has finite Pythagorean dimension if there exists a finite dimensional $\cAL$-module $H$ such that $\Pi(H)\simeq L$. Expanding on this,
we define the \emph{dimension vector} of a representation of $\cOL$.

\subsection{The usual dimension vector}

Let $H$ be a $\cAL$-module (note that this includes all representations of either $\cPL$ or $\cOL$). Then $H$ decomposes as $H = \oplus_{v\in \Lambda^0} a_v \cdot H$; we define $H_v := a_v \cdot H$ for each $v \in \Lambda^0$. We define the \emph{dimension vector of $H$} by
$$\dim_\Lambda(H):\Lambda^0\to \N\cup\{\infty\},\ v\mapsto \dim(H_v)$$
(here $\dim(H_v)$ is just the complex vector-space dimension of $H_v$).
We call $\dim_\Lambda(H)$ the \emph{usual dimension vector} of the representation $H$
and we call $\dim(H) = \sum_{v \in \Lambda^0} \dim(H_v) = \|\dim_\Lambda(H)\|_1$ the \emph{usual dimension} of $H$.

\subsection{Pythagorean dimension vector}

\begin{definition}
Fix $L\in\ob \Rep(\cOL)$. Let $\Pi:\Mod(\cAL)\to\Rep(\cOL)$ be the (essentially surjective) Pythagorean functor of Proposition~\ref{prop:functor}.
If there does not exist any $H\in\ob\Mod(\cAL)$ with $\dim(H)<\infty$ satisfying $\Pi(H) \simeq L$, then we say that $L$ has \emph{infinite Pythagorean dimension}.
Otherwise we say that $L$ has \emph{finite Pythagorean dimension}.
\end{definition}

\begin{proposition}\label{prop:Pdim}
Fix $L\in\ob \Rep(\cOL)$ and suppose that $L$ has finite Pythagorean dimension. Then the following three dimension vectors are well-defined and equal:
\begin{enumerate}
\item $\dim_\Lambda\circ \Sigma(L)$;
\item the coordinatewise minimum of $\{\dim_\Lambda(K):\ K\in \ob \Mod(\cAL)\text{ and } \Pi(K)\simeq L\}$;
\item the coordinatewise minimum of $\{\dim_\Lambda(H'):\ H'\subset UL\text{ is a complete submodule}\}$,
\end{enumerate}
where here $\Sigma$ denotes the object-map of the functor $\Sigma$.
\end{proposition}

\begin{definition}
Fix $L\in\ob \Rep(\cOL)$ and suppose that $L$ has finite Pythagorean dimension. We call the vector $\Pdim_\Lambda(L)$ of Proposition~\ref{prop:Pdim} the \emph{Pythagorean dimension vector}, or just the \emph{P-dimension vector}, of $L$. We define the \emph{Pythagorean dimension}, or just the \emph{P-dimension}, of $L$ to be
$$
    \Pdim(L):=\sum_{v\in\Lambda^0} \Pdim_\Lambda(L)(v).
$$
\end{definition}

\begin{proof}[Proof of Proposition~\ref{prop:Pdim}]
Fix an $\cAL$-module $H$ with $\Pi(H)\simeq L$.
By Proposition~\ref{prop:decomposition}(1) there are irreducible submodules $H_i$ of $H$ and a residual subspace $Z$ such that $H = H_1\oplus\cdots\oplus H_n\oplus Z$;
and by Proposition~\ref{prop:decomposition}(2)~and~(3), $H_s := H_1\oplus\cdots\oplus H_n$ is a complete submodule of $H$, and is unique up to isomorphism of modules.
Hence, the usual dimension vector $\dim_\Lambda(H_s)$ of $H_s$ is an invariant of $L$: it does not depend on the choice of $H$.
By Proposition~\ref{prop:decomposition}(4), $H_s\subset H$ is the smallest complete submodule of $H$, so $\Pdim_\Lambda(L) := \dim_\Lambda\circ \Sigma(L)$ is well defined.
That this agrees with the coordinatewise minima in (2)~and~(3) for the same reason.
\end{proof}

We deduce a new dimension vector for $\cAL$-modules.

\begin{definition}
Consider an $\cAL$-module $H$.
If for every complete submodule $H'\subset H$ we have $\dim(H')=\infty$, then we say that $H$ has infinite Pythagorean dimension and we do not define any Pythagorean dimension vector on $H$.
Otherwise, we set the \emph{Pythagorean dimension} of $H$ and the \emph{Pythagorean dimension vector} of $H$ to be the Pythagorean dimension and the Pythagorean dimension vector of $\Pi(H)$, and denote them by
$$ \Pdim^{\cAL}(H) \text{ and }  \Pdim^{\cAL}_\Lambda(H),$$
respectively.
These define maps on $\ob\Mod_{\Pfd}(\cAL)$ ($\cAL$-modules with finite P-dimension) satisfying the equalities
$$\Pdim_\Lambda^{\cAL}=\Pdim_\Lambda\circ \Pi \text{ and } \Pdim^{\cAL}=\Pdim\circ \Pi.$$
\end{definition}

\begin{remark}
If $H=H'\oplus Z$ with $H'$ a complete submodule $Z$ a residual subspace, then $\Pdim^{\cAL}(H)\leq \dim(H')$.
We have equality if $H'$ is the smallest complete submodule.

The definition of the Pythagorean dimension of a representation of $\cOL$ depends on the choice of the $k$-graph $\Lambda$. Our notation does not make this explicit because typically a fixed $k$-graph $\Lambda$ is given.
In Section~\ref{sec:example} we exhibit two graphs yielding two different Pythagorean dimension functions for the same representations of the Cuntz algebra $\cO_2$: we provide explicit examples where these dimension functions differ.
\end{remark}

Using our earlier classification results, we deduce the following.

\begin{corollary}\label{cor:classification}
Let $\Lambda$ be a row-finite and locally convex $k$-graph. For any finitely supported dimension vector $d:\Lambda^0\to\N$ the functors $\Pi$ and $\Sigma$ define bijections between the irreducible classes of $\cAL$-modules with usual dimension vector $d$ and the irreducible classes of representations of $\cOL$ of Pythagorean dimension vector $d$.
\end{corollary}
\begin{proof}
This follows from Propositions \ref{prop:decomposition} and \ref{prop:Pdim}.
\end{proof}

\section{Moduli spaces for representations of 1-graph \texorpdfstring{C$^*$}{C*}-algebras}\label{sec:moduli spaces}

We now derive sharper results for the specific case of ordinary directed graphs (that is, 1-graphs) $\Lambda$ as opposed to \emph{higher} rank graphs.

\begin{center}\textbf{From now on we assume that $\Lambda$ is an ordinary directed graph.}
\end{center}

We continue to identify $\Lambda$ with the path space of the underlying graph; that is, $(\Lambda^0, \Lambda^1, r, s)$ is a countable row-finite directed graph in the sense of \cite{Raeburn05},
and $\Lambda$ is its category of finite paths. Every $1$-graph is trivially locally convex.

The aim of this section is to define some geometric objects (certain smooth real manifolds) in natural bijection with components of the spectrum of $\cOL$: there will be one component per finite P-dimension vector $d$.

\begin{center}\textbf{From now on we fix a dimension vector $d:\Lambda^0\to\N\cup\{\infty\}, v\mapsto dv$ that is finite, i.e.~$\sum_{v\in\Lambda^0}dv<\infty$. }\end{center}

\subsection{Matrix modules of \texorpdfstring{$\cAL$}{AΛ}}
We introduce coordinates by defining $H_v:=\C^{dv}$ for each $v\in \Lambda^0$.
Set $$M(d):=\bigoplus_{\lambda\in \Lambda^1} M_{ds(\lambda),dr(\lambda)}(\C);$$
that is, $M(d)$ is the direct sum of the vector spaces $M_{ds(\lambda), dr(\lambda)}(\C)$ of $ds(\lambda) \times d(r\lambda)$ complex matrices, indexed by edges $\lambda \in \Lambda^1$.
For $A=(A_\lambda)_\lambda \in M(d)$ and $v\in \Lambda^0$, define
\[
    R_v=\oplus_{\lambda\in v\Lambda^1} A_\lambda,
\]
the $|v\Lambda^1| \times 1$ block column matrix with blocks $A_\lambda, \lambda \in v\Lambda^1$.
This is a $D(v)\times dv$ matrix, where $D(v):=\sum_{\lambda\in v\Lambda^1} ds(\lambda)$.

The matrix $A$ defines an $\cAL$-module structure on $H:=\oplus_{v\in \Lambda^0} H_v$ if and only if each $R_v$ is an isometry; that is, if and only if $R_v^*R_v=I_{dv}$ for each $v\in\Lambda^0$.
We write $\mr(d)$ for the collection of tuples with this block-isometry property. An element of $\mr(d)$ is called a \emph{$d$-dimensional matrix module}.

\begin{remark}\label{rem:nonempty}
	\begin{enumerate}
		\item Assume that $H$ is a $d$-dimensional matrix module.
		For each vertex $v$ the matrix $R_v : H_v \to \oplus_{\lambda \in v\Lambda^1} H_{s(\lambda)}$ is an isometry. Hence
		\begin{equation}\label{eq:dimension}
            dv\leq \sum_{\lambda\in v\Lambda^1} ds(\lambda)=: D(v).
        \end{equation}
		Hence, if $\mr(d)$ is nonempty, then $dv\leq D(v)$ for all $v\in\Lambda^0.$
		Conversely, for any fixed $v$ such that $dv \leq D(v)$, there exists an isometry
        $R_v : H_v \to \oplus_{\lambda\in v\Lambda^1} H_{s(\lambda)}$ by dimension considerations. So if $dv\leq D(v)$ for all $v$, then we can choose
        isometries $\big(R_v: H_v \to \oplus_{\lambda\in v\Lambda^1} H_{s(\lambda)}\big)_{v \in \Lambda^0}$,
        and hence there exists an $\cAL$-matrix module representation with dimension vector $d$.
		Hence
		$$
            \mr(d)\neq\emptyset \Leftrightarrow \eqref{eq:dimension} \text{ holds for all vertices}.
        $$
		Note that even when $\mr(d) \not= \emptyset$, there may be no \emph{irreducible} modules in $\mr(d)$.
		This occurs in the first example of Section~\ref{sec:example-CK} for almost all finite $d$.
		\item Let $\Lambda$ be the rose with $n$ petals---a single vertex $v$ with $n$ loops. So $\cOL = \cO_n$. Then a $d$-dimensional matrix module of $\cAL$ consists of matrices $A_1, \dots, A_n \in B(\C^d) = M_d(\C)$. The matrix $R_v$ is the $dn \times n$ matrix whose $i$th $n \times n$ block is $A_i$. This $R_v$ is an isometry precisely if $R_v^*R_v = I_d$. Here, \eqref{eq:dimension} trivially holds for all $d$-dimensional matrix modules.
	\end{enumerate}
\end{remark}

\begin{center}
\textbf{From now on we assume that $d$ satisfies $dv\leq D(v) \text{ for all } v\in \Lambda^0.$}
\end{center}

\subsection{The manifold structure on \texorpdfstring{$\mr(d)$}{MR(d)}}
We show that $\mr(d)$ is a compact submanifold of $M(d)$ and compute its dimension.
We do this by expressing $\mr(d)$ as the preimage of a point under the submersion $f$ defined below.
For $v \in \Lambda^0$, let
\[
    H(dv) :=(C_{v} \in M_{dv}(\C):\ C_{v}=C_{v}^*,\ v\in\Lambda^0),
\]
the Hermitian matrix spaces of size $dv$. Define
\[
    H(d):=\bigoplus_{v\in\Lambda^0} H(dv).
\]
Define $f:M(d)\to H(d)$ by
\[
    f\big((A_\lambda)_\lambda\big) := \left(\sum_{\lambda\in v\Lambda^1} A_\lambda^*A_\lambda\right)_{v\in\Lambda^0}.
\]
That is, $f$ sends each tuple of matrices indexed by edges to the tuple indexed by vertices whose $v$th entry is $\sum_{\lambda\in v\Lambda^1} A_\lambda^*A_\lambda$.
Write $I(d)$ for the identity element of $H(d)$, which is to say the direct sum $\bigoplus_{v \in \Lambda^0} I_{dv}$ of identity matrices. Then
\[
\mr(d) = f^{-1}(I(d)).
\]
The map $f$ is clearly smooth for the usual \emph{real} differential structures of the vector spaces $M(d)$ and $H(d)$.
(It is not complex-analytic.)

To simplify our discussion, regroup the elements of $M(d)$ into block matrices of size $D(v)\times dv$ by considering $R_v:=\oplus_{\lambda\in v\Lambda^1} A_\lambda$ as above.
In this picture, $f$ is given by
\[
    f:\oplus_{v\in\Lambda^0} M_{d(v),D(v)}(\C)\to \oplus_{v\in\Lambda^0}H(d(v)),\ f(R)=(R_v^*R_v)_{v\in\Lambda^0}.
\]

Fix a point $R=(R_v)_{v} \in M(d)$. The derivative of $f$ at $R$ is the (real) linear map
$$
    df_R:M(d)\to H(d),\ X=(X_{v})_{v}\mapsto R^*X+X^*R = (R_{v}^*X_{v} + X^*_{v} R_{v})_{v}.
$$
So to see that $f^{-1}(I(d))$ is a smooth submanifold of $M(d)$ it suffices to prove that $df_R$ is a surjective linear map for each $R \in f^{-1}(I(d)) \subset M(d)$ (see \cite[Corollary 5.14]{Lee12}).

For this, just note that for every $K=(K_{v})_{v}\in H(d)$, we have $K = df_R(RK/2) \in \operatorname{image}(df_R)$.
Hence, $\mr(d) = f^{-1}(I(d))$ is a smooth submanifold of $M(d)$ as claimed. Moreover, we can compute the dimension (as a real vector space) of $\mr(d)$ in terms of those of $M(d)$ and $H(d)$ as
$$
    \dim(\mr(d))=\dim(M(d))-\dim(H(d)).
$$
Recall that for $v \in \Lambda^0$,
\[
    D(v) =\sum_{\lambda\in v\Lambda^1} ds(\lambda).\
\]
By definition,
$$
    \dim(M(d))=\sum_{v\in\Lambda^0} 2 D(v) \cdot dv =\sum_{v\in \Lambda^0}  \sum_{\mu \in v\Lambda^1} 2 ds(\mu)\cdot dv = 2\sum_{\lambda \in \Lambda^1} ds(\lambda) \cdot dr(\lambda)
$$
Hence
$$
    \dim(H(d)) = \sum_{v\in\Lambda^0} \dim(H(dv))= \sum_{v\in \Lambda^0} dv^2,
$$
and so
$$
    \dim(\mr(d))=\sum_{v\in\Lambda^0} dv (2D(v) - dv) = 2\sum_{\lambda\in \Lambda^1} ds(\lambda) \cdot dr(\lambda) - \sum_{v\in \Lambda^0} dv^2.
$$
By assumption $dv\leq D(v)$ for all $v\in\Lambda^0$ (see Remark \ref{rem:nonempty}), so
$$
    \dim(\mr(d)\geq \sum_{v\in \Lambda^0} dv^2.
$$
Since $f$ is continuous, $\mr(d)$ is closed. Moreover, for each module $(A_\lambda)_\lambda$ we have $\|A_\lambda\|\leq 1$ for each $\lambda\in\Lambda^1$.
Hence $\mr(d)$ is closed and bounded and thus compact since  $M(d)$ is finite dimensional. Putting all this together, we have proved the following.

\begin{proposition}
The set $\mr(d)$ of $d$-dimensional matrix modules over $\cAL$ forms a compact smooth submanifold of $M(d):=\oplus_{\lambda\in\Lambda^1} M_{ds(\lambda),dr(\lambda)}(\C)$ of real dimension
$$\dim(\mr(d))=\sum_{v\in\Lambda^0} dv (2D(v) - dv) = 2\sum_{\lambda \in \Lambda^1} ds(\lambda) \cdot dr(\lambda) - \sum_{v \in \Lambda^0} dv^2.$$
Moreover, $\dim(\mr(d))\geq \sum_{v\in\Lambda^0} (dv)^2.$
\end{proposition}

\subsection{The manifold structure of irreducible matrix modules of \texorpdfstring{$\cAL$}{AΛ}}

Consider the subset $\Irr(d)\subset \mr(d)$ of matrix modules that are \emph{irreducible}.
We show that $\Irr(d)$ is an open subset of $\mr(d)$.
This will imply that $\Irr(d)$ is either empty or is a smooth submanifold of $\mr(d)$ of same dimension.
Section \ref{sec:example-CK} provides examples where both situations occur.

Consider a $d$-dimensional module structure $\pi:a_\lambda\mapsto A_\lambda$ on $H:=\oplus_{v\in\Lambda^0}\C^{dv}$; we write $H_v := \C^{dv}$ for each $v$.
This module is irreducible if $H$ contains no nontrivial proper subspace closed under the $A_\lambda$.
Equivalently, $H$ is irreducible if for every collection of subspaces $K_v \subset H_v$ for which $K := \bigoplus_{v \in \Lambda^0} K_v$ satisfies $\{0\}\neq K\neq H$, there exists $\lambda \in \Lambda^1$ satisfying $p_K^\perp A_\lambda p_K\neq 0$. Here, $p_K$ is the orthogonal projection from $H$ onto $K$ and $p_K^\perp=\id-p_K$.

Consider now
$$U(d):=\bigoplus_{v\in\Lambda^0}U(dv)$$
to be the direct product of the unitary groups $U(dv)= U(H_v)$.
By assumption $v\mapsto dv$ is finitely supported, so $U(d)$ is a finite direct product of unitary groups, which acts by conjugation on $M(d)$ and on $\mr(d)$.

We use $U(d)$ to generate all projections $p_K$ (with $K\subset H$ of the form $\oplus_v K_v$ as above) using finitely many $K$'s.
For this, let $D$ denote the set of all dimension vectors $f=(fv)_{v\in\Lambda^0}$ satisfying $fv\leq dv$ for all $v\in\Lambda^0$ and $0\neq f\neq d$.
Note that $D$ is finite. For each $f\in D$ and each $v \in \Lambda^0$, let $K(f)_v := \C^{fv} \subset \C^{dv} = H_v$ (i.e.~$K(f)_v=\text{span}\{e_1,\cdots,e_{fv}\}$ where $\{e_1,\cdots,e_{dv}\}$ is the standard basis of $\C^{dv}$)
and define $K(f) = \bigoplus_v K(f)_v \subset H$.
Then the union of orbits $$\{u(K(f)) : u\in U(d)\text{ and } f\in D\}$$ is precisely $\{K' \le H :\ K' = \bigoplus_v (K' \cap H_v)\}$.
We have the usual conjugating relation
$$
    u p_{K(f)} u^*=p_{u(K(f))}.
$$
The module $\pi$ is reducible if and only if there exists $f\in D$ and $u\in U(d)$ such that
$$
    p_{K(f)}^\perp u \pi(a_\lambda) u^* p_{K(f)} =0 \text{ for all } \lambda\in \Lambda^1.
$$
We claim that the set of reducible modules is closed in $\mr(d)$.
Indeed, fix $f\in D$ and write
\[
    \Red(f) := \{\pi\in \mr(d) :\ p_{K(f)}^\perp \pi(a_\lambda) p_{K(f)} =0 \text{ for all } \lambda\in \Lambda^1\}.
\]
Then $\Red(f)$ is closed by continuity of matrix multiplication. Consider the map
$$
    U(d)\times \cup_{f\in D}\Red(f)\to \mr(d), (u, \pi)\mapsto u\pi(\cdot) u^*
$$
This map is smooth, and in particular continuous. Moreover, $\cup_{f\in D} \Red(f)$ is a finite union of closed sets, and hence closed,
and therefore compact since $\mr(d)$ is compact.
Hence $U(d)\times \cup_{f\in D}\Red(f)$ is compact. Therefore its image under the continuous map described above is compact,
and in particular closed. Hence the set of reducible module is closed in $\mr(d)$. Thus $\Irr(d)$ is open.

\begin{proposition} \label{prop:irrep-manifold}
The set of irreducible matrix modules $\Irr(d)$ is open in $\mr(d)$. If it is nonempty, then it is a smooth submanifold of dimension
$$
    \dim(\Irr(d)) = \dim(\mr(d)) = \sum_{v\in\Lambda^0} dv(2D(v)-dv) = 2\sum_{\lambda \in \Lambda^1} ds(\lambda) \cdot dr(\lambda) - \sum_{v \in \Lambda^0} dv^2,
$$
where $D(v) = \sum_{\lambda \in s(\Lambda^1)}ds(\lambda)$.
\end{proposition}

\subsection{Moduli spaces of irreducible classes of representations of \texorpdfstring{$\cOL$}{OΛ}}

\begin{center}
\textbf{From now on, $d$ is a finite dimension vector for which $\Irr(d)$ is nonempty.}
\end{center}

\subsubsection{Free proper Lie group action on the irreducible modules}
Consider the group $U(d):=\bigoplus_v U(dv)$ equal to the direct sum over $v\in\Lambda^0$ of the unitary groups of the $H_v$ (of which only finitely many are nontrivial).
Let $\PU(d):=U(d)/S_1$ be $U(d)$ modulo the circle group.
This is a compact (Lie) group that acts by conjugation on the set $\mr(d)$ of $d$-dimensional matrix modules of $\cAL$.
Since unitary conjugation preserves irreducibility, we obtain an action $\PU(d)\act \Irr(d)$.
This is a \emph{Lie group} action (i.e.~the action is smooth).
The orbit space $\Irr(d)/PU(d)$ is by definition in bijection with the irreducible classes.
We construct a manifold structure on it using a standard result from the literature.
We first prove that $\PU(d)\act \Irr(d)$ is \emph{free}.

\begin{proposition}\label{prop:action}
The action
$$\PU(d)\times M(d)\to M(d), (u,m)\mapsto (u_{s(\lambda)} m_{\lambda} u_{r(\lambda)}^*)_{\lambda}$$
of the compact (Lie) group $$\PU(d):=\Big(\bigoplus_{v\in\Lambda^0}U(dv)\Big)/S_1$$
on the matrix space $$M(d):=\bigoplus_{\lambda\in \Lambda^1} M_{ds(\lambda),dr(\lambda)}(\C)$$
restricts into a smooth, proper and free action
$$\PU(d)\times \Irr(d)\to \Irr(d)$$
on the subspace $\Irr(d)$ of irreducible $\cAL$-matrix modules of dimension $d$.
\end{proposition}
\begin{proof}
Since matrix multiplication and involution are smooth, the action is smooth.
In particular, it is continuous. Since $\PU(d)$ is compact, we deduce that the action is proper.
It remains to prove freeness.

Consider $(A_\lambda)_\lambda \in \Irr(d)$ and $u\in U(d)$ such that $A_\lambda = u_{s(\lambda)} A_\lambda u_{r(\lambda)}^*$ for all $\lambda \in \Lambda$.
Suppose that $v \in \Lambda^0$ satisfies $u_v \not\in \C 1_{H_v}$. Since $u_v$ is a unitary matrix we can choose a basis for $H_v$ with
respect to which $u_v$ is diagonal with at least two distinct eigenvalues.
For each $z$ in the spectrum $\operatorname{spec}(u_v)$ of $u_v$, let $H_{v, z}$ be the corresponding eigenspace of $u_v$. So
$H_v = \bigoplus_{z \in \operatorname{spec}}(u_v) H_{v,z}$.
Fix $z \in \operatorname{spec}(u_v)$. We claim that the submodule $K$ of $H$ generated by $H_{v,z}$ is a proper submodule of $H$.

The module $K$ is nontrivial since it contains $H_{v,z}$. Fix $x \in \operatorname{spec}(u_v)\setminus \{z\}$.
We claim that $H_{v,x}$ is orthogonal to $K$. For this, we prove that if $\eta\in H_{v,z}$ and $\lambda\in \Lambda$,
then $A_\lambda \eta \in H_{v,x}^\perp.$
Indeed, fix $\lambda \in \Lambda$ and $\zeta\in H_{v,x}$.
If $r(\lambda)\neq v$, then by definition of the action $A_\lambda\eta=0$.
If $s(\lambda)\neq v$, then by definition of the action $A_\lambda \eta\in H_{s(\lambda)} \subset H_v^\perp \subset H_{v,x}^\perp$.
So suppose that $v=r(\lambda)=s(\lambda).$ By assumption $u_v A_\lambda u_v^* = A_\lambda$. Thus $A_\lambda(H_{v,z})\subset H_{v,z}$.
We have proven that $K\subset H$ was a proper nonzero submodule, a contradiction since $H$ is irreducible.

Hence, for each $v\in \Lambda^0$ there exists $z_v\in S_1 \subset \C$ such that $u_v=z_v \id_{H_v}$.
We show that $z_v = z_w$ for all $v,w$. For this, fix $\lambda \in \Lambda$ with $A_\lambda\neq 0$.
Then
$$
    A_\lambda = u_{s(\lambda)} A_\lambda u_{r(\lambda)}^* = z_{s(\lambda)} \cdot \overline{z_{r(\lambda)}} A_\lambda.
$$
Hence, $z_v=z_w$ whenever there exists $\lambda \in w\Lambda v$ such that $A_\lambda\neq 0$.
Since $H$ is irreducible, if $H_v\neq \{0\}\neq H_w$, then there is a path $\lambda\in\Lambda$ satisfying $A_\lambda\neq 0$ and $\{r(\lambda),s(\lambda)\}=\{v,w\}$.
Hence, $z_v=z$ is constant on $\{v\in\Lambda^0:\ H_v\neq \{0\}\}$ and $u = z I(d)$.
\end{proof}

\subsubsection{Manifold structure on the irreducible classes of modules}

Using Proposition \ref{prop:action} and a standard result on orbit spaces (see \cite[Theorem 21.10]{Lee12}) we deduce that the orbit space of $\PU(d)\act \Irr(d)$ is a topological manifold of dimension $\dim(\Irr(d))-\dim(\PU(d))$ which admits a unique (real) smooth structure such that the canonical map $\Irr(d)\onto \Irr(d)/\PU(d)$ is a smooth submersion.
From now on we equip the quotient manifold $\Irr(d)/PU(d)$ with this differential structure. We compute its dimension:
$$\dim(\Irr(d))=\dim(\mr(d))=\sum_{v\in \Lambda^0} dv(2D(v)-dv) \text{ with } D(v):=\sum_{\lambda\in v\Lambda^1}ds(\lambda)$$
and, writing $U(dv)$ stands for the unitary group of $\C^{dv}$,
$$
    \dim(\PU(d))=\dim(U(d)) - 1 = \Big(\sum_{v\in\Lambda^0} \dim(U(dv)) \Big) -1=\Big(\sum_{v\in\Lambda^0} dv^2 \Big)-1.
$$
Hence
\begin{align*}
\dim(\Irr(d)/\PU(d)) & = \dim(\Irr(d)) - \dim(PU(d))\\
& = \sum_{v\in \Lambda^0} dv(2D(v)-dv) - \left(\sum_{v\in\Lambda^0} dv^2 \right) +1\\
& = \Big(\sum_{v\in \Lambda^0} 2dv(D(v)-dv)\Big) +1.
\end{align*}

As always, recall that $D(v):=\sum_{\lambda\in v\Lambda^1}ds(\lambda)$.

\begin{proposition}
The set of irreducible classes of $d$-dimensional $\cAL$-modules is either empty or admits a bijection with the smooth manifold $\Irr(d)/\PU(d)$ of dimension
$$
    1+\sum_{v\in \Lambda^0} 2dv(D(v)-dv).
$$
\end{proposition}

\subsubsection{Moduli spaces of representations of \texorpdfstring{$\cOL$}{O_Lambda}}

We may now easily deduce our main result: Theorem \ref{theo:B}.
Indeed, let $\Irr(d)$ be the set of irreducible matrix modules of $\cAL$ with dimension vector $d$, and let $X(d)$ be the set of irreducible representations of $\cOL$ with Pythagorean dimension vector $d$.
By Proposition~\ref{prop:decomposition}, the image $\Pi(\Irr(d))$ of $\Irr(d)$ in $\Rep(\cOL)$ is contained in the set of irreducible representations of $\cOL$. Since irreducible modules are full, we have $\Pdim(\Pi(\pi)) = d$ for all $\pi \in \Irr(d)$, and hence $\Pi(\Irr(d)) \subset X(d)$. Consider unitary-equivalence map $q_{\text{u.e.}} : X(d)\onto \Spec(d)$. Then we obtain a map $q_{\text{u.e.}} \circ \Pi : \Irr(d)\to \Spec(d)$.
By Proposition \ref{prop:decomposition}(5)~and~(6), this map factorises through a map $\Irr(d)/\PU(d)\to \Spec(d)$, which is bijective by definition of Pythagorean dimension.

\section{Explicit examples}\label{sec:example}

\subsection{Cuntz-Dixmier \texorpdfstring{C$^*$}{C*}-algebras}
Take the $1$-graph with one vertex $v$ and $n$ loops $1, \dots, n$.  The algebra $\cOL=\cO_n$ is then the universal C$^*$-algebra generated by $n$ isometries having orthogonal codomains. For $n=2$ this was first studied by Dixmier, who proved that this C$^*$-algebra admits a quotient that is separable simple and purely infinite \cite{Dixmier64}.
Subsequently, Cuntz performed a deep analysis of the C$^*$-algebras $\cO_n$ and proved among other things that they are all simple, so that Dixmier's quotient map is in fact the identity, and $\cO_n$ itself is separable, simple and purely infinite \cite{Cuntz77}.

The non-self-adjoint algebra $\cAL=\cA_n$ is generated by operators $a_1,\cdots,a_n$ such that $\oplus_i a_i$ is an isometry from $H$ to $H^n$.
Any $n$-tuple of complex numbers $(z_1,\cdots,z_n)$ satisfying $\sum_i |z_i|^2=1$---that is, any element of the unit sphere in $\C^n$---defines an irreducible $\cA_n$-module $H_z = \C$,
the action of the generators being given by $a_i \cdot w = z_i w e_i \in \C^n$. Applying our functor $\Pi$ yields an irreducible representation $\pi_z : \cO_n \to B(L)$
The $\pi_z$ are pairwise inequivalent and constitute a set of representatives for $\Spec(1)$: the piece of the spectrum of $\cO_n$ corresponding to Pythagorean dimension one.
Hence $\Spec(1)$ is in natural bijection with the real unit sphere of dimension $2n-1$.

Now, for any $n\geq 2$ and any $d\geq 1$ we can find a $d$-dimensional $n$-tuples of matrices $\{A_i\}_i$ such that $\oplus_i A_i$ is isometric and
the $A_i$'s have no nonzero proper invariant subspace of $\C^d$. Hence $\Irr(d)\subset \mr(d)$ is nonempty and forms a submanifold of dimension $d(2D-d) = (2n-1)d^2.$
Thus $\Spec(d)$ is (in bijection with) a manifold of dimension
$$
    2d(D-d)+1=2d(nd-d)+1=2(n-1)d^2+1
$$
as proved in \cite[Theorems E.3 and F.2]{Brothier-Wijesena24}.

\subsection{Cuntz--Krieger \texorpdfstring{C$^*$}{C*}-algebras}\label{sec:example-CK}
Given a finite $n$ by $n$ matrix $A$ with entries in $\{0,1\}$, there is a
corresponding directed graph $\Lambda_A$, with $\Lambda_A^0 = \{1, \dots, n\}$,
and $i \Lambda_A^1 j = a_{ij}$ for all $i,j \in \Lambda_A^0$. The matrix $A$ is called
the \emph{adjacency matrix} of $\Lambda_A$. Cuntz and Krieger
defined the associated C$^*$-algebras $\cOL=\cO_A$ \cite{Cuntz-Krieger80}.
Kumjian, Pask, Raeburn and Renault extended this construction (and many
structural results on them) to row-finite directed graphs in
\cite{Kumjian-Pask-Raeburn-Renault97, Kumjian-Pask-Raeburn98} (see also
Raeburn's book \cite{Raeburn05}).

\subsubsection{Complete graph} \label{subsec:complete-graph}

Consider the complete directed graph $\Lambda$ with $2$ vertices: So $\Lambda^0 = \{1, 2\}$ and $\Lambda^1 = \{a_{ij} : i,j \in \{1,2\}$,
with $r(a_{ij}) = i$ and $s(a_{ij}) = j$.
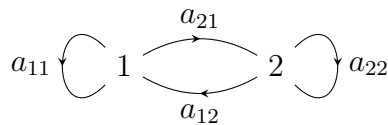
\[
    \begin{tikzpicture}[decoration={markings, mark=at position 0.5 with {\arrow{stealth}}}]
    \node (1) at (0,0) {1};
    \node (2) at (2,0) {2};
    \draw[postaction={decorate}] (1) .. controls +(-1, 1) and +(-1,-1) .. (1) node[pos=0.5, anchor=east] {$a_{11}$};
    \draw[postaction={decorate}] (2) .. controls +(1, 1) and +(1,-1) .. (2) node[pos=0.5, anchor=west] {$a_{22}$};
    \draw[postaction={decorate}, bend left] (1) to node[pos=0.5, anchor=south] {$a_{21}$} (2);
    \draw[postaction={decorate}, bend left] (2) to node[pos=0.5, anchor=north] {$a_{12}$} (1);
    \end{tikzpicture}
\]

The adjacency matrix of $\Lambda$ is
$$\begin{pmatrix} 1 & 1 \\ 1 &1 \end{pmatrix}.$$
An $\cAL$-module $H$ consists of a decomposition $H = H_1\oplus H_2$, and operators $A_{ij}:H_i\to H_j$ such that,
writing $P_i : H \to H_i$ for the orthogonal projections,
$$
    A_{11}^*A_{11}+A_{12}^*A_{12}=P_1 \text{ and } A_{21}^*A_{21}+A_{22}^*A_{22}=P_2.
$$
Fix $q \ge 1$ and let $d = (q, q) \in \N^2$. We claim that
$\Irr(d)$ is nonempty. This implies that $\Spec(d)$ is nonempty and
forms a smooth manifold of dimension $1+2q^2$.

To see that $\Irr(d)$ is nonempty, fix orthonormal bases $(f_1,\cdots,f_q)$ and
$(g_1,\cdots,g_q)$ for $H_1$ and $H_2$, respectively.
Let $S_1$ denote the unit complex circle, fix $\omega, \epsilon \in \T^q$, each with
pairwise distinct entries, and define
\[
A_{11} f_i = \dfrac{\omega_i}{\sqrt2} f_i\quad\text{ and }A_{22} g_j = \dfrac{\epsilon_j}{\sqrt2} g_j.
\quad\text{ for $i,j \le q$}.
\]
Define $A_{12} : H_1 \to H_2$ and $A_{21} : H_2 \to H_1$ by
$$
    A_{12}f_i:= \dfrac{1}{\sqrt{2q}}\Big(\sum_{r=1}^q g_r - 2 g_i\Big)
    \quad\text{ and }\quad
    A_{21}g_j := \dfrac{f_j}{\sqrt 2},
$$
so
\[
    (A_{12} f_i)_j = \frac{(-1)^{\delta_{i,j}}}{\sqrt{2q}} \quad\text{ and }\quad (A_{21} g_j)_i = \frac{\delta_{i,j}}{\sqrt{2}}
\]
This makes $H$ into an $\cAL$-module with dimension vector $d$.

Suppose that $K\subset H$ is a nontrivial submodule. then $K=K_1\oplus K_2$
with $K_i\subset H_i,i=1,2$. Since $K$ is closed under $A_{11}$, which has
distinct eigenvalues corresponding to the canonical basis, we must have
$K_1 = \operatorname{span}\{f_i : i \in I\}$ for some $I \subset \{1, \dots, q\}$.
Likewise, $K_2 = \operatorname{span}\{g_j : j \in J\}$ for some $J \subset \{1, \dots, q\}$.
If $K_1 = 0$ then $A_{21}(K_2) = 0$, and since $A_{12} + A_{22}$ is an isometry on $K_2$,
this implies that $A_{22}$ is an isometry on $K_2$, forces $K_2 = 0$  because
$A_{22}^* A_{22} = \frac{1}{2} 1_{H_2}$. Since we have assumed that $K$ is nontrivial,
we deduce that $K_1$ is nonzero. So $I \neq 0$; fix $i \in I$.
Then $\langle A_{12}(f_i), g_j\rangle = 1 - 2\delta_{i,j} \neq 0$ for all $j$.
Since $K_2 = \operatorname{span}\{g_j : j \in J\}$, we obtain $J = \{1, \dots, q\}$,
and so $K_2 = H_2$. Now $A_{21}(H_2)=H_1$. So $K_1=H_1$ and thus
$K=H$. Hence $H$ is irreducible and therefore $\Irr(d)\neq
\emptyset.$

\subsubsection{A graph admitting few irreducible representations with finite Pythagorean dimension}

Consider the graph $\Lambda$ with vertex set $\{1,2,3\}$ and edge set $\{(11), (21),(31)\}$ with $(ij)$ being an edge from $j$ to $i$.
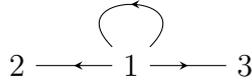
\[
    \begin{tikzpicture}[decoration={markings, mark=at position 0.5 with {\arrow{stealth}}}]
    \node (1) at (2,0) {1};
    \node (2) at (0.5,0) {2};
    \node (3) at (3.5,0) {3};
    \draw[postaction={decorate}] (1) .. controls +(1, 1) and +(-1,1) .. (1);
    \draw[postaction={decorate}] (1)--(2);
    \draw[postaction={decorate}] (1)--(3);
    \end{tikzpicture}
\]
An $\cAL$-module $H$ decomposes as $H=H_1\oplus H_2\oplus H_3$ and is described by 3 operators $A_{i1}:H_i\to H_1$ for $1\leq i\leq 3$.
Assume that $H$ is finite dimensional and irreducible.
Then $H_2\oplus H_3$ clearly contains no nonzero submodule, so is residual.
Since $H$ is irreducible, this forces $H=H_1$.
Now $A_{11} : H_1 \to H_1$ is an isometry, and hence a unitary.
Since $H$ is irreducible, $A_{11}$ has no invariant subspace, and this forces $\dim(H_1)\leq 1$ and $A_{11}=z$ for some $z \in S_1$.
This defines a module $M_{z}=(H,A_{i1})$ of P-dimension $(1,0,0)$, with $H=\C$, $A_{11}=z$ and $A_{21}=A_{31}=0$.
The $M_z$ are all irreducible (since they are of dimension one) and pairwise inequivalent.
Moreover, every finite P-dimensional irreducible module has this form up to conjugacy.
Using our main theorem yields that the spectrum of $\cOL$ partitions into
$$\Spec(1,0,0)\sqcup \Spec(\infty).$$
The first piece is in bijection with $S_1$ and corresponds to the irreducible classes of representations of $\cOL$
of finite Pythagorean dimension; the second piece corresponds to the classes with infinite Pythagorean dimensions.

For this specific example, we can say more about $\Spec(\infty)$, though our general theorems are
silent on its structure. Consider an irreducible representation $L$ of $\cOL$ of infinite P-dimension.
By essential surjectivity of $\Pi$ there exists a module $H$ so that $\Pi(H)\simeq L$.
Note that $H_1\subset H$ is a submodule. Irreducibility implies that $H_1$ is either trivial or complete.
If $H_1=\{0\}$ then $H_2=H_3=\{0\}$ since $a_{j1} : H_j \to H_1$ is an injection for each $j$.
Hence, $H_1$ is complete and thus $\Pi(H_1)=\Pi(H)$. So we may assume that $H=H_1$.
The Wold decomposition applied to $A_{11}:H_1\to H_1$ implies that $A_{11}$ is conjugate to a direct sum
of copies of the unilateral shift operator $S:\delta_n\mapsto \delta_{n+1}$ of $\ell^2(\N)$, and a unitary.
Furthermore, every invariant subspace of $A_{11}$ yields a submodule of $H$.
Irreducibility of $\Pi(H)$ implies that $A_{11}$ either is conjugate to $S$ or is a unitary.
The latter case reduces to the situation of the preceding paragraph, so $H$ must be 1-dimensional contradicting
that $L$ has infinite P-dimension.

Consequently, $A_{11}$ is conjugate to $S$. For each $N$, write $(11)^N$ for the path in $\Lambda^N$ obtained as the $N$-fold concatenation of the loop $(11)$. Then $\Pi(H)$ is the closed linear span of the vectors $[(11)^N,\delta_n]$.
The formula $[(11)^N,\delta_n]\mapsto \varepsilon_{n-N}$ provides a unitary transformation $\varphi:\Pi(H)\to \ell^2(\Z)$
where $\varphi_k$ denotes the standard basis of $\ell^2(\Z)$.
Moreover, this transformation conjugates the operator $X_{11}$ of $\cOL$ with the shift operator of $\ell^2(\Z)$ which is a unitary.
Any invariant subspace of $X_{11}$ defines a subrepresentation of $\Pi(H)$. There are many nontrivial invariant subspace for the unitary shift operator of $\ell^2(\Z)$.
Hence, $\Pi(H)$ is reducible. So $\cOL$ admits no irreducible representation of infinite P-dimension.
We conclude that the whole spectrum of $\cOL$ is equal to $\Spec(1,0,0)$ and is in bijection with the circle $S_1$.

This is as expected: by direct computation, there is an isomorphism $\cOL \cong M_3(C(S_1))$
that carries $x_{(j1)}$ to the (constant) matrix unit $\theta_{j1}$ for each of $j = 2,3$, and carries $x_{(11)}$ to the matrix
$\operatorname{diag}(z\mapsto(z, 0, 0))$. So the inclusion of $C^*(\{x_{(11)}\})$ in $\cOL$ is a Morita equivalence from
$C(S_1)$ to $\cOL$, and therefore induces a homeomorphism of spectra from $S_1$ to the spectrum of $\cOL$.

\subsubsection{Maps between moduli spaces of isomorphic C$^*$-algebras}

We investigate the maps between our moduli spaces arising from known isomorphisms
between Cuntz--Krieger algebras of different graphs.

\textbf{An isomorphism between Cuntz--Krieger C$^*$-algebras.}

Consider again the complete graph $\Lambda$ of Section \ref{subsec:complete-graph} with vertex set $\{1,2\}$ and edge set $\{(ji): 1\leq i,j\leq 2\}$.
For short, the concatenation of edges is denoted as $kji := (kj)(ji)$. More generally, we consider binary words $w = (w_n\dots w_1) \in \Lambda$ to denote the concatenation of $n-1$ edges $(w_nw_{n-1}) \dots (w_3w_2)(w_2w_1)$. Additionally, consider the rose with two petals $\Lambda_2$ which has a single vertex $u$ and two loops $e_1, e_2$.

It is well-known that $\cOL \cong \cO_2 = \cO_{\Lambda_2}$. To see this, denote the canonical generators of $\cOL$ as in Definition~\ref{def:O} by $\{y_\alpha : \alpha \in \Lambda\}$.
Denote the generators of $\cO_2 = \cO_{\Lambda_2}$ by $\{x_\beta : \beta \in \Lambda^2\} = \{x_\beta : \beta \in \{0,1\}^*\}$. For $i = 1,2$, let $p_i := x_ix_i^*$.
Direct computation shows that the elements
\[
X_i := y_{i1} + y_{i2} \in \cOL,\quad i = 1,2
\]
determine a representation $\cO_2$, and that the elements
\[
Y_{ij} := x_ip_j \in \cO_2,\quad 1 \le i,j \le 2
\]
determine a representation of $\cOL$. To see why, consider a representation $\cOL \act L$.
Then $L$ decomposes as $L_1 \oplus L_2$, and more generally $L_w$ decomposes as $L_{w1} \oplus L_{w2}$ for a binary word $w$. So
$y_w$ acts as a partial isometry from $L_{sw}$ to $L_{w}$. Consequently, $y_{i1}+y_{i2}$ acts as an \textit{isometry} from $L$ to $L_i$.
Thus, this defines an action of $\cO_2$ on $L$ as claimed. Since $x_j^*$ acts as a partial isometry from $L_j$ to $L$, each $x_ix_jx_j^* = x_ip_j$
acts as a partial isometry from $L_j$ to $L_{ij}$, so that the $Y_{ij}$ determine a representation of $\cOL$.

By the universal properties of $\cOL$ and $\cO_2$ there are homomorphisms $\theta : \cO_2 \to \cOL$ and $\tau : \cOL \to \cO_2$ such
that $\theta(x_i) = X_i$ and $\tau(y_{ij}) = Y_{ij}$. We have
\[
\theta(\tau(y_{ij}))
    = \theta(x_i x_j x^*_j)
    = (y_{i1} + y_{i2})(y_{j1} + y_{j2})(y_{j1} + y_{j2})^*
    = (y_{i1} + y_{i2}) y_j
    = y_{ij}
\]
and
\[
\tau(\theta(x_i)) = \tau(y_{i1} + y_{i2}) = x_i (x_1 x^*_1 + x_2 x^*_2) = x_i.
\]
So $\theta$ and $\tau$ are mutually inverse, and therefore isomorphisms. Henceforth we just write $\theta^{-1}$ for $\tau$.

\textbf{Equivalent representations from isomorphic Cuntz--Krieger algebras.}
The isomorphism $\theta : \cO_2 \to \cOL$ induces a bijection $(\sigma, L) \mapsto (\sigma \circ \theta, L)$ from representations of $\cOL$ to representations of $\cO_2$ that preserves irreducibility and equivalence classes of representations.
Given that the Pythagorean functor $\Pi: \Mod(\cAL) \rightarrow \Rep(\cOL)$ is essentially surjective, it is natural to ask: given an $\cAL$-module $\pi$, how do we construct an $\cA_2$-module $\tau$ such that $\Pi(\tau)$ is equivalent to $\Pi(\pi) \circ \theta$, and vice versa.
For this, first, note that $\theta$ maps the non-self-adjoint subalgebra $\cA_{\Lambda_2}$ of $\cO_2$ generated by $x_1^*$ and $x_2^*$ into the corresponding non-self-adjoint subalgebra $\cAL$ of $\cOL$.
However, $\theta^{-1}$ does not map $\cAL$ into $\cA_2$. So we start with the easier question of how representations of $\cOL$ map to representations of $\cO_2$.

Fix an $\cAL$-module $\pi: \cAL \act H$ with finite dimension vector $d$. Let $\sigma : \cOL \act L$ be the corresponding representation of $\cOL$. Let $\tau := \sigma \circ \theta : \cO_2 \act L$.
We construct the $\cA_2$-module whose corresponding representation of $\cO_2$ is equivalent to $\tau$.
Write $b_{ij}$ for the generators of $\cAL$ and $a_k$ for the generators of $\cA_2$ for $1 \leq i,j,k \leq 2$.
Recall that $H$ decomposes over $\Lambda^0 = \{1,2\}$ as $H = H_1 \oplus H_2$, and the operators $B_{ij} := \pi(b_{ij})$ determine bounded linear operator from $H_i$ to $H_j$.
For $(\xi,\eta) \in H_1 \oplus H_2 = H$, define $A_1\xi, A_2\xi \in H = H_1 \oplus H_2$ by
\begin{equation}\label{eq:A2-mod from A-mod}
    A_1(\xi, \eta) = (B_{11}(\xi), B_{12}(\xi)) \textrm{ and } A_2(\xi, \eta) = (B_{21}(\eta), B_{22}(\eta)).
\end{equation}

Identifying elements of $B(H) = B(H_1 \oplus H_2)$ with block matrices $(T_{ij})_{i,j = 1,2}$ such that each $T_{ij} \in B(H_j, H_i)$, the operators $A_1$ and $A_2$ have block decompositions
\[A_1 = \begin{pmatrix}
	B_{11} & 0 \\ B_{12} & 0
\end{pmatrix},\
A_2 = \begin{pmatrix}
	0 & B_{21} \\ 0 & B_{22}
\end{pmatrix}.
\]
By construction, $\pi' = \pi \circ \theta$, so $\pi'$ is an $\cA_2$-module, determining a representation $\sigma' : \cO_2 \act L'$.

We construct a map between $L$ and $L'$ that intertwines the two actions.
Recall that a typical spanning element of $L'$ has the form $[\lambda, (\xi, \eta)]$ for some $\lambda \in \Lambda_2$ and $(\xi, \eta) \in H_1 \oplus H_2$.
We claim that there is an invertible bounded linear operator $\alpha: L' \rightarrow L$ such that
\[[\lambda, (\xi, \eta)] \mapsto [\lambda 1, \xi] + [\lambda 2, \eta],\]
and that $\alpha^{-1} : L' \rightarrow L$ satisfies
\[[\lambda 1, \xi] \mapsto [\lambda, (\xi, 0)] \textrm{ and } [\lambda 2, \eta] \mapsto [\lambda, (0, \eta)].\]
To see this define a sesquilinear form on
the space $L^{00}$ of Section~\ref{sec:first construction} by
\[
[(\lambda, (\xi,\eta)), (\mu, (\zeta,\tau))] = \langle [\lambda 1, \xi] + [\lambda 2, \eta], [\mu 1, \zeta] + [\mu 2, \tau]\rangle_{L'}.
\]
For $\lambda \in \Lambda_2$ and $i \in \{1,2\}$ and $(\xi_1, \xi_2) \in H$, we have
\begin{align*}
[(\lambda, (\xi_1,\xi_2)), (\lambda i, A_i(\xi_1,\xi_2))]
    &= [(\lambda, (\xi,\eta)), (\lambda i, (B{i1}\xi_i, B_{i2}\xi_i))]\\
    &= \sum_{j,k = 1}^2 \langle [\lambda j, \xi_j], [\lambda i k, B_{ik}\xi_i]\rangle\\
    &= \sum_{k = 1}^2 \langle [\lambda i, \xi_i], [\lambda i k, B_{ik}\xi_i]\rangle\\
    &= \langle [\lambda i, \xi_i], [\lambda i, \xi_i]\rangle.
\end{align*}
Applying this four times shows that
\[
\Big[(\lambda, (\xi_1,\xi_2)) - \sum_i (\lambda i, A_i(\xi_1,\xi_2)), (\lambda, (\xi_1,\xi_2)) - \sum_i (\lambda i, A_i(\xi_1,\xi_2))\Big] = 0
\]
This and that the vector-space operations in $L'$ are in the second coordinate show that the kernel $N$ of the quotient map $L^{00} \to L^0$ is contained in the kernel of $[\cdot,\cdot]$, and so $[\cdot, \cdot]$ descends to a sesquilinear form $\llangle \cdot, \cdot\rrangle$ on $L^0$. We use the characterisation of the inner product on $L$ from Proposition~\ref{prp:L inner prod} to show that this form agrees with the inner product on $L^0$. For this, fix $\lambda,\mu \in \Lambda_2$ and $(\xi,\eta), (\zeta,\tau) \in H$.
If $r(\lambda) \not= r(\mu)$, then $r(\lambda i) \not= r(\mu j)$ for all $i,j$ and so
$\llangle [\lambda, \xi], [\mu,\eta]\rrangle = \langle [\lambda 1, \xi] + [\lambda 2, \eta], [\mu 1, \zeta] + [\mu 2, \tau]\rangle_{L'} = 0$. If $d(\lambda) = d(\mu)$, then
$d(\lambda i) = d(\mu j)$ for all $i,j$ and we have $\delta_{\lambda i, \mu i} = \delta_{\lambda,\mu} \delta_{i,j}$. So
\begin{align*}
\llangle [\lambda, \xi], [\mu,\eta]\rrangle
    &= \langle [\lambda 1, \xi] + [\lambda 2, \eta], [\mu 1, \zeta] + [\mu 2, \tau]\rangle_{L'} \\
    &= \delta_{\lambda 1, \mu 1} \langle\xi,\zeta\rangle
        + \delta_{\lambda 1, \mu 2} \langle\xi,\tau\rangle
        + \delta_{\lambda 2, \mu 1} \langle\eta,\zeta\rangle \delta_{\lambda 2, \mu 2} \langle\eta,\tau\rangle\\
    &= \delta_{\lambda,\mu} \langle(\xi,\eta), (\zeta,\tau)\rangle_H.
\end{align*}
Condition~(3) of Proposition~\ref{prp:L inner prod} is vacuous because $\Lambda$ is a $1$-graph with no sources.

Consequently the uniqueness assertion of Proposition~\ref{prp:L inner prod} shows that $\llangle \cdot, \cdot\rrangle = \langle \cdot, \cdot\rangle_{L}$. Therefore, there is an isometric linear map $\alpha : L \to L'$ as claimed. Each $\lambda \in \Lambda$ has form $\lambda = \lambda' i$ for some $\lambda' \in \Lambda_2$ and $i \in \{1,2\}$, and we have $[\lambda 1, \xi] = \alpha([\lambda, (\xi,0)])$ and $[\lambda 2, \xi] = \alpha([\lambda, (0, \xi)])$, so $\alpha$ is surjective and hence an isomorphism; the same calculation shows that $\alpha^{-1}$ satisfies the given formula.

For $\lambda, \mu \in \Lambda_2$ with $r(\lambda) = s(\mu)$ we have
\[\alpha(\sigma'(x_\mu)[\lambda, (\xi, \eta)]) = \alpha([\mu\lambda, (\xi\, \eta)]) = [\mu\lambda 1, \xi] + [\mu\lambda 2, \eta]\]
and
\[\tau(x_\mu)\alpha([\lambda, (\xi, \eta)]) = \sigma(y_{\mu1} + y_{\mu2})([\lambda 1, \xi] + [\lambda 2, \eta]) = [\mu\lambda 1, \xi] + [\mu\lambda 2, \eta].\]
Hence, $\alpha$ intertwines the operators $\sigma'(x_\mu)$ and $\tau(x_\mu)$; and since it is unitary, it also intertwines the adjoints $\sigma'(x_\mu^*)$ and
$\tau(x_\mu^*)$.
Hence $\sigma'$ is equivalent to $\tau$ as required.

Now fix an $\cA_2$-module $\pi' : \cA_2 \act H$ with finite dimension vector $d$. We construct an $\cAL$-module $\pi$ so that the resulting representations of $\cO$ are equivalent. For this, we first dilate $\pi'$ to a larger non-full $\cA_2$-module. Continue to write $a_i$ for the generators of $\cA_2$, $b_{ij}$ for the generators of $\cAL$, and $\Pi(\pi') =: \sigma': \cO_2 \act L'$. Let $\pi'_d : \cA_2 \act H \oplus H$ be the $\cA_2$-module such that
\[
    \pi'_d(a_1)(\xi, \eta) = (\pi'(a_1)\xi, \pi'(a_2)\xi) \textrm{ and } \pi'_d(a_2)(\xi, \eta) = (\pi'(a_1)\eta, \pi'(a_2)\eta)
\]
for all $(\xi, \eta) \in H \oplus H$. Then $\pi'_d$ is the submodule $H_{\vee} \subset L'$ (see Section \ref{subsec:scond-constr-Hilb} for notation) obtained by regarding $L'$ as an $\cA_2$-module. Hence, $\Pi(\pi'),\ \Pi(\pi'_d)$ are equivalent representations. Define $\pi: \cAL \act H \oplus H$ to be the $\cAL$-module such that
\[
    \pi(b_{11}) = \pi(b_{21}) := \pi'(a_1) \textrm{ and } \pi(b_{12}) = \pi(b_{22}) := \pi'(a_2)
\]
and denote the corresponding representation by $\sigma: \cOL \act L$. Observe that the $\cA_2$-module obtained from $\pi$ using the formula~\eqref{eq:A2-mod from A-mod} is precisely $\pi'_d$. Hence the representations $\sigma',\ \sigma \circ \theta$ are equivalent.

\textbf{Functor between representation categories}
We have now seen how to obtain an $\cA_2$-module from an $\cAL$-module and vice versa. This process is in fact functorial. Specifically, the formulas
\[\Psi(\sigma)(a_i)(\xi_1, \xi_2) = (\sigma(b_{i1})\xi_i, \sigma(b_{i2})\xi_i) \textrm{ for } \xi_i \in H_i,\ i = 1,2,\]
and
\[\Omega(\pi)(b_{ij}) = \rho(a_j), \textrm{ for } 1 \leq i,j \leq 2\]
determine functors
\[\Psi: \Mod(\cAL) \rightarrow \Mod(\cA_2),\ (\sigma, H_1 \oplus H_2) \mapsto (\Psi(\sigma), H_1 \oplus H_2),\ \alpha \mapsto \alpha\]
and
\[\Omega: \Mod(\cA_2) \rightarrow \Mod(\cAL),\ (\pi H) \mapsto (\Omega(\pi), H \oplus H),\ \alpha \mapsto \alpha \oplus \alpha.\]
Composing these with the functors $\Pi$ and $\Sigma$ of Proposition~\ref{prop:functor} and Corollary~\ref{cor:categories} yields
functors between the representation categories of $\cO_2$ and $\cOL$. To lighten notation, we denote these by
\[
    \Psi: \Rep_{\Pfd}(\cOL) \rightarrow \Rep_{\Pfd}(\cO_2),\quad\text{ and }\quad \Omega: \Rep_{\Pfd}(\cO_2) \rightarrow \Rep_{\Pfd}(\cOL).
\]
We have
\[
    \Psi(\sigma) \cong \sigma \circ \theta\quad\text{ and }\quad \Omega(\pi) \circ \theta \cong \pi.
\]
We have $\sigma \cong \Omega(\Psi(\sigma))$ and $\pi \cong \Psi(\Omega(\pi))$ and these isomorphisms are natural, so they determine an equivalence of categories.

\textbf{Mapping between moduli spaces.}
The results of Section~\ref{sec:moduli spaces} provide nontrivial moduli spaces of $\cO_2,\cOL$ for each finite dimension (vector).
Thus, it is of interest how $\Psi$ and $\Phi$ behave as maps between the moduli spaces of $\cOL$ and $\cO_2$.
We have constructed the Pythagorean dimension of a representation of $\cOL$ and the associated piece of spectrum which depends on the choice of the graph $\Lambda$.
To emphasise this dependency we add in this paragraph a superscript $\Lambda$ to avoid confusion.
Recall that $\Lambda_2,\Lambda$ denote the graph with one vertex and two loops and the complete graph with 2 vertices, respectively.
Since $\dim(\Psi(\sigma)) = \dim(\sigma)$ for every $\cAL$-module $\sigma$, and since $\dim(\Omega(\pi)) = 2\cdot\dim(\pi)$ for every $\cA_2$-module $\pi$, we see that
$$
    \Pdim^{\Lambda_2}(\Psi(\rho)) \leq \Pdim^\Lambda(\rho) \leq 2\cdot\Pdim^{\Lambda_2}(\Psi(\rho))
$$
for all representations $\rho \in \ob \Rep(\cOL)$.
We now examine some examples that show that $\Psi$ and $\Omega$ need not preserve Pythagorean dimension on the nose. Throughout the following we identify all finite-dimensional Hilbert spaces $H$ with the concrete spaces $\C^{\dim(H)}$ by fixing a choice of orthonormal basis for each.

\textit{Mapping of $\Spec^{\Lambda_2}(1)$ and $\Spec^{\Lambda}(1,1)$.}
Take $a_1,a_2 \in \C$ satisfying $\vert a_1\vert^2 + \vert a_2\vert^2 = 1$. This gives a one-dimensional $\cA_2$-module $\pi = (a_1,a_2, \C)$.
The space of such modules is in bijection with the real 3-sphere $S_3$ and with $\Spec^{\Lambda_2}(1)$ (the space of irreducible classes of representations of $\cO_2$ with Pythagorean dimension $1$ with respect to the graph $\Lambda_2$).
The associated $\cAL$-module
$$
    \sigma:=\Omega(\pi)=(B_{11},B_{12},B_{21},B_{22},\C\oplus\C)
$$
is by definition 2-dimensional and satisfies
$$
    B_{11}(\xi_1,\xi_2)=(a_1\xi_1,0), B_{12}(\xi_1,\xi_2)=(0,a_2\xi_1), B_{21}(\xi_1,\xi_2) = (a_1\xi_2,0), B_{22}(\xi_1,\xi_2)=(0,a_2\xi_2).
$$
If $a_1\neq 0\neq a_2$ then $\sigma$ is an irreducible $\cAL$-module and thus the associated representation $\Pi(\sigma)$ of $\cOL$ is irreducible with P-dimension $2$ and P-dimension vector $(1,1)$ with respect to the graph $\Lambda$.
We deduce that $\Pi(\sigma)$ has P-dimension $2$ with respect to the complete graph $\Lambda$ while it has P-dimension $1$ with respect to bouquet of two roses $\Lambda_2$.

If $a_2 = 0$ then $\sigma$ is reducible with complete submodule $\C\oplus\{0\}$. In this case we have $\Pdim^{\Lambda}(\Pi(\sigma)) = 1$ and its P-dimension vector is $(1,0)$.
Similarly, if $a_1=0$ then $\Pdim^{\Lambda}(\Pi(\sigma))=1$ and its P-dimension vector is $(0,1)$.
We obtain a bijection
$$
    \Spec^{\Lambda_2}(1)\to X\sqcup \Spec^\Lambda(1,0)\sqcup \Spec^\Lambda(0,1);
$$
on the right, $\Spec^\Lambda(1,0)$ and $\Spec^\Lambda(0,1)$ are the great circles $z_i=0,i=1,2$, and $X$ is the complement $\{(z_1, z_2) : z_1z_2 \not= 0\}$ in the 3-sphere $\{(z_1,z_2)\in\C^2:\ |z_1|^2+|z_2|^2=1\}.$
Indeed, $X$ corresponds to the \emph{diffuse} part of $\Spec^{\Lambda_2}(1)$ and $ \Spec^\Lambda(1,0)\sqcup \Spec^\Lambda(0,1)$ corresponds to the \emph{atomic} part as in \cite[Section 2.3]{Brothier-Wijesena24}.

Conversely, fix an irreducible $\cAL$-module $\sigma$ such that $\Pdim^{\Lambda}(\sigma)=(1,1)$ and $\Pi(\sigma) \notin X$. Then $(\sigma(b_{11})\ \sigma(b_{12})) \neq c(\sigma(b_{21})\ \sigma(b_{22}))$ for any $c \in \C$ (after identifying the operators $\sigma(b_{ij})$ with scalars). We then have
\[\Psi(\sigma) = \left(
\begin{pmatrix}
	\sigma(b_{11}) & 0\\ \sigma(b_{12}) & 0
\end{pmatrix},\
\begin{pmatrix}
	0 & \sigma(b_{21}) \\ 0 & \sigma(b_{22})
\end{pmatrix}, \C^2
\right).\]
Since $\sigma$ is irreducible, so is $\Psi(\sigma)$, so $\Psi$ maps $\Spec^{\Lambda}(1,1)\setminus X$ maps into $\Spec^{\Lambda_2}(2)$.

\textit{Mapping of moduli spaces of higher Pythagorean dimension.}
More generally, consider an irreducible $\cA_2$-module $\pi = (A_1,A_2, \C^d)$ and its corresponding $\cAL$-module $(\sigma, \C^d \oplus \C^d) := \Omega(\pi)$. If $\sigma$ is irreducible then $\Pdim^{\Lambda}(\Pi(\sigma)) = 2 \cdot \Pdim^{\Lambda_2}(\Pi(\pi))$. If $\sigma$ is reducible, then there exist subspaces $\fK_1, \fK_2 \subset \C^d$, with at least one being proper, such that
\[A_1(\fK_i) \subset \fK_1 \textrm{ and } A_2(\fK_i) \subset \fK_2\]
for $i=1,2$. This implies that $\fK_1 \oplus \fK_2$ is a submodule of $\sigma$.
This submodule is complete because $\pi$ is irreducible and hence $\Pi(\pi), \Pi(\sigma)$ are both irreducible. Irreducibility of $\pi$ further implies that $\C^d = \textrm{span}\{\fK_1, \fK_2\}$, and thus $\textrm{ran}(A_1) = \fK_1$ and $\textrm{ran}(A_2) = \fK_2$. Therefore,
\[\Pdim^{\Lambda}(\Pi(\Omega(\pi))) = \textrm{rank}(\pi(a_1)) + \textrm{rank}(\pi(a_2)). \]

We conclude with a concrete example such that $\Pdim^{\Lambda}(\Pi(\Omega(\pi)))$ lies \emph{strictly} between $\Pdim^{\Lambda_2}(\Pi(\pi))$ and $2\cdot \Pdim^{\Lambda_2}(\Pi(\pi))$. Consider the $\cA_2$-module $\pi = (A_1,A_2,\C^3)$ defined by
\[A_1 =
\begin{pmatrix}
	1/\sqrt{3} & 1/3 & 0\\
	1/\sqrt{3} & 1/3 & 0\\
	1/\sqrt{3} & -2/3 & 0
\end{pmatrix}
\textrm{ and }
A_2 =
\begin{pmatrix}
	0 & 1/3 & -2/\sqrt{6}\\
	0 & 1/3 & 1/\sqrt{6}\\
	0 & 1/3 & 1/\sqrt{6}
\end{pmatrix}.
\]
Then direct calculation shows that $\pi$ is irreducible, so $\Pdim^{\Lambda_2}(\pi) = 3$. From the above we have $\Pdim^{\Lambda}(\Pi(\Omega(\pi))) = 4$.

\subsection{Higher-rank graph \texorpdfstring{C$^*$}{C*}-algebras}

Consider the $2$-graph $\Lambda$ with one vertex and $4$ loops $f_1,f_2,g_1,g_2$ with the $f_i$'s of one colour and the $g_j$'s of the other colour,
and with factorisation property satisfying
$$
    f_i g_j = g_j f_i \text{ for all } 1\leq i,j\leq 2.
$$
The algebra $\cAL$ is generated by $a_1,a_2,b_1,b_2$ satisfying $ a_1^*a_1+a_2^*a_2=1=b_1^*b_1+b_2^*b_2$ and $a_ib_j=b_ja_i$ for $1\leq i,j,\leq 2.$
Recall that $\cA_2$ is the non-self-adjoint algebra corresponding to the Cuntz algebra $\cO_2$ associated to the graph with one vertex and two loops.
Then $\cAL \cong \cA_2\otimes \cA_2$ via an isomorphism taking $a_i$ to $A_i \otimes 1$ and $b_j$ to $1 \otimes A_j$.
This isomorphism descends to an isomorphism $\cOL\simeq \cO_2\otimes \cO_2$.
By Kirchberg's absorption theorem, we have $\cOL\simeq \cO_2$ under a non-explicit isomorphism.
Constructing representations of $\cOL$ using the presentation given by $\Lambda$ has the advantage that they restrict naturally to unitary representations of Brin's higher dimensional Thompson group $2V$ \cite{Brin04}.
Indeed, $2V$ embeds in the unitary group of $\cOL$ in a manner to the embedding of the classical Thompson group $V$ in $\cO_2$ via the Birget--Nekrashevych map, see \cite{Birget04,Nekrashevych04}.

\textbf{One-dimensional modules.}
A one-dimensional matrix module over $\cAL$ is a quadruple $(a_1,a_2,b_1,b_2)$ of complex numbers satisfying $\sum_{i=1}^2 |a_i|^2=1=\sum_{j=2}^2 |b_j|^2$.
Hence, the product $S_3\times S_3$ of the 3-sphere with itself indexes the irreducible classes of one-dimensional (irreducible) $\cAL$-modules.
Corollary~\ref{cor:classification} implies that $\Pi$ determined a bijection from $S_3\times S_3$ to the set $\Spec(1)$ of irreducible classes of representations $\pi$ of
$\cOL$ with P-dimension (with respect to the $2$-graph $\Lambda$) $\Pdim^\Lambda(\pi) = 1$.

\textbf{Two-dimensional modules.}
We construct a family of irreducible matrix modules $(A_i,B_j,\C\oplus\C)$ of dimension $2$.
Take $A_1=\diag(\lambda_1,\lambda_2)$ diagonal with distinct eigenvalues $\lambda_1,\lambda_2$ in the open unit ball $\Ball$ of $\C$.
Then necessarily $B_1,B_2$ are also diagonal since they commute with $A_1$.
If either $B_1$ or $B_2$ has distinct eigenvalues, then $A_2$ must also be diagonal since it commutes with the $B_i$.
Hence $\C\oplus\{0\}$ and $\{0\}\oplus \C$ are submodules. So to obtain an irreducible module, we must take
$B_i=\mu_i I$ for scalars $\mu_i$ satisfying $|\mu_1|^2+|\mu_2|^2=1$ (that is $(\mu_1,\mu_2) \in S_3$, the 3-sphere).
Fix any matrix $A_2$ such that $A_2^*A_2 = I - A_1^*A_1 = \diag(1-|\lambda_1|^2,1-|\lambda_2|^2)$ and such that $A_2$ is not triangular (upper or lower).
These are exactly the matrices $A_2=(c_{ij})_{ij}$ described as follows.
\begin{itemize}
\item $|c_{11}|^2+|c_{21}|^2=1-|\lambda_1|^2$ with $c_{21}\neq 0$ (this is possible since $|\lambda_1|< 1$). These choices correspond to the complement of the circle $c_{21}=0$ in $S_3$;
\item $c_{12}$ satisfying $|c_{12}|^2=\dfrac{1-|\lambda_2|^2}{1+ \frac{|c_{11}|^2}{|c_{21}|^2} }$, which is nonzero since $|\lambda_2|<1$ by assumption. This gives a circle $S_1$ of choices;
\item the remaining coefficient satisfies $c_{22}:= -\dfrac{\bar c_{11} \cdot c_{12}}{\bar c_{21}}$.
\end{itemize}
This family of $\cAL$-modules is parameterised by
$$\Ball\times \Ball\times S_3\times (S_3\setminus S_1)\times S_1$$
a smooth real manifold of dimension $11$.

By applying the functor $\Pi$ we obtain a large family of irreducible representations of $\cOL$ of Pythagorean dimension 2, and thus some unitary representations of $2V$ by restriction.
Classes of these associated irreducible representations of $\cOL$ are indexed by the orbits of $\PU(2)$ acting diagonally by conjugation on the set of quadruples of matrices $(A_1,A_2,B_1,B_2)$. This is a smooth manifold of dimension $8$ that embeds in the piece of spectrum $\Spec(2)$ of $\cOL$.

\textbf{Higher-dimensional modules.}
Take any two irreducible $\cA_2$-modules $\pi_1 = (A_1, A_2, \C^{d_1})$ and $\pi_2 = (B_1, B_2, \C^{d_2})$ (recall, $\cA_2$ is the non-self-adjoint algebra associated to the $1$-graph with one vertex and two loops). Define the operators
\[\ti A_i := A_i \otimes \id \textrm{ and } \ti B_j := \id \otimes B_j \]
for $i,j = 1,2$ which act on $\C^{d_1} \otimes \C^{d_2}$. It is clear that $\ti A_i$ commutes with $\ti B_j$. Thus, $\pi := (\ti A_i, \ti B_j, \C^{d_1} \otimes \C^{d_2})$ forms an $\cAL$-module. Moreover, since $\pi_1, \pi_2$ are irreducible, then it easily follows that $\pi$ is also irreducible.
From Proposition \ref{prop:irrep-manifold}, the set of irreducible representations of $\cO_2$ of P-dimension $d$ form a smooth manifold of dimension $3d^2$ (as explained earlier in this section that this set is nonempty). Subsequently, we obtain that the class of irreducible representations of $\cOL$ of P-dimension $d_1d_2$ contains a smooth manifold of dimension $3d_1^2 + 3d_2^2$. However, it is not immediately clear when the $\cAL$-modules of the form $\pi$ above are equivalent to each other.

\newcommand{\etalchar}[1]{$^{#1}$}

\end{document}